\documentclass[12pt,reqno,a4paper]{amsart}
\usepackage{extsizes}
\usepackage{epsfig}
\usepackage[english]{babel}
\usepackage{bigints}
\usepackage{esint}
\usepackage{nccmath}
\usepackage{relsize}
\usepackage{amssymb}
\usepackage{amsmath}
\usepackage{graphicx}
\usepackage{mathabx}
\usepackage{bbm}
\usepackage{bm}
\usepackage{enumerate}
\usepackage[shortlabels]{enumitem}   
\usepackage{times} 

\newcommand{\A}{\mathcal{A}}
\newcommand{\C}{\mathbb{C}}

\newcommand{\R}{\mathbb{R}}
\newcommand{\E}{\mathbb{E}}
\newcommand{\1}{\mathbb{1}}
\renewcommand{\P}{\mathbb{P}}
\newcommand{\Z}{\mathbb{Z}}
\newcommand{\N}{\mathbb{N}}
\newcommand{\W}{\mathbb{W}}

\newcommand{\B}{\mathcal{B}}
\newcommand{\ce}{\mathbb E}

\newcommand{\setdef}{\stackrel {\rm {def}}{=}}
\newcommand{\ds}{\displaystyle}

\newcommand{\tif}{\tilde f}
\newcommand{\tig}{\tilde g}
\newcommand{\bmu}[0]{\bm \mu}
\newtheorem*{thank}{\ \ \ Acknowledgment}
\newcounter{tictac}

\def\1{\,\rlap{\mbox{\small\rm 1}}\kern.15em 1}

\def\build#1_#2^#3{\mathrel{\mathop{\kern 0pt#1}\limits_{#2}^{#3}}}
\def\tend#1#2{\build\hbox to 12mm{\rightarrowfill}_{#1\rightarrow #2}^{ }}

\def\tendn{\tend{n}{\infty}}
\def\converge#1#2#3#4{\build\hbox to
#1mm{\rightarrowfill}_{#2\rightarrow #3}^{\hbox{\scriptsize #4}}}
\def\tendx#1#2#3{\build\hbox to 12mm{\rightarrowfill}_{#1\rightarrow #2}^{ #3}}
\newcommand{\printdate}{\today}

\theoremstyle{definition}

\newtheorem{thm}{Theorem}[section]
\newtheorem{prop}[thm]{Proposition}
\newtheorem{lemm}[thm]{Lemma}
\newtheorem{defn}[thm]{Definition}

\newtheorem{claim}[thm]{Claim}

\newtheorem{Prop}[thm]{Proposition}
\newtheorem{que}[thm]{Question}
\newtheorem{obs}[thm]{Observation}
\newtheorem{rem}[thm]{Remark}
\newtheorem{rems}[thm]{Remarks}

\newtheorem{Cor}[thm]{Corollary}

\newtheorem*{xrem}{Remark}

\newcommand{\Prod}{\mathop{\mathlarger{\mathlarger{\mathlarger{\prod}}}}}

\def \f{\overline{f}}

\def \e{\varepsilon}
\def\InvFourier#1{\stackrel{\vee}{#1}}
\def\Invstar#1{\widetilde{#1}}

\def\Pr{\mathbb P}
\newcommand{\beq}{\begin{equation}}
\newcommand{\eeq}{\end{equation}}

\begin{document}
\title[Banach problem]{On the spectral type of rank one flows and Banach problem with calculus of generalized Riesz products on the real line  }
\author{e.\ H. \ {el} Abdalaoui }
\address{ University of Rouen Normandy \\
  LMRS UMR 60 85 CNRS\\
Avenue de l'Universit\'e, BP.12 \\
76801 Saint Etienne du Rouvray - France . \\}
\email{elhoucein.elabdalaoui@univ-rouen.fr}




{\renewcommand\abstractname{Abstract}
\begin{abstract}
It is shown that a certain class of Riesz product type measures on $\R$ is realized a spectral type of rank one flows. As a consequence, we will establish that some class of rank one flows has a singular spectrum. Some of the results presented here are even new for the $\Z$-action. Our method is based, on one hand, on the extension of Bourgain-Klemes-Reinhold-Peyri\`ere method, and on the other hand, on the extension of the Central Limit Theorem approach to the real line which gives a new extension of Salem-Zygmund Central Limit Theorem. We extended also a formula for Radon-Nikodym derivative between two generalized Riesz products obtained by el Abdalaoui-Nadkarni and a formula of Mahler measure of the spectral type of rank one flow but in the weak form. We further present an affirmative answer to the flow version of the Banach problem, and we discuss some issues related to flat trigonometric polynomials on the real line in connection with the famous Banach-Rhoklin problem in the spectral theory of dynamical systems.
\vspace{4cm}\\

\hspace{-0.7cm}{\em AMS Subject Classifications} (2010): 37A15, 37A25, 37A30, 42A05, 42A55.\\

\printdate
\end{abstract}
\thispagestyle{empty}
\maketitle
\newpage
\section{Introduction}\label{intro}

The purpose of this paper is to study the spectral type of some class of rank one flows and its connection to Riesz products on real line. 
we will also establish similar results to those proved for the spectral type of rank one maps, and we will extended, as far as possible, the results obtained by M. Nadkarni and the author for the generalized Riesz products on the circle \cite{elabdal-Nad1}. So, this later paper can be seeing as a companion to this paper. Notice that, therein, the authors proved an entropy formula for the spectral type of rank one by computing its Mahler measure and by  applying some methods from $H^p$ theory. Here, using the entropy method, we will extended as far as possible those results. Let us point also that some of the results presented here are new even for the $\Z$-action. Precisely, Klemes-Parreau theorem on the singularity of the linear staircase and our result on the singularity of exponential straicase.   we will further discuss some issues related to flat trigonometric polynomials on real line. \\

We recall that there exists a several generalized Riesz products 
{\cite{Queffelec1}, \cite{Brown-Hewitt}, \cite{Hewitt-Zuckermann}, \cite{Host-mela-parreau}}. 
But all those generalized Riesz products is based on the notion of dissociation 
(Roughly speaking the alphabet is dissociated if there exist at most one way to product a word of given length). 
As in the classical case, the dissociation property is needed to prove the existence of  
those generalized Riesz products. In our case we deals with generalized Riesz products 
without dissociated property, nevertheless, we shall prove the existence of 
these kind of generalized Riesz products in the case of $\R$ action using the dynamical properties. In 1993, 
The same  generalization for $\Z$ action has been produced by Bourgain in \cite{Bourgain}. 
Precisely, J. Bourgain in \cite{Bourgain} introduced a new method of generalized 
Riesz products and proved that almost surely Ornstein's
transformations have singular spectrum. Subsequently, using the
same method, I. Klemes \cite{Klemes1} and I. Klemes \& K. Reinhold
\cite{Klemes2} show that mixing
staircase transformations of Adams \cite{Adams1} and Adams \&
Friedman \cite{Adams2} have singular spectrum. Here, we will extended, as far as possible, those results.\\

Rank one flows have simple spectrum and using a random Ornstein procedure \cite{Ornstein},
A. Prikhod'ko in \cite{prikhodko-orn} produce a family of mixing
rank one flows. It follows that the mixing
rank one flows may possibly contain a candidate for the flow version of the
Banach's well-known problem whether there exists a dynamical flow $(\Omega,{\mathcal
{A}},\mu,(T_t)_{t \in \R})$ with simple Lebesgue spectrum. In \cite{prikhodko},
A. Prikhod'ko introduced a class of rank one flows called {\it {exponential staircase rank one flows}} and studied its spectral type. The paper contained many interesting ideas and facts. Moreover, therein, the author stated also that in this class the answer to the flow version of Banach problem is affirmative by establishing that there is a $L^1$-locally flat trigonometric polynomials on real line. Unfortunately, as we shall see
, those polynomials are far from being $L^1$-locally flat (see the appendix). In the very recent  republished version of Prikhod'ko's paper in the same journal (Sb. Math. 211, 2020, \cite{prikhodko20}), the issue of flatness is not addressed. Precisely, the proof of Theorem 1 and Theorem 2 of that paper are missing.  Moreover, assuming the existence of $L^1$-locally flat polynomials, the author does not address the problem of  constructing the rank one flow acting on a probability space with simple Lebesgue spectrum for which the spectral type is given as the generalized Riesz product associated to those polynomials. However, therein, the author establish some consequences in ergodic theory under the assumption of the existence of rank one flow acting on a probability space with absolutely continuous spectrum or Lebesgue spectrum (see Lemma 4, Lemma 5, Lemma 14 and Lemma 15 in \cite{prikhodko20}). Here, we will further proved that the spectrum of a large class of exponential staircase rank one flow is singular. \\

Our main tools are on one hand an extension to $\R$ of the
CLT method (introduced in \cite{elabdaletds} for the torus) and on the other hand
the extension of Bourgain methods \cite{Bourgain} for the generalized Riesz products on $\R$. It turns out that one of the most crucial idea here is to generalize  first the Bourgain lemma in \cite{Bourgain} which may be independently with some interest. For that, we shall give a new generalization of Weiner Wintner to flows. We further present a new extension of the Salem-Zygmund CLT Theorem \cite{Zygmund} to the trigonometric sums
with \textbf{real frequencies}.\\

Originally Salem-Zygmund CLT Theorem concerns the asymptotic stochastic behavior
of the lacunary trigonometric sums on the torus.
Since Salem-Zygmund pioneering result, the central limit theorem for
trigonometric sums has been intensively studied by many authors,  Erd\"os \cite{Erdos},
J.P. Kahane \cite{Kahane}, J. Peyri\`ere \cite{Peyriere}, Berkers \cite{Berkes}, Murai \cite{Murai},
Takahashi \cite{Takahashi}, Fukuyama and  Takahashi
\cite{Fukuyama}, and many others. The same method is used to
study the asymptotic stochastic behaviour of Riesz-Raikov sums \cite{petit}.
Nevertheless all these results concern only the trigonometric sums on the torus.\\

It turns out that the fundamental ingredient  in our proof is based on the famous
Hermite-Lindemann Lemma  in the transcendental number theory \cite{Waldschmidt}.\\

Notice that the main argument used in the torus case \cite{elabdaletds}
is based on the density of trigonometric polynomials. This argument cannot be applied here since the density of trigonometric polynomials in $L^1(\R,\omega(t) dt)$ ($\omega$ is a positive function in $L^1(\R)$),
is not verified unless $\omega$ satisfies some extra-condition.
Nevertheless, using the density of the functions with compactly supported Fourier transforms, we are able to conclude.\\

We recall that Banach problem is stated in Ulam's book \cite[p.76]{Ulam} as follow 
\begin{que}[Banach Problem]
		Does there exist a square integrable function $f(x)$ and a measure preserving transformation $T(x)$,
		$-\infty<x<\infty$, such that the sequence of functions $\{f(T^n(x)); n=1,2,3,\cdots\}$ forms a complete
		orthogonal set in Hilbert space? \footnote{Professor M. Nadkarni pointed to me that the question contain an oversight. The  sequence of functions should be bilateral, that is, $n \in \Z$.}
\end{que}}

Therefore, the flow can act on a $\sigma$-finite measure space. Banach problem was solved positively here \cite{elabdal-Banach} by establishing that there is a sequence of $L^1$-flat trigonometric polynomials on the torus. This is accomplished by applying a result due to M. Nadkarni and the author which say that the existence of $L^1$-flat polynomials implies that there is rank one map acting on infinite measure space. As, we shall see,  this result can be extended to the case of the flow of rank one. We will thus obtain a positive answer to Banach problem by establishing the existence of rank one flow acting on infinite measure space.\\

We should point out here that the famous spectral problem in ergodic theory asks on the existence of a measure preserving transformation \textbf{on a probability space} with simple Lebesgue spectrum. This problem should be attributed to \linebreak Banach and Rokhlin. Indeed, Rokhlin asked whether there exist an ergodic measure preserving transformation on a finite measure space whose spectrum is \linebreak Lebesgue type with finite multiplicity \cite[p.219]{Rohlin}.\\

Later,  Kirillov in his 1966's paper \cite{Kiri} wrote ``there are grounds for thinking that such examples do not exist". However he has described a measure preserving  action (due to M. Novodvorskii) of
the group $(\bigoplus_{j=1}^\infty{\mathbb {Z}})\times\{-1,1\}$  on the compact dual of discrete rationals whose unitary group has Haar spectrum of multiplicity 2. Similar group actions with higher finite even multiplicities are also given.\\

Subsequently, finite measure preserving transformation having Lebesgue component of finite even multiplicity have been
constructed by J. Mathew and M. G. Nadkarni \cite{MN}, Kamae \cite{Kamae}, M. Queffelec \cite{Q}, and  O. Ageev \cite{Ag}.
Fifteen years later, M. Guenais \cite{Guenais} used a $L^1$-flat generalized Fekete polynomials on some torsion groups 
to construct a group action with simple Lebesgue component. 
A straightforward application of Gauss formula yields that the generalized Fekete polynomials constructed by Guenais
are ultraflat. Very recently, el Abdalaoui and Nadkarni strengthened Guenais's result \cite{Abd-Nad2} by proving that there exist an ergodic non-singular dynamical 
system with simple Lebesgue component. However, despite all these efforts, it is seems that the question of Rokhlin still open since the maps constructed does not have {\it a pure
	Lebesgue spectrum.}\\

It is noticed in \cite{elabdal-Banach} that this problem is a ``dark continent" for the ergodic theory.\\

We should mention also that the famous problem in ergodic theory whether there is  flow acting on the probability space with finite multiplicity and \textbf{pure Lebesgue spectrum} is due to Rokhlin \cite[p.219 and p.238]{Rohlin}. As far as the author know, this problem is still open for any finite multiplicity.\\

The paper is organized as follows. In section 2, we review some standard facts from the
spectral theory of dynamical flows. In section 3, we recall the basic construction of the
rank one flows obtained by the cutting and stacking method and some related definitions. In sections 4, 5, 6 and 7, we summarize and extend the relevant material on the
Bourgain criterion concerning the  singularity of the generalized Riesz products on $\R$. In section 8, we extend el Abdalaoui-Nadkarni theorem \cite{elabdal-Nad1} by presenting a formula of Radon-Nikodym for the Riesz products on $\R$. In section 9, we extend partially the el Abdalaoui-Nadkarni formula for the Mahler measure of the spectral type of rank one maps established in \cite{elabdal-Nad1}.  In section 10, we extend the results from \cite{elabdal-Nad1} and \cite{Abd-Nad3} to the real line with connection to flat polynomials  and Banach problem. In section 11, we extend Bourgain theorem \cite{Bourgain} by establishing the singularity almost sure of Ornstein rank one flows. In sections 12 and 13, we present the real line version of Klemes-Reinhold  theorem and  Klemes-Parreau theorem.   In sections 14 and 15, we develop the CLT method for trigonometric sums with real frequencies and
we prove the singularity of some class of exponential staircase rank one flows. In section 16, we present present an affirmative answer to the flow version of Banach problem. Finally, in the appendix, 
we introduce the notion of $L^1$-locally flat polynomials and we discuss some issues related to the main result in \cite{prikhodko}.

\section{Basic facts from spectral theory of dynamical flows }\label{sec:1}
 A dynamical flow  is  a quadruplet $(X,\mathcal{A},\mu,(T_t)_{t \in \R})$
where $(X,\mathcal{A},\mu)$ is a Lebesgue probability space
 and $(T_t)_{t\in \R}$ is a measurable action of the group $\R$ by measure preserving transformations. (It means that
 \begin{itemize}\item
 each $T_t$ is a bimeasurable invertible transformation of the probability space such that, for any
$A \in \mathcal{A}$, $\mu(T_t^{-1}A)=\mu(A)$,
\item for all $s,t\in\R$, $T_s\circ T_t = T_{s+t}$,
\item the map $(t,x)\mapsto T_t(x)
$ is measurable from $\R\times X$ into $X$.
\end{itemize}

Let us recall some classical definitions. A dynamical flow is \emph{ergodic} if
every measurable set which is invariant under all the maps $T_t$ either has measure zero or one. A number
$\lambda$ is an  \emph{eigenfrequency} if there exists nonzero function $f \in L^2(X)$ such that, for all $t\in\R$,
$f \circ T_t=e^{i \lambda t} f$. Such a function $f$ is called an  \emph{eigenfunction}.
An ergodic flow $(X,\mathcal{A},\mu,(T_t)_{t \in \R})$ is  \emph{weakly mixing} if every eigenfunction is constant (a.e.).
A flow $(X,\mathcal{A},\mu,(T_t)_{t \in \R})$ is  \emph{mixing} if for all $f,g \in L^2(X)$,
\[\bigintss f \circ T_t(x) \overline{g}(x) d\mu(x) \tend{|t|}{+\infty} \bigintss f(x) d\mu(x) \bigintss \overline{g}(x) d\mu(x).\]

Any dynamical flow $(T_t)_{t \in \R})$ induces an action of $\R$ by unitary operators acting on
$L^2(X)$ according to the formula  $U_{T_t}(f)=f \circ T_{-t}$. When there will be no ambiguity on the choice of the flow, we will denote $U_t=U_{T_t}$.

The \emph{spectral properties} of the flow are the property attached to the unitary representation associated to the flow. We recall below some classical facts; for details and references see \cite{CFS} or \cite{Handbook-1B-11}.

Two dynamical flows
$(X_1,\mathcal{A}_1,\mu_1,(T_t)_{t \in \R})$ and $ (X_2,\mathcal{A}_2,\mu_2,(S_t)_{t \in \R})$ are \emph{metrically isomorphic} if there exists a measurable map $\phi$ from $(X_1,\mathcal{A}_1,\mu_1)$ into $(X_2,\mathcal{A}_2,\mu_2)$, with the following properties:
\begin{itemize}
\item $\phi$ is one-to-one,
 \item For all $A \in \mathcal{A}_2$, $\mu_1(\phi^{-1}(A))=\mu_2(A).$
\item $S_t \circ \phi=\phi \circ T_t$, $\forall t \in \R$.
\end{itemize}
If two dynamical flows $(T_t)_{t \in \R}$ and $(S_t)_{t \in \R}$ are metrically isomorphic then the isomorphism $\phi$ induces
an isomorphism $V_{\phi}$ between the Hilbert spaces $L^2(X_2)$ and $L^2(X_1)$ which acts according to the formula
$V_{\phi}(f)=f \circ \phi$. In this case, since
$V_{\phi}U_{S_t}=U_{T_t}V_{\phi}$, the adjoint groups $(U_{T_t})$ and $(U_{S_t})$ are unitary equivalent. Thus
if two dynamical flows are metrically isomorphic then the corresponding adjoint groups of unitary operators are unitary
equivalent. It is well known that the converse statement is false \cite{CFS}.

By Bochner theorem, for any $f \in L^2(X)$, there exists a unique finite Borel measure $\sigma_f$ on $\R$ such that
\[\widehat{\sigma_f}(t)=\bigintss_{\R} e^{-it\xi}\ d\sigma_f(\xi)=\langle U_tf,f \rangle=
\bigintss_{X} f \circ T_t(x)\cdot \overline{f}(x) \ d\mu(x).\]
$\sigma_f$ is called the \emph{spectral measure} of $f$. If $f$ is eigenfunction with eigenfrequency $\lambda$ then
the spectral measure is the Dirac measure at $\lambda$.

The following fact derives directly from the definition of the spectral measure: let $(a_k)_{1\leq k\leq n}$ be complex numbers and $(t_k)_{1\leq k\leq n}$ be real numbers; consider $f\in L^2(X)$ and denote $F=\sum_{k=1}^n a_k \cdot f\circ T_{t_k}$. Then the spectral measure $\sigma_F$ is absolutely continuous with respect to the spectral measure $\sigma_f$ and
\begin{equation}\label{radon-spectral}
\frac{d\sigma_F}{d\sigma_f}(\xi)=\left|\sum_{k=1}^n a_k e^{it_k\xi}\right|^2.
\end{equation}

Here is another classical result concerning spectral measures : let $(g_n)$ be a sequence in $L^2(X)$, converging to $f\in L^2(X)$ ; then the sequence of real measures $(\sigma_{g_n}-\sigma_f)$ converges to zero in total variation norm.
\medskip

The \emph{maximal spectral type} of $(T_t)_{t \in \R}$ is the equivalence class of Borel
measures $\sigma$ on $\R$ (under the equivalence relation $\mu_1
\sim \mu_2$ if  and only if $\mu_1<<\mu_2$ and $\mu_2<<\mu_1$),
such that
 $\sigma_f<<\sigma$ for all $f\in L^2(X)$ and
if $\nu$ is another measure for which $\sigma_f<<\nu$
for all $f\in L^2(X)$ then $\sigma << \nu$.

The maximal spectral type is realized as the spectral measure of one function: there exists $h_1\in L^2(X)$
such that $\sigma_{h_1}$ is in the equivalence class defining the maximal spectral type of $(T_t)_{t \in \R}$.
By abuse of notation, we will call this measure the maximal
spectral type measure.

The reduced maximal type $\sigma_0$ is the
maximal spectral type of ${(U_{t})}_{t \in \R}$ on $L_0^2(X)\stackrel{\rm
{def}}{=}\left\{f \in L^2(X)~:~ \displaystyle \bigintss f d\mu=0 \right\}$. The
spectrum of $(T_t)_{t \in \R}$ is said to be discrete (resp. continuous, resp.
singular, resp. absolutely continuous , resp. Lebesgue) if
$\sigma_0$ is discrete (resp. continuous, resp. singular with respect to Lebesgue measure, resp.
absolutely continuous with respect to Lebesgue measure).

The cyclic space of $h \in L^2(X)$ is
\[Z(h) \stackrel{\rm {def}}{=} \overline {{\rm {span}} \{U_{t}h\,:\,t \in \R \} }.
\]

There exists an orthogonal decomposition of $L^2(X)$ into cyclic spaces
\beq\label{rozkladsp} L^2(X)=\bigoplus_{i=1}^\infty Z(h_i),\;\;
\sigma_{h_1}\gg\sigma_{h_2}\gg\ldots \eeq Each
decomposition~(\ref{rozkladsp}) is be called a {\em spectral
decomposition} of $L^2(X)$ (while the sequence of measures is called a {\em
spectral sequence}). A spectral decomposition is unique up to equivalent class of
the spectral sequence. The spectral decomposition is determined
by the maximal spectral type and the
{\em multiplicity function}
$M:\R\to\{1,2,\ldots\}\cup\{+\infty\}$, which is
defined $\sigma_{h_1}$-a.e. by $ M(s)=\sum_{i=1}^\infty
1_{Y_i}(s)$, where $Y_1=\R$ and
$Y_i=supp\,\frac{d\sigma_{x_i}}{d\sigma_{x_1}}$ for $i\geq2$.

The flow has {\em simple spectrum} if
$1$ is the only essential value of $M$. The multiplicity is
{\em homogeneous} if there is only one essential value of $M$.
The essential supremum of $M$ is called the {\em maximal
spectral multiplicity}.

Von Neumann showed that the flow ${(T_t)_{t \in \R}}$ has homogeneous Lebesgue spectrum if and only if the
associated group of unitary operators ${(U_{t})_{t \in \R}}$ satisfy the Weyl commutation relations for some
one-parameter group $(V_t)_{t \in \R}$ i.e.
\[
U_{t} V_s=e^{-ist}V_sU_{t},~~~~~~~~~s,t \in \R,
\]
where $e^{-ist}$ denotes the operator of multiplication by $e^{-ist}$.\\ It is easy to show that
the Weyl commutation relations implies that the  maximal spectral type is
invariant with respect to the translations.
The proof of von Neumann homogeneous Lebesgue spectrum theorem can be found in \cite{CFS}.

\section{Rank one flows by Cutting and Stacking method}\label{CSC}

Several approach of the notion of rank one flow have been proposed in the literature. The notion of \emph{approximation of a flow by periodic transformations} has been introduced by Katok and Stepin in \cite{Kat-Step} (see Chapter 15 of \cite{CFS}). This was the first attempt of a definition of a rank one flow.

In \cite{Deljunco-Park}, del Junco and Park adapted the classical Chacon construction \cite{lebonchacon} to produce
similar construction for a flow. The flow obtain by this method is called the Chacon flow.

This cutting and stacking construction has been extended by Zeitz (\cite{zeitz}) in order to give a general definition of a rank one flow. In the present paper we follow this cutting and stacking (CS) approach and we recall it now.
We assume that the reader is familiar with the CS construction of a rank one map acting on certain
measure space which may be finite or $\sigma$-finite. A nice account may be founded in~\cite{Friedman}.

Let us fix a sequence $(p_n)_{n\in\N}$ of integers $\geq2$ and a sequence of finite sequences of non-negative real numbers $\left({(s_{n,j})}_{j=1}^{p_{n-1}}\right)_{n>0}$.

Let ${\overline{B_0}}$ be a rectangle of height $1$ with horizontal base $B_0$. At stage one divide $B_0$ into $p_0$ equal
parts $(A_{1,j})_{j=1}^{p_0}$. Let $(\overline{A}_{1,j})_{j=1}^{p_0}$ denotes the flow towers over
$(A_{1,j})_{j=1}^{p_0}$. In order to construct the second flow tower,
put over each tower $\overline{A}_{1,j}$ a rectangle spacer of height $s_{1,j}$ (and base of same measure as $A_{1,j}$) and form a stack of height $h_{1}=p_0 +\sum_{j=1}^{p_0}s_{1,j}$ in the usual
fashion. Call this second tower $\overline{B_1}$, with $B_1=A_{1,1}$.

At the $k^{th}$ stage, divide the base $B_{k-1}$
of the tower ${\overline{B}_{k-1}}$ into $p_{k-1}$ subsets $(A_{k,j})_{j=1}^{p_{k-1}}$ of equal measure.
Let $(\overline{A}_{k,j})_{j=1}^{p_{k-1}}$ be the towers over
$(A_{k,j})_{j=1}^{p_{k-1}}$ respectively. Above each tower $\overline{A}_{k,j}$, put a rectangle spacer of height $s_{k,j}$ (and base of same measure as $A_{k,j}$). Then form a stack of height $h_{k} = p_{k-1}h_{k-1} + \sum_{j=1}^{p_{k-1}}s_{k,j}$ in the usual
fashion. The new base is $B_k=A_{k,1}$ and the new tower is $\overline{B_k}$.

All the rectangles are equipped with Lebesgue two-dimensional measure that will be denoted by $\nu$. Proceeding this way we construct what we call a rank one flow
${(T_t)_{t \in \R}}$ acting on a certain measure space $(X,{\mathcal B} ,\nu)$ which may
be finite or $\sigma-$finite depending on the number of spacers added at each stage. \\
This rank one flow will be denoted by

$$(T^t)_{t \in \R} \stackrel{\text{def}}{=}
\left(T^t_{(p_n, (s_{n+1,j})_{j=1}^{p_{n}})_{n\geq0}}\right)_{t \in \R}$$

The invariant measure $\nu$ will be finite if and only if
$$\displaystyle \sum_{k=0}^{+\infty}
\frac{\sum_{j=1}^{p_{k}} s_{k+1,j}}{p_kh_k}<+\infty.$$
\noindent{}In that case, the measure will be normalized in order to have a probability.

\begin{rems}$\quad \quad$\\ 
\begin{itemize}
     \item The only thing we use from  \cite{zeitz} is the definition of rank one flows. Actually, a careful reading of Zeitz paper \cite{zeitz} shows that the
     author assumes that for any rank one flow there exists always at least one time $t_0$ such that
     $T_{t_0}$ has rank one property. But, as we shall see in the next point, this is not the case in general. Furthermore, if this property was  satisfied
     then the weak closure theorem for flows would hold as a direct consequence of the King weak closure theorem \cite{King} which state that for any rank one map $T$, we have the centralizer of $T$ is the weak closure of the powers, that is,
     $C(T)=WCT(T),$  where $C(T)$ is the centralizer of $T$ and
     $WCT(T)$ is the weak closure of $\{T^n,n \in \Z\}$. Indeed, by King's theorem, we get
     $C(T_{t}) \subset C(T_{t_0}) \stackrel{\rm {J.King}}{=}WCT(T_{t_0})\subset WCT({T_t}) \subset C(T_{t})$,.
	\item Let us further notice that an alternative definition of a rank one flow has been proposed by Ryzhikov in \cite{RyzhikovWCT}. We don't know if these two definitions are equivalent. Here is Ryzhikov definition.
	\begin{defn}\label{rd}A dynamical flow $(X,\mathcal{A},\mu,(T_t)_{t \in \R})$ has rank one if there exists a sequence $(E_j)$ in $\A$, a sequence $(s_j)$ of real numbers and a sequence $(h_j)$ of positive integers such that, for each $j$,$$ \xi_j:=\left(E_j,T_{s_j}E_j,T_{2s_j}E_j,\cdots,T_{(h_j-1)s_j}E_j,X\setminus\bigsqcup_{i=0}^{h_j-1}T_{is_j}E_j
		\right)$$
		is a partition of $X$ and
		\begin{itemize}\item
			the sequence $(\xi_j)$ converges to the $\sigma$-algebra (that is, for every $A\in\A$ and every $j$, we 
			can find a $\xi_j$-measurable set $A_j$ in such a way that\\ 
			$\mu(A \bigtriangleup~A_j) \tend{j}{+\infty}~0)$;
			\item  $s_jh_j \tend{j}{+\infty}+\infty$. 
		\end{itemize}
	\end{defn}
	
	\item Based on the previous definition, Ryzhikov proved the weak closure theorem for flows, that is, $C(T_t)=WCT(T_t)$. It follows that for a mixing rank one flow, we have
	$C(T_t)=\{T_t, t \in \R\}.$ Hence for each $t$, $T_t$ is not a rank one map since its centralizer is uncountable.

	\item After Ryzhikov's paper, del Junco in \cite{deljunco98} extended the cutting and stacking methods to the case of local Abelian group using
	a F\o{}lner sequence. He introduced a new construction called CF construction.
	The definition using the CF construction can be founded in \cite{Danilenko-Silva}. But we must point out that
	this definition is more general than the CS definition. In fact, CF definition corresponds to the notion of funny rank one flows.	
\end{itemize}
\end{rems}

\section{On Bourgain's lemma for flows}
Let $(X,\B,\mu,T)$ be an ergodic probability measure preserving dynamical system. In \cite{Bourgain} Bourgain uses the following fact : for all $f\in L^2(X)$, for $\mu$-almost all $x$, the sequence of finite measures
$$
\sigma_{f,N,x}(d\theta)=\left|\frac1{\sqrt{N}}\sum_{j=0}^{N-1} f(T^jx)\,e^{2i\pi j\theta}\right|^2\ d\theta
$$
converges weakly to the spectral measure of $f$.

This fact is a simple consequence of a classical harmonic analysis lemma \cite{Coquet-Kamae-Mendes-France}, 
combined with the ergodic theorem which asserts that, for $\mu$-almost all $x$, for all $k\in\Z$,
$$
\lim_{N\to+\infty} \frac1N \sum_{n=1}^N f(T^{n+k}x) \cdot\overline f(T^nx)=\int_X f\circ T^k\cdot\overline f\ d\mu=\widehat{\sigma_f}(k).
$$

We will state a similar result for $\R$-actions. 

Let us consider a kernel $\{k_\lambda\}_{\lambda>0}$, that is a family of bounded positive integrable functions $k_\lambda$ on $\R$, with $\int k_\lambda(t)\,dt=1$ such that the family of probability measures $k_\lambda(t)\,dt$ converges weakly to the Dirac mass at zero, when $\lambda$ goes to zero.

If $\psi$ is any bounded continuous function on $\R$ we have : for all $t\in\R$,
$$
\lim_{\lambda\to0} \psi\ast k_\lambda(t)=\psi(t).
$$

Thus for all $\phi\in L^2(dt)$ and for all $t\in\R$ we have
$$
\lim_{\lambda\to0} (\phi\ast\Invstar{\phi})\ast k_\lambda(t)=\phi\ast\Invstar{\phi}(t),
$$
where we denote $\Invstar{\phi}(t)=\overline\phi(-t)$.
Using Fourier transformation and its inverse this can be written :
\begin{equation}\label{four}
\lim_{\lambda\to0} \InvFourier{\overbrace{\left(\left|\widehat\phi\right|^2\cdot\widehat{k_\lambda}\right)}}=
\lim_{\lambda\to0} \InvFourier{\overbrace{\left(\widehat{(\phi\ast\Invstar{\phi})}\cdot\widehat{k_\lambda}\right)}}=
\phi\ast\Invstar{\phi}.
\end{equation}

\medskip

Let $(T_t)_{t\in\R}$ be an ergodic flow on a Lebesgue space $(X,\B,\mu)$, and $f\in L^2(X)$. We claim that, for $\mu$-almost all $x$, for all $S>0$, Equation (\ref{four}) can be applied to the function $\phi(t)=f(T_tx){\mathbf 1}_{[0,S]}(t)$. 

The function
$$
[0,S]\times X\ni(t,x)\mapsto f(T_tx)
$$
is measurable and
$$
\int_0^S\left(\int_X \left|f(T_tx)\right|^2\,d\mu(x)\right)\,dt=S\|f\|_2^2<+\infty.
$$
By taking a sequence of $S$'s going to infinity, we conclude that, for $\mu$-almost all $x$, for all $S>0$,
$$
\int_0^S\left|f(T_tx)\right|^2\,dt <+\infty.
$$
This proves our claim.

Equation (\ref{four}) applied to our particular functions gives : for $\mu$-almost all $x$, for all $S>0$, 
\begin{multline*}
\lim_{\lambda\to0} \int_{\R}\left|\int_0^Sf(T_tx)e^{-i\theta t}\,dt\right|^2\left(\int_\R k_\lambda(s) e^{-i\theta s}\,ds\right)\ e^{i\theta u}\,d\theta\\=\int_{\max\{u,0\}}^{\min\{u+S,S\}} f(T_tx)\cdot\overline f(T_{t-u}x)\,dt
\end{multline*}

In the next section, we will prove that for $\mu$-almost all $x$, for all $u\in\R$, we have
$$
\lim_{S\to+\infty} \frac1S\int_{\max\{u,0\}}^{\min\{u+S,S\}} f(T_tx)\cdot\overline f(T_{t-u}x)\,dt=\int_X f\circ T_u\cdot\overline f\ d\mu.
$$

This allows us to state the following lemma.

\begin{lemm}\label{BLemma}
	For $f\in L^2(X)$ and $x\in X$, we consider the absolutely continuous measure
	$$
	\sigma_{f,\lambda,S,x}(d\theta) = \left|\frac1{\sqrt S}\int_0^S f(T^tx) e^{-i\theta t}\,dt\right|^2\widehat{k_\lambda}(\theta) \,d\theta.
	$$
	For all $f\in L^2(X)$, for $\mu$-almost all $x$, we have
	$$
	\lim_{S\to+\infty}\lim_{\lambda\to 0^+} \sigma_{f,\lambda,S,x} = \sigma_f,
	$$
	in the sense of weak convergence.
\end{lemm}

\section{Almost sure approximation of the correlation function}

We state and prove in this section a version of the ergodic theorem which gives an almost sure approximation of the function $s\mapsto\widehat{\sigma_f}(s)$. The statement of Corollary~\ref{wwcor} can be found in \cite{wiener-wintner}, but we have not been able to find a proof in the literature. 

We prove the result for functions which are Lipschitz regular along the trajectories, then use a density argument. This method can be applied to a wide class of ergodic theorems perturbed in time, including convergence of multiple ergodic averages with weight of Wiener-Wintner type.\\

\noindent We consider an \emph{ergodic} dynamical flow $(X,\mathcal{A},\mu,(T_t)_{t \in \R})$.

\begin{thm}\label{wwhm} For any $S>0$ and all $f,g\in  L^2(X)$, for almost all $x \in X$, we have
	$$\lim_{\tau \to +\infty}\frac1{\tau}\int_{0}^{\tau}f(T_{t+s}x)\cdot g(T_tx)\,dt=\int_{X}f\circ T_s\cdot g\,d\mu$$
	uniformly for $s$ in the interval $[-S,S]$.
\end{thm}
\noindent This yields the result we need as a corollary.
\begin{Cor} \label{wwcor}Let $f\in L^2(\mu)$. There exist a full measure subset $X_f$ of $X$ such that, for any $x\in X_f$ and any $s\in\R$, we have
	$$
	\lim_{\tau\to\infty} \frac1{\tau}\int_0^{\tau} f(T_{t+s}x)\cdot \overline{f}(T_tx)\, dt=\int_X f\circ T_s \cdot \overline{f}\,d\mu.
	$$
\end{Cor}
\noindent As mentioned above we shall need the following crucial lemma.
\begin{lemm}\label{lemme1} The set of bounded measurable functions $f$ on $X$ such that, for any $x \in X$, the function
	$s \longmapsto f(T_sx)$ is  Lipschitz continuous with Lipschitz constant not depending on $x$, is dense in $L^2(X)$.
\end{lemm}
\begin{proof}{}
	Let $f$ be a bounded measurable function on $X$ and $a$ be a positive number. We define the function $f_a$ on $X$ by
	$$
	f_a(x)=\frac1{a}\int_0^a f(T_tx)\,d t.
	$$
	Since
	$
	\displaystyle \lim_{t\to 0} f\circ T_t =f\text{ in }L^2(X)$, we deduce that
	$ \displaystyle
	\lim_{a\to 0} f_a =f\text{ in }L^2(X).
	$
	But, for any $x$, the function $s\mapsto f_a(T_sx)$ is Lipschitz continuous with Lipschitz constant  equal to $2\|f\|_\infty/a$. In fact, we have
	$$
	| f_a(T_sx)-f_a(T_{s'}x)|=\frac1{a}\left |\int_s^{a+s}f(T_tx)\,d t-\int_{s'}^{a+s'}f(T_tx)\,d t\right |,
	$$
	and  if  $0\leq s-s'\leq a$, then
	\begin{multline*}
	| f_a(T_sx)-f_a(T_{s'}x)|=\frac1{a}\left |\int_s^{s'}f(T_tx)\,d t-\int_{a+s}^{a+s'}f(T_tx)\,d t\right |\\\leq\frac1{a}\left (\int_s^{s'}| f(T_tx)|\,d t+\int_{a+s}^{a+s'}|f(T_tx)|\,d t\right )\leq\frac2a\|f\|_\infty(s-s').
	\end{multline*}
\end{proof}


\begin{lemm}\label{lemme2} Let $f,g$ be two bounded measurable functions on $X$. Assume that, for any $x \in X$, the function
	$s \longmapsto f(T_sx)$ is Lipschitz continuous. Then, for any $S>0$, for almost all
	$x \in X$, we have
	$$\lim_{\tau \to +\infty}\frac1{\tau}\int_{0}^{\tau}f(T_{t+s}x)\cdot g(T_tx)\,dt=\int_{X}f\circ T_s\cdot g\,d\mu$$
	uniformly for $s$ in the interval $[-S,S]$.
\end{lemm}
\begin{proof}{} By the Ergodic Theorem, for almost all $x \in X$ and all rational number $r$, we have
	$$
	\lim_{\tau\to\infty}\frac1{\tau}\int_0^{\tau} f(T_{t+r}x)\cdot g(T_tx)\, d t=\int_X f\circ T_r \cdot g\,d\mu.
	$$
	Let us choose $x$ and $S$ as above. Let $\varepsilon$ be a positive number and $q$ a positive rational number such that $\frac1{q}<\varepsilon$. Then, there exists $\tau_0>0$ such that , for all $\tau>\tau_0$, for any rational number $r=\frac{p}{q}$ belonging to the interval  $[-S,S]$,
	$$
	\left| \frac1{\tau}\int_0^{\tau} f(T_{t+r}x) \cdot g(T_tx)\, d t-\int_X f\circ T_r \cdot g\,d\mu\right|\leq\e.
	$$
	Let $s\in[-S,S]$. There exists a rational number $r=p/q$ such that $|s-r|\leq\varepsilon$. We have, for all $\tau>\tau_0$,
	\begin{multline*}
	\left| \frac1{\tau}\int_0^{\tau} f(T_{t+s}x) \cdot g(T_tx)\, d t-\int_X f\circ T_s \cdot g\,d\mu\right|\\
	\leq
	\left| \frac1{\tau}\int_0^{\tau} (f(T_{t+s}x)-f(T_{t+r}x) \cdot g(T_tx)\, d t\right| \\+
	\left| \frac1{\tau}\int_0^{\tau} f(T_{t+r}x) \cdot g(T_tx)\, d t-\int_X f\circ T_r \cdot g\,d\mu\right| 
	\\+ \left|\int_X f\circ T_s \cdot g\,d\mu-\int_X f\circ T_r \cdot g\,d\mu\right|.
	\end{multline*}
	Using Lipschitz condition, we see that first and third terms are bounded by $c\|g\|_\infty\e$. The second term is bounded by $\e$ by our choose of  $x$ and $\tau$ and the proof of the lemma is complete.
\end{proof}

\begin{proof}[Proof of Theorem \ref{wwhm}]
	Let $S$ be a positive number. By Lemmas \ref{lemme1} and \ref{lemme2}, it is sufficient to show that the set of pairs
	of $L^2$ functions $(f,g)$ for which the theorem holds is closed in $L^2(X)\times L^2(X)$. 
	
	Let $f,g\in L^2(X)$ and $\e$ be a positive 
	number, assume that there exist $\tif, \tig \in L^2(X)$ such that 
	$
	||f-\tif||_2 < \e \ \text{ and }\  ||g-\tig||_2 < \e,
	$
	and for almost all $x$,
	$$
	\lim_{{\tau}\to\infty}\frac1{\tau}\int_0^{\tau} \tif(T_{t+s}x)\cdot \tig(T_tx)\, d t = 
	\int_X \tif\circ T_s\cdot \tig\,d\mu,
	$$
	uniformly for $ s\in[-S,S]$. For such an $x$ we have
	\begin{multline*}
	\limsup_{\tau\to\infty}\sup_{-S\leq s\leq S}\left|\frac1{\tau}\int_0^{\tau} f(T_{t+s}x)\cdot g(T_tx)\, 
	d t-\int_X f\circ T_s \cdot g\,d\mu\right|\\
	\leq
	\limsup_{{\tau}\to\infty}\sup_{-S\leq s\leq S}\left|\frac1{\tau}\int_0^{\tau} f(T_{t+s}x) \cdot g(T_tx)\, 
	d t-\frac1{\tau}\int_0^{\tau} \tif(T_{t+s}x) \cdot\tig(T_tx)\, d t\right|\\
	+\sup_{-S\leq s\leq S}\left|\int_X f\circ T_s \cdot g\,d\mu-\int_X \tif\circ T_s\cdot \tig\,d\mu\right|
	\end{multline*}   
	We need to estimate two terms on the right hand side. For the first one, we have
	\begin{multline*}
	\limsup_{{\tau}\to\infty}\sup_{-S\leq s\leq S}\left|\frac1{\tau}\int_0^{\tau} f(T_{t+s}x) \cdot g(T_tx)\, 
	d t-\frac1{\tau}\int_0^{\tau} \tif(T_{t+s}x) \cdot\tig(T_tx)\, d t\right|\\
	\leq
	\limsup_{{\tau}\to\infty}\sup_{-S\leq s\leq S}\left|\frac1{\tau}\int_0^{\tau} 
	f(T_{t+s}x)\cdot g(T_tx)\, d t-\frac1{\tau}\int_0^{\tau} \tif(T_{t+s}x) \cdot g(T_tx)\, d t\right|\\
	+ \limsup_{{\tau}\to\infty}\sup_{-S\leq s\leq S}\left|\frac1{\tau}\int_0^{\tau} \tif(T_{t+s}x)\cdot g(T_tx)\, d t-
	\frac1{\tau}\int_0^{\tau} \tif(T_{t+s}x)\cdot \tig(T_tx)\, d t\right|\\
	\leq
	\limsup_{{\tau}\to\infty}\sup_{-S\leq s\leq S}\left(\frac1{\tau}\int_0^{\tau}
	\left|(f-\tif)(T_{t+s}x)\right|^2 \, d t\right)^{1/2}\left(\frac1{\tau}\int_0^{\tau}
	\left| g(T_tx)\right|^2 \, d t\right)^{1/2}\\
	+
	\limsup_{{\tau}\to\infty}\sup_{-S\leq s\leq S}\left(\frac1{\tau}
	\int_0^{\tau}\left|\tif(T_{t+s}x)\right|^2 \, 
	d t\right)^{1/2}\left(\frac1{\tau}\int_0^{\tau}\left|(g-\tig)(T_tx)\right|^2 \, d t\right)^{1/2}\\
	\leq
	\limsup_{{\tau}\to\infty}\left(\frac1{\tau}\int_{-S}^{\tau+S}
	\left|(f-\tif)(T_{t}x)\right|^2 \, d t\right)^{1/2}\left(\frac1{\tau}\int_0^{\tau}
	\left| g(T_tx)\right|^2 \, d t\right)^{1/2}\\
	+
	\limsup_{{\tau}\to\infty}\left(\frac1{\tau}
	\int_{-S}^{\tau+S}\left|\tif(T_{t}x)\right|^2 \, 
	d t\right)^{1/2}\left(\frac1{\tau}\int_0^{\tau}\left|(g-\tig)(T_tx)\right|^2 \, d t\right)^{1/2}\\
	\end{multline*}
	Observe that we may assume in addition that the ergodic theorem holds at the point $x$ for
	the four functions $|\tif|^2$, $|g|^2$, $|f-\tif|^2$ and $|g-\tig|^2$. Hence, we have
	\begin{multline*}
	\limsup_{{\tau}\to\infty}\sup_{-S\leq s\leq S}\left|\frac1{\tau}\int_0^{\tau} 
	f(T_{t+s}x) \cdot g(T_tx)\, d t-\frac1{\tau}\int_0^{\tau} \tif(T_{t+s}x) \cdot\tig(T_tx)\, d t\right|\\
	\leq
	\|f-\tif\|_2\|g\|_2+\|\tif\|_2\|g-\tig\|_2
	\leq \e\|g\|_2  + (\e+\|f\|_2)\e.
	\end{multline*}  
	
	Consider now the second term.
	\begin{multline*}
	\sup_{-S\leq s\leq S}\left|\int_X f\circ T_s \cdot g\,d\mu-\int_X \tif\circ T_s\cdot \tig\,d\mu\right|\\
	\leq 
	\sup_{-S\leq s\leq S}\left|\int_X f\circ T_s \cdot g\,d\mu-\int_X \tif\circ T_s\cdot g\,d\mu\right| \\+\sup_{-S\leq s\leq S}\left|\int_X \tif\circ T_s \cdot g\,d\mu-\int_X \tif\circ T_s\cdot \tig\,d\mu\right|\\
	\leq
	\sup_{-S\leq s\leq S}\left(\int_X\left|(f-\tif)\circ T_s\right|^2\,d\mu\right)^{1/2}\left(\int_X |g|^2\,d\mu\right)^{1/2}\\
	+
	\sup_{-S\leq s\leq S}\left(\int_X |\tif\circ T_s|^2\,d\mu\right)^{1/2}\left(\int_X\left|(g-\tig)\right|^2\,d\mu\right)^{1/2}\\
	= 
	\|f-\tif\|_2\|g\|_2+\|\tif\|_2\|g-\tig\|_2
	\leq \e\|g\|_2  + (\e+\|f\|_2)\e.
	\end{multline*}
	Therefore,
	for a set of $x$'s off full measure, 
	$$
	\limsup_{{\tau}\to\infty}\sup_{-S\leq s\leq S}\left|\frac1{\tau}\int_0^{\tau} 
	f(T_{t+s}x) \f(T_tx)\, d t-\int_X f\circ T_s \cdot\f\,d\mu\right|
	$$
	is arbitrarily small. The proof of the theorem is complete.
\end{proof}

\begin{xrem}
	Using the same ideas of Bourgain in \cite{Bourgain} combined with the lemma \ref{BLemma} one may compute the maximal spectral type of rank one flows. 
	But,as we shall see in the next section, it is easier to do the calculations directly following the ideas of 
	Choksi-Nadkarni \cite{Nadkarni1} and Klemes-Reinhold \cite{Klemes2}.
\end{xrem}

\section{On the maximal spectral type of any rank one flow by CS Construction}
\label{mstcs}
In the spirit of the methods used in the case of $\Z$ actions by 
Host-Mela-Parreau~\cite{Host-mela-parreau}, Choksi-Nadkarni \cite{Nadkarni1} and
Klemes-Reinhold \cite{Klemes2}, we shall compute the maximal spectral type of the rank one flow given by CS construction. Since the dual
group of $\R$ is $\R$ which is not compact, we need to produce a probability measure on $\R$ for which we can
define the notion of Riesz product. It turns out that a probability measure on $\R$ naturally associated to rank one construction is given by Fej\'{e}r kernel density. Precisely, we have

	

\begin{thm}[Maximal spectral type of rank one flows]\label{max-spectr-typ}
	For any $s\in(0,1]$, the spectral measure $\sigma_{0,s}$ is the weak
	limit of the sequence of probability measures
	$$\prod_{k=0}^{n}|P_k(\theta)|^2 K_s(\theta)\,d\theta,$$
	\noindent{} where
	\begin{eqnarray*}
		&&P_k(\theta)=\frac 1{\sqrt{p_k}}\left(
		\sum_{j=0}^{p_k-1}e^{{i\theta(jh_k+\bar s_{k}(j))}}\right),~~\bar s_{k}(j)=\sum_{i=1}^js_{k+1,i},
		~\bar s_{k}(0)=0 \nonumber.  \\
		\nonumber
	\end{eqnarray*}
	\noindent{} and
	$$K_s(\theta)=\frac{s}{2\pi}\cdot{\left(\frac{\sin(\frac{s\theta}2)}{\frac{s\theta}2}\right)^2}.$$

	In addition the continuous part of spectral type of the rank one flow 	$(X,{\mathcal{A}},\nu,{(T_t)}_{t \in \R})$  is equivalent
	to the continuous part of $\ds \sum_{k \geq 1} 2^{-k}\sigma_{0,\frac{1}{k}}.$ 
\end{thm}
As customary, for a fixed $s \in (0,1]$, the spectral measure $\sigma_{0,s}$  will be denoted by 
$$ \sigma_{0,s}= \prod_{k=0}^{+\infty}|P_k(\theta)|^2.$$

The theorem above gives a new generalization of Choksi-Nadkarni Theorem \cite{Nadkarni1}, \cite{Nadkarni4}.
We point out that in \cite{elabdal}, the author generalized the Choksi-Nadkarni Theorem
to the case of funny rank one group actions for which the group is compact and Abelian.\\

\subsection{Spectral interpretation of CS construction}\label{spect-interp}
We start from the CS construction described in the preceding section. Let ${\overline{B}_{n,s}}$
be the rectangle of height $s \in ]0,h_n[$ and base $A_{n,1}=B_n$ in the $n^{th}$ flow tower.
By
construction, we have
\[
\overline{B}_{n,s}=\bigcup_{j=0}^{p_{n}-1} T_{jh_{n}+\bar{s}_{n}(j)}\overline{B}_{n+1,s},
\]
where $\bar{s}_{n}(j):=s_{n+1,1}+s_{n+1,2}+\ldots +s_{n+1,j}$ and $\bar{s}_n(0)=0.$  We have
\[
\nu \left(\overline{B}_{n,s}\right)=p_{n}\,\nu \left(\overline{B}_{n+1,s}\right),
\]

\noindent Put
\[
f_{n,s}=\frac 1{\sqrt{\nu \left(\overline{B}_{n,s}\right)}} \1_{\overline{B}_{n,s}},
\]
\noindent where $\1_{\overline{B}_{n,s}}$ is the indicator function of $\overline{B}_{n,s}$. So
\[f_{n,s}=\frac1{\sqrt{p_n}}\sum_{j=0}^{p_n-1}f_{n+1,s}\circ T^{-(jh_n+\bar{s}_n(j))}
\]
that we can write
\begin{equation}\label{poly}
f_{n,s}=\ P_n(U_{\bullet})f_{{n+1},s} {\rm {~~with~~}}  P_n(t)=\frac1{\sqrt{p_n}}\sum_{j=0}^{p_n-1}e^{i t(jh_n+\bar{s}_{n}(j))}.
\end{equation}

\noindent{}It follows from (\ref{radon-spectral}) that

\begin{eqnarray}\label{recusive}
d\sigma _{k,s}=\left| P_k\right| ^2d\sigma _{k+1,s}=\ldots
=\prod_{j=0}^{m-1}\left| P_{k+j,s}\right| ^2d\sigma _{k+m,s},
\end{eqnarray}

\noindent where $\sigma_{k,s}$ denotes the spectral measure of $f_{k,s}$, $k \geq 0, s \in ]0,h_k[$.\\

\begin{lemm}\label{H-croi}Denoting, for each integer $n\geq0$,
	$$H_n:={\text{linear span}}{\left\{U_{t}(f_{n,s}): s \in ]0,h_n[, t \in [0,h_n-s]\right\}},$$
	we have $H_n\subset H_{n+1}$.
\end{lemm}
\begin{proof}{}On one hand we have (\ref{poly});
	on the other hand, for $s \in ]0,h_n[$ and $t \in [0,h_n-s]$,
	since
	$$
	\1_{T_t\overline{B}_{n,s}}= \1_{\overline{B}_{n,s+t}}- \1_{\overline{B}_{n,t}},
	$$
	we have
	$$
	U_{t}(f_{n,s})=\sqrt{1+\frac{t}{s}}\;
	f_{n,s+t}-\sqrt\frac{t}{s}\;f_{n,t}.
	$$
	These two facts show that $H_n \subset H_{n+1}$. \end{proof} \begin{lemm} 
	$\overline{{\ds \bigcup_{n=0}^{+\infty}H_n}}=L^2(X).$
\end{lemm}
\begin{proof}
	Notice that
	$\ds 
	\{\{T_t(\overline{B}_{n,s})\}_{ s \in ]0,h_n[, t\in [0,h_n-s]}\}_{n=0}^{\infty} $
	generates a dense  subalgebra of the Borel $\sigma$--algebra,
	(here we are using the metric (modulo sets of measure zero) given by
	$d(A,B)=$ Lebesgue measure of $A\triangle B$). It follows that
	the linear subspace generated by  
	$\{U_{t}(f_{n,s}): s \in ]0,h_n[, t \in [0,h_n-s], 0\leq n<\infty\}$
	is dense in  $L^2(X)$. But Lemma \ref{H-croi} shows that this linear subspace is $\ds\cup_nH_n$.
\end{proof}

\begin{lemm}\label{spectraltype} The maximal spectral type $\sigma$ of the rank one flow
	$(X,{\mathcal{A}},\nu,{(T_t)}_{t \in \R})$
	is absolutely continuous with respect to
	$\ds \sum_{n\geq0,k\geq1}2^{-{(n+k)}} \sigma_{n,\frac1{k}}$.
\end{lemm}
This lemma tells us that the class of $\sum_{n\geq0,k\geq1}2^{-{(n+k)}} \sigma_{n,\frac1{k}}$ is the maximal spectral type  of the flow.
\begin{proof}{}Given $f\in L^2(X)$, since the span of the family
	$\{U_{t}(f_{n,s}): s \in ]0,h_n[, t \in [0,h_n-s], 0\leq n<\infty\}$ in dense in $L^2(X)$,
	$f$ can be approximated by functions
	$g_n\in L^2(X)$ which are constant on the levels of the flow tower $\overline{B}_n$. We can find
	$$g_n=\ds \sum_{j=1}^{k_n}a_j^{(n)}U_{{t_{j}^{(n)}}}f_{n,s_j^{(n)}}$$ 
	where $k_n \in \N,~{a_j^{(n)}} \in \C,\ {t_{j}^{(n)}}\in[0,h_n-s_j^{(n)}[,$ $j=1, \cdots,k_n$ and $||f-g_n||_2\tend{n}{+\infty}0.$ Hence
	$$d\sigma_f=d\sigma_{g_n}+d\nu_n$$
	\noindent{}where $\|\nu_n\|\rightarrow 0$ as $n\rightarrow\infty$. \\
	
	We have
	$$d\sigma_{g_n}=\sum_{i,j=1}^{k_n}{a_i^{(n)}}\overline{{a_j^{(n)}}}d{\sigma_{t_i^{(n)},{s_i^{(n)}},t_j^{(n)},{s_j^{(n)}}}} ,$$
	\noindent{}where $\sigma_{t_i^{(n)},s_i^{(n)},t_j^{(n)},s_j^{(n)}}$ is the complex measure on $\R$ whose Fourier transform is
	$$t\mapsto
	\left\langle U_{t}\left(U_{{t_j^{(n)}}}f_{n,s_j^{(n)}}\right),U_{{t_i^{(n)}}}f_{n,s_i^{(n)}})\right\rangle.$$
	It is easy to see that ${\sigma}_{t_i^{(n)},s_i^{(n)},t_j^{(n)},s_j^{(n)}}$ is absolutely continuous with respect to
	${\sigma}_{n,s_i^{(n)}}$ and ${\sigma}_{n,s_j^{(n)}}$. 
	
	We claim
	that, for all $n \in \N$ and $s>0$, the spectral measure
	${\sigma}_{n,s}$ is absolutely continuous with respect to
	$\ds \sum_{n\geq0,k\geq1}2^{-{(n+k)}} \sigma_{n,\frac1{k}}$. 
	
	First consider the case where $s$ is rational. In this case, put $s=\frac{p}{q}$ where $p \in \Z$ and 
	$q \in \Z^*$. Then
	$$f_{n,s}=\frac1{\sqrt p} \sum_{j=1}^p T_{\frac{j-1}{q}}{f_{n,\frac{1}{q}}},$$
	which implies that $\sigma_{n,s}\ll\sigma_{n,\frac1q}$.
	
	The case where $s$ is irrational can be handled using the existence of rational numbers 
	$(\frac{p_m}{q_m})_{m \in \N}$ which converge
	to $s$. We have 
	$$\|f_{n,s}-f_{n,\frac{p_m}{q_m}}\|_2\tend{m}{\infty}0.$$
	This yields that $\sigma_{n,\frac{p_m}{q_m}}$ converges in the sens of the norm variation to 
	$\sigma_{n,s}$ and the proof of the claim is complete. 
	
	Therefore, if $A$ is a Borel set such that, for any 
	$(n,k) \in \N \times \N^*$, we have $\ds \sigma_{n,\frac1{k}}(A)=0$, then, for all $(n,s) \in \N \times ]0,h_n[$,
	$\ds \sigma_{n,s}(A)=0$.  We deduce that $\sigma_{g_n}=0$. Thus
	$$0 \leq \sigma_{f}(A)=\nu_n(A)\leq ||\nu_n|| \tend{n}{+\infty}0.$$
	This achieves the proof of the lemma.
\end{proof}
\begin{xrem}
	Lemma \ref{spectraltype} can be derived from the more general result which follows.\\
	Let $(U_{t})_{t \in \R}$ be a family of bounded operators on the Hilbert space $H$ such that 
	$U_{t+s}=U_t \circ U_s$, for any $s,t \in \R$ and let $(x_i)_{i \geq 1}$ be a bounded sequence in $H$ such that
	$$\overline{\rm{span}}\left\{U_{t}x_i~~/~~ t \in \R,~~i \geq 1\right\}=H.$$
	Then the maximal spectral type $(U_{t})_{t \in \R}$ can be given by 
	$$ \sigma_{max}=\sum_{i=1}^{+\infty}\frac{\sigma_{x_i}}{2^i}.$$
\end{xrem}

As mentioned in the introduction, we shall show that the maximal spectral type of a rank one flow given by CS
construction is given by some kind of generalized Riesz product. In the classical theory, the Riesz product on $\R$ is
defined using some kernel function as is done in \cite{Peyriere} (a function $K$ is called a kernel function
if $K$ is a positive integrable function on $\R$ with positive Fourier
transform $\widehat{K}$ such that the support of $\widehat{K}$ is contained in some bounded interval).

In our case, we will see that the kernel is given by the dynamics as the weak limit of the sequence of measures
$(\sigma_{n,s})_{n \in  \N}$. We summarize this fact in the following lemma. (We consider only values of $s$ in $]0,1[$, so that the measures $\sigma_{n,s}$ are always well defined.)
\begin{lemm}
	The sequence of spectral measures $(\sigma_{n,s})_{ n \in \N}$ converges weakly to the measure $K_s(t)\,dt$,
	where $\ds K_s(t)=\frac{s}{2\pi} \cdot{\left(\frac{\sin(\frac{st}2)}{\frac{st}2}\right)}^2$.
\end{lemm}
\begin{proof}{}Let $t \in [0,+\infty[$ and choose $n_0 \in \N^*$ such that for all $n \geq n_0$, $h_n-s> t$.

	By definition of
	$\sigma_n$, we have 
	$$\widehat{\sigma_{n,s}}(t)=\frac{\nu(T_tB_{n,s} \bigcap B_{n,s})}{\nu(B_{n,s})}.$$

	When comparing $t$ to $s$ we must distinguish two cases
	\begin{itemize}
		\item[1.] $t>s$. In this case we have $\ds T_tB_{n,s} \bigcap B_{n,s}=\emptyset$,
		\item[2.] $t \leq s$. We claim that we have
		\begin{eqnarray}\label{Tri}
		\frac{\nu(T_tB_{n,s} \bigcap B_{n,s})}{\nu(B_{n,s})}=1-\frac{t}{s}.
		\end{eqnarray}
		Indeed the flow move by the unit speed uniformly and if we put $B_{n,s}=[0,\alpha_n]\times [0,s]$, we have
		$T_tB_{n,s}=[0,\alpha_n] \times [t,s+t]$. It follows easily that we have
		$T_tB_{n,s} \bigcap B_{n,s}=[0,\alpha_n] \times [t,s]$ which yields to
		\[
		\frac{\nu(T_tB_{n,s} \bigcap B_{n,s})}{\nu(B_{n,s})}=\frac{t-s}{s}=1-\frac{t}{s}.
		\]
		and the proof of (\ref{Tri}) is complete.
	\end{itemize}
	
	Notice that $\ds \widehat{K_s}(t)=\left(1-\frac{|t|}{s} \right)\1_{[-s,s]}(t)$. Therefore,
	\[
	\widehat{\sigma_n}(t)\tend{n}{\infty} \widehat{K_s}(t),
	\]
	which achieves the proof of the lemma.
\end{proof}

\begin{rem}
	Since $K_s(t)$ is positive $a.s.$ with respect to Lebesgue measure and
	\[
	\int_{-\infty}^{+\infty}K_s(t) dt=\widehat{K_s}(0)=1,
	\]
	the probability measure $K_s(t)\,dt$ is equivalent to
	the Lebesgue measure on $\R$.
\end{rem}

\subsection{A generalized Riesz product on $\R$ and dynamics}\label{Riesz-dynamics}

The trigonometric polynomials $P_n$ are related to the cutting and stacking construction of a rank one flow. They are defined by 
$$
P_n(\theta)=\frac1{\sqrt{p_n}}\sum_{j=0}^{p_n-1}e^{i \theta(jh_n+\bar{s}_{n}(j))}.
$$
We also recall that $K_s$ denotes the Fej\'{e}r Kernel on $\R$, characterized by its Fourier transform :
$$
\widehat{K_s}(\theta) =\left(1-\frac{|\theta|}{s} \right)\1_{[-s,s]}(\theta).
$$

\begin{lemm}\label{riesz-prop} Let $K$ be an integrable function on $\R$, whose Fourier transform $\widehat K$ is null outside the interval $[-1,1]$.
	Let $0\leq n_1<n_2<\ldots<n_k$ be integer numbers. We get
	$$
	\int_\R \prod_{j=1}^k \left|P_{n_j}(\theta)\right|^2\cdot K(\theta)\ d\theta =\int_\R K(\theta)\ d\theta.
	$$
\end{lemm}
This applies in particular to $K=K_s$ when $0<s<1$.
\begin{proof}
	We start with
	$$\left|P_{n_j}(\theta)\right|^2=1+\frac{1}{p_{n_j}}\sum_{\substack {0\leq a,b<p_{n_j}\\ a \neq b}}e^{i \theta ((b-a)h_{n_j}+\bar s_{n_j}(b)-\bar s_{n_j}(a))}.$$
	Let us define
	$$W_j=\{(b-a)h_{n_j}+\bar s_{n_j}(b)-\bar s_{n_j}(a) \mid  b\neq a \in \{0,\cdots,p_{n_j}-1\}\}.$$
	Expanding the product of the $ \left|P_{n_j}(\theta)\right|^2$, we can write
	$$
	\int_\R \prod_{j=1}^k \left|P_{n_j}(\theta)\right|^2\cdot K(\theta)\ d\theta -1
	$$
	as a sum of terms of the type
	$$
	\frac1{p_{n_1}\cdots p_{n_k}}\widehat{K}\left(\sum_{j=1}^k \epsilon_jw_j\right),
	$$
	where each $\epsilon_j$ is 0 or 1 (not all $=0$) and each $w_j$ belongs to $W_j$.
	It is sufficient to prove that each of these terms is null. Since the Fourier transform of $K$ is supported on $[-s,s]$, it is sufficient to prove that each of the numbers $\sum_{j=1}^k \epsilon_jw_j$ has absolute value $\geq1$.
	
	We consider one of these expressions $\sum_{j=1}^k \epsilon_jw_j$ and denote by $j_0$ the greater index $j$ such that $\epsilon_j\neq0$. We have
	$$\left|\sum_{j }\epsilon_j w_j\right|\geq \left|w_{j_0}\right|-\sum_{j<j_0} \left|w_{j}\right|.$$
	Moreover, for all $j$,
	$$h_{n_j}\leq \left|w_{j}\right| \leq (p_{n_j}-1)h_{n_j}+\bar s_{n_j}(p_{n_j}).$$
	Thus it is sufficient to prove that
	$$
	h_{n_{j_0}}-\sum_{j=0}^{j_0-1}(p_{n_j}-1)h_{n_j}+\bar s_{n_j}(p_{n_j})\geq 1.
	$$
	And this is true since
	$$
	h_{n}=p_{n-1}h_{n-1}+\bar s_{n-1}(p_{n-1})
	$$
	implies by induction
	$$
	h_{n}-\sum_{k=0}^{n-1}(p_{k}-1)h_{k}+\bar s_k(p_k)=1.
	$$
\end{proof}

\begin{lemm}\label{Key1} Let $K$ be a positive integrable function on $\R$, whose Fourier transform $\widehat K$ has compact support. The sequence of measures $|P_n(\theta)|^2 K(\theta)\,d\theta$ converges weakly
	to  $K(\theta)\,d\theta$.
\end{lemm}
\begin{proof}{}Let us compute the Fourier transform of $|P_n|^2 K$. Fix $t\in\R$.
	\begin{eqnarray*}
		\ds &&\int_{\R}e^{-it \theta} |P_n(\theta)|^2 K(\theta)\,d\theta = \\ \ds
		&&  \int_{\R}e^{-it \theta}\left (1+\frac1{p_n}\sum_{p \neq q}e^{i \theta ((p-q)h_n+(\bar{s}_{n}(p)-\bar{s}_{n}(q))}\right) K(\theta)\,d\theta= \\
		&&\ds \widehat{K}(t)+\frac1{p_n}\sum_{p \neq q}\widehat{K}\left(t-(p-q)h_n-(\bar{s}_{n}(p)-\bar{s}_{n}(q))\right).
	\end{eqnarray*}
	
	Since, for $p\neq q$,  $\left|t-(p-q)h_n-\bar{s}_{n}(p)+\bar{s}_{n}(q)\right| \geq h_n-|t|$, we can choose $n_0>0$ such that for all
	$n \geq n_0$, we have $t-(p-q)h_n-\bar{s}_{n}(p)+\bar{s}_{n}(q)$ outside the support of $\widehat K$. For all  such $n$,
	\[
	\int_{\R}e^{-it \theta} |P_n(\theta)|^2 K(\theta)\,d\theta=\widehat{K}(t).
	\]
	We proved that the Fourier transform of $|P_n|^2 K$ converges everywhere to the Fourier transform of $K$. The proof of Lemma~\ref{Key1} is complete.
\end{proof}

\begin{prop}\label{vague}
	The sequence of measures $|P_n(\theta)|^2 d\theta$ converges weakly
	to \linebreak Lebesgue measure.
\end{prop}

The weak convergence here is the convergence for the vague topology, where the space of test functions is the set of continuous functions with compact support. Of course, when we consider the convergence of a sequence of finite measures without ``loss of mass", this weak convergence is also the convergence in the narrow topology, where the space of  test functions is the set of all bounded continuous functions. Thus in Lemma \ref{Key1} we can speak of narrow convergence, while in Proposition \ref{vague} we have to speak of vague convergence.

\begin{proof}[Proof of Proposition \ref{vague}]
	Lemma \ref{Key1} applies in particular to the Fej\'{e}r Kernel $K_s$, which is strictly positive on the interval $(-\frac{2\pi}{s},\frac{2\pi}{s})$.
	
	Let $f$ be continuous function on $\R$ with compact support $S$. For $s$ small enough, we have $K_s>0$ on $S$. Then
	\begin{align*}
	\int_{\R} f(\theta)|P_n(\theta)|^2 \,d\theta &= \int_S \frac{f(\theta)}{K_s(\theta)}|P_n(\theta)|^2 K_s(\theta)\,d\theta\\&\longrightarrow\int_S \frac{f(\theta)}{K_s(\theta)} K_s(\theta)\,d\theta=\int_{\R} f(\theta) \,d\theta.
	\end{align*}
	
\end{proof}

\begin{prop}\label{Keyd}
	For any $s\in(0,1]$, the spectral measure $\sigma_{0,s}$ is the weak
	limit of the sequence of probability measures 
	$$\prod_{k=0}^{n}|P_k(\theta)|^2 K_s(\theta)\,d\theta.$$
\end{prop}
\begin{proof}
	
	As we did in the proof of Lemma \ref{riesz-prop}, we can write
	$$
	\prod_{k=0}^{n}|P_k(\theta)|^2=\frac1{p_0p_1\cdots p_n}\sum_{m\in M_n} e^{im\theta},
	$$
	where $M_n$ is the family of all sums of the type
	$$
	R\left((a_k,b_k)_{0\leq k\leq n}\right):=\sum_{k=0}^n (b_k-a_k)h_k+\bar s_k(b_k)-\bar s_k(a_k),\quad 0\leq a_k,b_k\leq p_k-1.
	$$ 
	
	We know (see proof of Lemma \ref{riesz-prop}) that for all choice of $(a_k,b_k)_{0\leq k\leq n}$ as above, \begin{equation}\label{inegun}\left|R\left((a_k,b_k)\right)\right|\leq h_{n+1}-1.\end{equation}
	
	The same argument shows that the sign of $R\left((a_k,b_k)\right)$ is the sign of $b_j-a_j$ where~$j$ is the greatest index $k$ such that $a_k\neq b_k$.\\
	
	We will use the following fact:
	\\
	Given $(a_k,b_k)_{0\leq k\leq n}$ as above,
	$0\leq \ell\leq n$, and $t\in[0,h_\ell]$
	\begin{equation}\label{fact}
	\left[\exists j\in\{\ell,\ell+1,\ldots,n\}, a_j-b_j<p_j-1\right]\ \Longrightarrow\ \left|R\left((a_k,b_k)\right)-t\right|\leq h_{n+1}-1.
	\end{equation}
	Let us prove this fact.\\
	If $R\left((a_k,b_k)\right)\geq0$, it is clear that $\left|R\left((a_k,b_k)\right)-t\right|\leq h_{n+1}-1$, since we have (\ref{inegun}) and $0\leq t\leq h_{n+1}-1$.\\
	Suppose now that  $R\left((a_k,b_k)\right)<0$. Denote by $j$ an index satisfying the hypothesis of (\ref{fact}). Using the fact that $0\leq t\leq h_\ell\leq h_j$, we can write
	\begin{multline*}
	\left|R\left((a_k,b_k)\right)-t\right| = t + R\left((b_k,a_k)\right) \leq h_j + R\left((b_k,a_k)\right) \\= (a_j-b_j+1)h_j + \bar s_j(a_j)-s_j(b_j)+\sum_{\substack{0\leq k\leq n\\ k\neq j}} (a_k-b_k)h_k+\bar s_k(a_k)-\bar s_k(b_k),
	\end{multline*}
	and this is 
	$ \leq h_{n+1}-1$, once more by the argument used in the proof of Lemma \ref{riesz-prop}.\\
	
	\noindent The fact (\ref{fact}) is proved. It says that, if $t\in[0,h_\ell]$, in order to see  $\left|R\left((a_k,b_k)\right)-t\right|> h_{n+1}-1$, we must have $a_j=p_j-1$ and $b_j=0$ for all $j$ between $\ell$ and $n$ ; as a consequence, the number of elements $m$ of $M_n$ such that $|m-t|>h_{n+1}-1$ is bounded by $(p_0p_1\cdots p_{\ell-1})^2$. Of course, by symmetry, the result is identical for $t\in[-h_\ell,0]$.
	\\
	
	We have
	$$
	\int_{\R}e^{-it \theta} \prod_{k=0}^n|P_k(\theta)|^2 K_s(\theta)\,d\theta = \frac1{p_0p_1\cdots p_n}\sum_{m\in M_n} \widehat{K_s}(t-m).
	$$
	
	For $|t|\leq s$ we have $\widehat{K_s}(t)=\widehat{\sigma_{n+1,s}}(t)$, and for $|t|\geq s$ we have $\widehat{K_s}(t)=0$. We obtain
	\begin{equation}\label{noyun}
	\int_{\R}e^{-it \theta} \prod_{k=0}^n|P_k(\theta)|^2 K_s(\theta)\,d\theta 
	=\frac1{p_0p_1\cdots p_n}\sum_{\substack{m\in M_n\\|m-t|<s}} \widehat{\sigma_{n+1,s}}(t-m).
	\end{equation}
	
	On the other hand, since $\ds \frac{\text{d}\sigma_{0,s}}{\text{d}\sigma_{n+1,s}}=\prod_{k=0}^n|P_k(\theta)|^2$, we have
	\begin{equation}\label{hierar}
	\widehat{\sigma_{0,s}}(t)=\frac1{p_0p_1\cdots p_n}\sum_{m\in M_n}\widehat{\sigma_{n+1,s}}(t-m).
	\end{equation}
	By construction we have $\widehat{\sigma_{n+1,s}}(t)=0$ for all $t\in[s,h_{n+1}-s]$. \\Moreover, for all $t$, $\left|\widehat{\sigma_{n+1,s}}(t)\right|\leq1$. We deduce from (\ref{hierar}) that
	$$
	\left|\widehat{\sigma_{0,s}}(t)-\frac1{p_0p_1\cdots p_n}\sum_{\substack{m\in M_n\\|m-t|<s}} \widehat{\sigma_{n+1,s}}(t-m)\right|\leq \frac{\#\left\{m\in M_n\mid |m-t|>h_{n+1}-s\right\}}{p_0p_1\cdots p_n}.
	$$
	
	We can conclude that if $|t|\leq h_\ell$, then
	$$
	\left|\widehat{\sigma_{0,s}}(t)-\frac1{p_0p_1\cdots p_n}\sum_{\substack{m\in M_n\\|m-t|<s}} \widehat{\sigma_{n+1,s}}(t-m)\right|\leq \frac{(p_0p_1\cdots p_{\ell-1})^2}{p_0p_1\cdots p_n}.
	$$
	Associated with (\ref{noyun}), this shows that, for all $t\in\R$,
	$$
	\lim_{n\to+\infty} \int_{\R}e^{-it \theta} \prod_{k=0}^n|P_k(\theta)|^2 K_s(\theta)\,d\theta =\widehat{\sigma_{0,s}}(t),
	$$
	and this gives the announced weak convergence result.
\end{proof}


\begin{proof}[\textbf{Proof of Theorem \ref{max-spectr-typ}}]
	With the same argument that allowed us to go from Lemma \ref{Key1} to Proposition \ref{vague}, we deduce
	from Proposition \ref{Keyd}, that the sequence of measures $\prod_{k=0}^{n}|P_k(\theta)|^2 \,d\theta$ is weakly convergent to a limit denoted by $\sigma$.  
	
	Let $S$ be a compact subset of $\R$. For all $\ell\geq1$ and all non negative continuous function $f$ with support in $S$, we have
	\begin{multline*}
	\inf_{\theta\in S} K_{\frac1{\ell}}(\theta)\cdot\int_\R f(\theta) \prod_{k=0}^{n}|P_k(\theta)|^2 \,d\theta\\\leq\int_\R f(\theta) \prod_{k=0}^{n}|P_k(\theta)|^2 K_{\frac1{\ell}}(\theta)\,d\theta\\\leq \sup_{\theta\in S} K_{\frac1{\ell}}(\theta) \cdot\int_\R f(\theta) \prod_{k=0}^{n}|P_k(\theta)|^2 \,d\theta,
	\end{multline*}
	and letting $n$ go to infinity, we have
	$$
	\inf_{\theta\in S} K_{\frac1{\ell}}(\theta)\cdot\int_\R f(\theta) \,d\sigma(\theta)\leq\int_\R f(\theta) \,d\sigma_{0,{\frac1{\ell}}}(\theta)\leq \sup_{\theta\in S} K_{\frac1{\ell}}(\theta)\cdot\int_\R f(\theta) \,d\sigma(\theta).
	$$
	The second inequality proves that the measure $\ds \sum_{\ell\geq1}2^{-\ell} \sigma_{0,\frac1{\ell}}$ is absolutely continuous with respect to $\sigma$.
	In order to use the first inequality, we choose an integer $\ell$ large enough so that the kernel $K_{\frac1{\ell}}$ is strictly positive on $S$, and we conclude that, on $S$, 
	$$
	\sigma\ll\sigma_{0,\frac1\ell}.
	$$
	Hence, on the whole real line we get
	$$
	\sigma\ll\sum_{\ell\geq1}2^{-\ell} \sigma_{0,\frac1{\ell}},
	$$
	and at the end
	$$
	\sigma\approx\sum_{\ell\geq1}2^{-\ell} \sigma_{0,\frac1{\ell}}
	$$
	We recall now (\ref{recusive}), which states that the measure $\sigma_{0,s}$ is absolutely continuous 
	with respect to the measure $\sigma_{n,s}$ and 
	the Radon-Nicodym derivative has only countably many zeros. 
	It follows that the continuous parts of these measure are equivalent. 
	
	Thus the continuous part of $\sum_{\ell \geq 1}2^{-\ell}\sigma_{0,1/\ell}$ 
	is equivalent to the continuous part of the 
	maximal spectral type $\sum_{n\geq0,\ell\geq1}2^{-{(n+\ell)}} \sigma_{n,\frac1{\ell}}$ and the proof of theorem is complete.
	
\end{proof}

\section{On the singularity  criterion of Generalized Riesz Products on $\R$}\label{Bourgain-sin}


The notion of Riesz products on $\R$ is introduced by Peyri\`ere in \cite{Peyriere}. As observed by Peyri\'ere in the context of classical
Riesz products one may check that Zygmund theorem and Peyri\`ere theorem holds for the classical Riesz products in the case of the torus
can be extended to the case of $\R$. We shall show in this section that the same holds for the generalized Riesz
products on $\R$ coming from dynamical systems. More precisely, we shall stated and proved Bourgain singularity
criterion for the generalized Riesz Products on $\R$ and the Guenais sufficient condition on
the $L^1$ flateness of the polynomials which implies the existence of rank one maps with Lebesgue component. Guenais conditions is
connected to the strong $L^1$ flateness of the polynomials.\\



We keep the notations of the CS construction described in Sections \ref{CSC} and \ref{mstcs}. We denote by $\sigma$ the weak limit of the sequence of measures $\prod_{k=0}^{n}|P_k(\theta)|^2 \,d\theta$. By Proposition \ref{Keyd}, the spectral measure $\sigma_{0,s}$ has density $K_s$ with respect to $\sigma$. For fix $s\in(0,1)$, recall that
$$K_s(\theta)=\frac{s}{2\pi}\cdot{\left(\frac{\sin(\frac{s\theta}2)}{\frac{s\theta}2}\right)^2},$$
and $\lambda_s$ is the probability measure of density $K_s$ on $\R$, that is,
$$d\lambda_s(\theta)=K_s(\theta)\,d\theta.$$

Here is the $\R$ version of Bourgain singularity criterion.

\begin{thm}[Bourgain criterion]\label{Bourg-cri}
	Fix $s\in(0,1]$. The following are equivalent:
	\begin{enumerate}
		\item [(i)]  $\sigma_{0,s}$ is singular with respect to Lebesgue measure.
		
		\item[(ii)]  {\it $\inf \left \{\displaystyle  \displaystyle \int_\R
			\prod_{\ell=1}^L\left| {P_{n_\ell}(\theta)}\right|\cdot K_s(\theta)\; d\lambda_s\;:\; L\in
			{\N},~n_1<n_2<\ldots <n_L\right\}=0.$ }
	\end{enumerate}
\end{thm}

We recall here that the maximal spectral type of the flow is singular if and only if, for all $s$, the measure $\sigma_{0,s}$ is singular. (See the end of Section \ref{mstcs}.) 
\\

Let us begin by a Lemma which gives a simplest version of condition \emph{(ii)} of Theorem \ref{Bourg-cri}.

\begin{lemm}\label{CS-B}
	The following are equivalent
	\begin{enumerate}
		\item $\ds \int_{\R} \prod_{k=0}^{N}|P_k(\theta)|\cdot K_s(\theta)\;  d\theta \tend{N}{+\infty}0.$
		\item {\it $\inf \left\{\displaystyle \displaystyle \int_\R 
			\prod_{\ell=1}^L\left| {P_{n_\ell}(\theta)}\right|\cdot d\lambda_s~:~ L\in
			{\N},~n_1<n_2<\ldots <n_L\right\}=0.$ }
	\end{enumerate}
\end{lemm}
\begin{proof}{}The proof is simply a double application of Cauchy-Schwarz inequality. 
	Consider $n_1<n_2<\ldots <n_L$ and $N\geq n_L$. Denote $\mathcal{N}=\left\{n_1<n_2<\ldots <n_L\right\}$ and $\mathcal{N}^c$ its  complement set in $\left\{1,\cdots,N\right\}$. 
	We have
	\begin{multline*}
	\int \prod_{k=0}^{N}|P_k|d\lambda_s = \int \prod_{k\in\mathcal N}|P_k|^{\frac12}  \times \prod_{k\in\mathcal N^c}|P_k|^{\frac12}\prod_{k=0}^N|P_k|^{\frac12}\  d\lambda_s\\\leq
	\left(\int \prod_{k\in\mathcal N}|P_k|\  d\lambda_s\right)^{\frac12}\left(\int \prod_{k\in\mathcal N^c}|P_k|\times\prod_{k=0}^N|P_k|\; d\lambda_s\right)^{\frac12}
	\\\leq
	\left(\int \prod_{k\in\mathcal N}|P_k|\  d\lambda_s\right)^{\frac12}\left(\int \prod_{k\in\mathcal N^c}|P_k|^2\  d\lambda_s\right)^{\frac14}\left(\int \prod_{k=0}^N|P_k|^2\  d\lambda_s\right)^{\frac14}\\=
	\left(\int \prod_{k\in\mathcal N}|P_k|\  d\lambda_s\right)^{\frac12}.
	\end{multline*}
	
	The last equality comes from two uses of Lemma \ref{riesz-prop}.
	
\end{proof}

\begin{proof} [Proof of Theorem \ref{Bourg-cri}]Assume that \emph{(i)} holds. Denote by $\lambda_s$ the measure $d\lambda_s(\theta)=K_s(\theta)\;d\theta$.  To prove that $\sigma_{0,s}$ is singular , it suffices to show that
	$\sigma_{0,s} \perp \lambda_s$. For that it suffices to show that 
	for any $\epsilon>0$, there is a Borel set $E$ with
	$\lambda_s(E)<\epsilon$ and $\sigma_{0,s}(E^c)<\epsilon$. Let $0<\epsilon<1$.
	
	Fix $N_0$ such that, for any $N>N_0$ we have
	$\int \prod_{k=0}^{N}|P_k| \;d\lambda_s <\epsilon^2$. The
	set $E=\left\{\theta\in\R\;:\;\prod_{k=0}^{N}|P_k(\theta)| \geq\epsilon\right\}$ satisfies:
	$$\lambda_s(E)\leq \frac1\epsilon\left\|\prod_{k=0}^{N}P_k\right\|_1\leq \epsilon^2/\epsilon=\epsilon,$$
	and, by the Portmanteau Theorem, since $E^c$ is open set, it follows 
	\begin{eqnarray*}
		\sigma_{0,s}(E^c) &\leq &\liminf_{M\to+\infty}\int_{E^c}\prod_{k=0}^{M}|P_k|^2\; d\lambda_s \\
		& \leq & \liminf_{M\to+\infty}\int_{E^c}\prod_{k=0}^{N}|P_k|^2  \prod_{k=N+1}^{M}|P_k|^2\;d\lambda_s \\
		& \leq& \epsilon^2 \lim_{M\to+\infty} \int_{\R}\prod_{k=N+1}^{M}|P_k|^2\; d\lambda_s=
		\epsilon^2 <\epsilon.
	\end{eqnarray*}
	
	\medskip
	
	For the converse. Given $0<\epsilon<1$,
	there exists a continuous function  $\varphi$  on $\R$ such that:
	$$0\leq\varphi\leq 1,\qquad
	\sigma_{0,s}(\{\varphi\neq 0\})\leq\epsilon\qquad\hbox{and}\quad
	\lambda_s(\{\varphi\neq 1\})\leq\epsilon.$$
	
	Let $f_N=\ds \prod_{k=1}^N |P_k|$. By a double use of Cauchy-Schwarz inequality, we have
	\begin{multline*}
	\int f_N\  d\lambda_s  =  \ds \int_{\{\varphi\neq 1\}} f_N\ d\mu_s + \int_{\{\varphi= 1\}} f_N\ d\lambda_s 
	\\
	\leq \lambda_s(\{\varphi\neq 1\})^{1/2}\left(\ds\int_\R f_N^2\ d\mu_s\right)^{1/2} + \left(\ds \int_{\{\varphi= 1\}} f_N^2\ d\lambda_s\right)^{1/2} 
	\mu_s(\{\varphi= 1\})^{1/2}
	\\\leq \sqrt{\epsilon} + \left(\ds \int f_N^2\, \varphi\ d\lambda_s\right)^{1/2}. 
	\end{multline*}
	But since $d\sigma_{0,s}=W- \lim f_N^2 \;d\lambda_s$, 
	$$\lim_{N\rightarrow\infty} \int f_{N}^2\, \varphi\ d\lambda_s=\int \varphi\ d\sigma_{0,s}
	\leq \sigma_{0,s}(\{\varphi\neq 0\})\leq\epsilon.$$
	Thus, 
	$\limsup\ds \int f_N\ d\mu_s \leq 2 \sqrt{\epsilon}$.
	Since $\epsilon$ is arbitrary,  we get
	$\lim_{N\rightarrow\infty} \ds\int f_{N}\ d\lambda_s=0,$
	and this completes the proof.
\end{proof}
In the following lemma we state a sufficient condition for the existence of an absolutely continuous component for the generalized Riesz product associated to CS-construction
of a rank one flow. In the case of $\Z$ action, the lemma is due to 
M. Guenais \cite{Guenais}, and the proof is similar. We present its proof by sake of completeness. We keep the notation $d\lambda_s(\theta)=K_s(\theta)\;d\theta$.
\begin{lemm}\label{Absc-Guenais}
	If $\displaystyle \sum_{k=1}^{+\infty}\sqrt{1-\left(\int_{\R}|P_k(\theta)|\;d\lambda_s(\theta)\right)^2}<\infty$ then $\sigma_{0,s}$ admits an absolutely continuous
	component.
\end{lemm}
(Of course, if there exists $s$ such that $\sigma_{0,s}$ admits an absolutely continuous
component, then it is the same for the maximal spectral type of the flow.)
\begin{proof}{} We denote by $\|\cdot\|_p$ the norm in $L^p(\lambda_s)$.\\
	
	For all functions $P$ and $Q$ in $L^2(\lambda_s)$, by Cauchy-Schwarz inequality
	we have 
	\begin{multline}\label{Guenais}
	\|P\|_1 \|Q\|_1 -\|PQ\|_1=-\ds \int\left(|P|-\|P\|_1\right)\left(|Q|-\|Q\|_1\right)\ d\lambda_s\\\leq \||P|-\|P\|_1\|_2\; \||Q|-\|Q\|_1\|_2.\quad \quad \quad \quad \quad \quad \quad 
	\end{multline}
	
	By assumption,  
	$\displaystyle \sum_{k=1}^{+\infty}\sqrt{1-\|P_k\|_1^2}<\infty$. Hence $\displaystyle \sum_{k=1}^{+\infty}1-\|P_k\|_1^2<\infty$ and the infinite product $\ds\prod_k\|P_k\|_1$ is convergent:
	\begin{equation}\label{inf-prod}\prod_{k=0}^{+\infty}\|P_k\|_1>0.\end{equation} 
	Let $n_0\leq n$ be positive integers.
	If  $P=P_n$ and $Q= \ds \prod_{k=n_0}^{n-1}P_k$, \\then $\||P|-\|P\|_1\|_2=\sqrt{1-\|P\|_1^2}$ and $\||Q|-\|Q\|_1\|_2\leq1$ ;
	hence by \eqref{Guenais} we have
	$$
	\|PQ\|_1\geq\|P\|_1 \|Q\|_1-\sqrt{1-\|P\|_1^2}.
	$$
	Using also the fact that $\|P_k\|_1\leq1$, we obtain by induction
	$$
	\left\|\prod_{k=n_0}^{n}P_k\right\|_1\geq\prod_{k=n_0}^n\|P_k\|_1-\sum_{k=n_0}^n\sqrt{1-\|P_k\|_1^2}\geq\prod_{k=n_0}^n\|P_k\|_1-\sum_{k=n_0}^{+\infty}\sqrt{1-\|P_k\|_1^2}\;.
	$$
	From lemma assumption and from (\ref{inf-prod}) we deduce that, for $n_0$ chosen large enough
	$$
	\lim_{n\to+\infty}\prod_{k=n_0}^n\|P_k\|_1-\sum_{k=n_0}^{+\infty}\sqrt{1-\|P_k\|_1^2}>0,
	$$
	hence the sequence $\left(\prod_{k=n_0}^{n}P_k\right)$ does not go to zero in $L^1$-norm.
	It follows from Bourgain criterion that the generalized Riesz product 
	$ \prod_{k=n_0}^{+\infty}|P_k|^2 d\lambda_s$ is not purely singular. As 
	$ \prod_{k=0}^{n_0-1}|P_k|^2$ has only countably  many zeros, we conclude that $\sigma$ admits 
	also an absolutely continuous component.
\end{proof}

As observed by Bourgain and Klemes in the case of the torus, in order 
to prove the singularity of the spectrum of the rank one it is sufficient to prove that a weak limit point of the sequence $\left(\left||P_m|^2-1\right|\right)$ is bounded below by a positive constant. More precisely we have the following
proposition. 
The proof is similar to the one given in the case of $\Z$ in \cite{elabdaletds}. 
However, for the sake of completeness, we will give the proof here.

\begin{Prop}\label{CDsuffit} Let $E$ be an infinite set of positive integers. Suppose that there exists a constant $c>0$ such that, for all integer $L>0$ and all integers $0\leq n_1<n_2<\ldots<n_L$,
	$$\liminf_{\overset{m\longrightarrow +\infty}{m\in E}} \int_{\R}\left||P_m|^2-1\right|\prod_{\ell=1}^L\left|P_{n_\ell}\right|\;d\lambda_s \geq c \int_{\R}\prod_{\ell=1}^L\left|P_{n_\ell}\right| \;d\lambda_s.$$
	Then $\sigma_{0,s}$ is singular.
\end{Prop}

The following lemma comes from \cite{Bourgain} and it can also be found in \cite{elabdalisr} (Lemma 3.2).

\begin{lemm}\label{limsup} Let $E$ be an infinite set of positive integers. Let $L$ be a positive integer and  $
	0\leq n_1<n_2<\cdots<n_L$ be integers. Denote $Q=\prod_{\ell=1}^L\left | {P_{n_\ell}}\right|$. Then
	\begin{eqnarray*}
		\displaystyle \limsup_{\overset{m\longrightarrow +\infty}{m\in E}}\int  Q \left| P_m\right|\;d\lambda_s
		\leq 
		\displaystyle \int Q \;d\lambda_s -\frac 18\left(
		\liminf_{\overset{m\longrightarrow +\infty}{m\in E}}
		\displaystyle \int  Q \left| \left| P_m\right| ^2-1\right|
		\;d\lambda_s\right) ^2. 
	\end{eqnarray*}
\end{lemm}
The proof of this lemma relies on the following inequality (for $m>n_L$):

\begin{eqnarray*}
	\displaystyle \int  Q \left| P_m\right|\;d\lambda_s
	\leq 
	\frac 12\left( \displaystyle \int Q \;d\lambda_s +\displaystyle \int
	Q \left| P_m\right| ^2 d\lambda_s\right) -\frac 18\left(
	\displaystyle \int  Q \left| \left| P_m\right| ^2-1\right|
	\;d\lambda_s\right) ^2. 
\end{eqnarray*}

\begin{proof}[Proof of Proposition \ref{CDsuffit}]${}$\\
	Let $\displaystyle \beta=\inf\left\{\int Q\;
	d\lambda_s~:~Q=\prod_{\ell=1}^L\left | {P_{n_\ell}}\right|, L \in \N,
	0\leq n_1<n_2<\cdots<n_L\right\}$. Then, for any such $Q$, we have
	\begin{eqnarray*}
		\int Q
		\;d\lambda_s \geq \beta\quad\text{and}\quad \liminf \int Q |P_m| \;d\lambda_s \geq
		\beta.
	\end{eqnarray*}
	
	\noindent{} Thus by Lemma \ref{limsup} and by taking the infimum over all $Q$ we get
	\begin{eqnarray*}
		\beta \leq \beta-\frac18(c\beta)^2
	\end{eqnarray*}
	\noindent It follows that
	\[
	\beta=0,
	\] 
	and the proposition follows from Theorem \ref{Bourg-cri}.
\end{proof}



The previous argument is refined as in \cite{elabdaletds}, as follows.\\

\begin{prop}\label{Gw-lim} Let $s>0$. Then,
there exist a subsequence of
the sequence $\left (\left |\left| P_m(t)\right|-1 \right|\right)$
which converge weakly in $L^2(\lambda_s)$ to some non-negative function $\phi$ which
satisfy  $ \phi \leq $2, almost surely with respect the Lebesgue
measure.
\end{prop}
\begin{proof}
	\noindent{}The
	sequence $\left |\left| {P_m(t)}%
	\right|-1 \right |$ is bounded in $L^2(\lambda_s)$. It follows that there
	exist a subsequence which converges weakly to some non-negative
	$L^2(\lambda_s)$ function $\phi$. Let $\omega$ be a non-negative continuous
	function, then we have
	
	\begin{eqnarray*}
		&&\int \omega\left |\left| {P_m(t)} \right|-1 \right |d\lambda_s(t)\\
		&&\leq \int \omega\left| {P_m(t)} \right| d\mu_s(t)+\int \omega(t)
		d\lambda_s(t)\\
		&&\leq {(\int \omega(t) d\mu_s(t))}^{\frac12} {(\int \omega (t)\left |
			{P_m(t)} \right|^2d\lambda_s(t))}^{\frac12}+\int \omega(t) d\lambda_s(t).
	\end{eqnarray*}
	\noindent{}Hence
	\[
	\int \omega(t) \phi(t) d\lambda_s(t) \leq 2 \int \omega(t) d\lambda_s(t).
	\]
	\noindent{}Now, Apply Lusin theorem to get that for any Borel set
	$A$ and for any $\varepsilon>0$,  there exist a compact set $K$
	such that
	\[
	(\alpha+\lambda_s)(\R \setminus K) < \varepsilon {\rm {~~and~~}}
	{\chi_A}_{|K} {\rm {~is~a~contiuous~non-negative~ function}}.
	\]
	\noindent where ${\chi_A}_{|K}$ is the restriction of the
	indicator function of $A$ to $K$. Hence, we have
	\[
	\int_A \omega(t) ~~\phi(t) d\lambda_s \leq 2 \lambda_s(A).
	\]
	\noindent and the proposition follows.
\end{proof}

\noindent Put
\[
\alpha_s = \phi~~ d\lambda_s.
\]
\noindent By applying the same arguments as before it is easy to see that we have the following.\\

\begin{Prop}\label{A-singular} For any $s>0$,
	$\alpha_s \bot \sigma.$.
\end{Prop}


In the next section we will generalize the previous results by establishing a formula for Radon-Nikodym derivative of two Riesz products on real line obtained in \cite{elabdal-Nad1}.

\section{A Formula for Radon Nikodym Derivative.}

Let $s \in (0,1]$ and consider two generalized Riesz products $\mu_s$ and $\nu_s $ based on polynomials $P_j, j=1,2,\cdots$ and $Q_j, j =1,2,\cdots$ where $\nu_s$ is continuous except for a possible mass at $1$. Under suitable assumptions we prove the formula:\\
$$  \sqrt{\frac{d\mu_s}{d\nu_s}} =\lim_{n\rightarrow \infty}\frac{\prod_{j=1}^n \mid P_j\mid}{\prod_{j=1}^n\mid Q_j\mid},$$
in the sense of $L^1(\R, \nu_s)$ convergence.\\
Let $\sigma$ and $\tau$ be two measures on the real line. Then, by Lebesgue decomposition of  $\sigma$ with respect to $\tau$, we have
\[
\sigma=\frac{d\sigma}{d\tau} d\tau+\sigma_{\mathfrak{s}},
\]
where $\sigma_{\mathfrak{s}}$ is singular to $\tau$ and $\frac{d\sigma}{d\tau}$ is the Radon-Nikodym derivative. In the case of two Riesz products $ \mu_s = \prod_{j=1}^\infty\mid P_j\mid^2$ and $ \nu_s =\prod_{j=1}^\infty\mid Q_j\mid^2$, we are able  to extend el Abdalaoui-Nadkarni theorem \cite{elabdal-Nad1} by proving  that the ratios $ \frac{\prod_{j=1}^n\mid P_j\mid}{\prod_{j=1}^n\mid Q_j\mid}, k=1,2,\cdots$,  converge in $L^1(\nu_s)$ to $\sqrt\frac{d\mu_s}{d\nu_s}$, assuming that $\nu_s$ has no point masses except possibly at $1$. 

\begin{thm}\label{th6}
	Let $ \mu_s=\prod_{j=0}^{\infty}\mid P_j\mid^2$,
	$ \nu=\prod_{j=0}^{\infty}\mid Q_j\mid^2$ be two generalized Riesz products. Let
	$$\mu_{n,s}=\prod_{j=n+1}^{\infty}\mid P_j\mid^2,~~
	\nu_{n,s}=\prod_{j=n+1}^{\infty}\mid Q_j\mid^2
	$$
	Assume that
	\begin{enumerate}
		\item  $\nu_s=\nu_s'+b\delta_1$, $\nu_s'$ is continuous measure, $0\leq b <1$.
		\item $ \prod_{j=0}^{n}\mid P_j\mid^2 d\nu_{n,s} \longrightarrow \mu_s$ weakly as $n\longrightarrow \infty$
		\item $ \prod_{j=0}^{n}\mid Q_j\mid^2 d\mu_{n,s} \longrightarrow \nu_s$ weakly as $n\longrightarrow \infty$
	\end{enumerate}
	Then the finite products $ R_n= \prod_{k=1}^{n}{\left|\frac{P_k(t)} {Q_k(t)}\right|}, n=1,2,\cdots$ converge in $L^1(\R,\nu_s)$ to $\sqrt{\frac{d\mu_s}{d\nu_s}}$.
\end{thm}

\noindent{} To prove this we need the following proposition.
\begin{Prop}\label{prop1}
	The sequence $\ds  \prod_{j=0}^{n}\left|\frac{P_j(t)}{Q_j(t) }\right|, n=1,2,\cdots$ converges weakly in $L^2(\R,\nu_s)$ to $\ds \sqrt{\frac{d\mu_s}{d\nu_s}}$.
\end{Prop}

\begin{proof} Put
	$ f=\sqrt{\frac{d\mu_s}{d\nu_s}}$
	and let $n$ be a positive integer. Now
	
	$$\bigintss_{\R} R_n^2 d\nu_s = \bigintss_{\R} \prod_{j=1}^{n}\mid {}P_j|^2 d\nu_{n,s} \rightarrow \bigintss_{\R}d\mu_s = 1$$
	by assumption (2). Hence $\int_{\R} R_n^2 d\nu_s , n =1,2,\cdots$ remain bounded. Thus, the weak closure of $ R_n(t), n =1,2,\cdots$ in $ L^2(\R, \nu_s)$ is not empty.
	
	We show that this weak closure has only one point, namely,
	$\sqrt{\frac{d\mu_s}{d\nu_s}}$. Indeed, let $g$ be a weak subsequential limit, say, of $R_{n_j}(t), j =1,2,\cdots$. Then, for any continuous positive function $h$, we have, by judicious applications of Cauchy-Schwarz inequality,
	
	$$\Biggl(\bigintss_{\R} f(t) h(t) d\nu_s(t)\Biggr)^2 = \Biggl(\bigintss_{\R} h(t) R_{n_j}(t) \frac{1}{R_{n_j}(t)}\sqrt{\frac{d\mu_s}{d\nu_s}}  d\nu_s(t)\Biggr)^2 $$
	$$\leq  \Biggl(\bigintss_{\R} h(t) R_{n_j}(t) d\nu_s(t)\Biggr) \Biggl(\bigintss_{S_1} h(t) R_{n_j}(t) \frac{1}{R_{n_j}^2(t)}\frac{d\mu_s}{d\nu_s} d\nu_s(t)\Biggr)$$
	$$\leq \Biggl(\bigintss_{\R} h(t) R_{n_j}(t) d\nu(t)\Biggr) \Biggl(\bigintss_{\R} h(t) \frac{1}{R_{n_j}(t)} d\mu_s\Biggr)$$
	$$\leq \bigintss_{\R} h(t) R_{n_j}(t) d\nu(t) \Biggl(\bigintss_{\R} h(t) d\mu_s \Biggr)^{\frac{1}{2}}\Biggl(\bigintss_{\R} h(t) \frac{d\mu}{R_{n_j}^2(t)}\Biggr)^{\frac{1}{2}}$$
	$$\leq  \Biggl(\bigintss_{\R} h(t) R_{n_j}(t) d\nu_s(t)\Biggr) \Biggl(\bigintss_{\R} h(t) d\mu_s\Biggl)^{\frac{1}{2}}
	\Biggl(\bigintss_{\R} h(t)\mid\prod_{k=1}^{n_j} Q_k\mid^2 d\mu_{n_j,s}\Biggr)^{\frac{1}{2}} $$
	
	Letting $j \rightarrow +\infty$, from our assumption (3), we get
	
	$$\Biggl(\bigintss_{\R} f h d\nu_s\Biggr)^2 \leq
	\Biggl(\bigintss_{\R} h g d\nu_s\Biggr) \Biggl(\bigintss_{\R} h d\mu_s\Biggr)^{\frac{1}{2}}
	\Biggl(\bigintss_{\R} h d\nu_s\Biggr)^{\frac{1}{2}}  \eqno (2).$$
	
	But, since the space of continuous functions is dense in $L^2(\mu+\nu)$, we deduce from (2) that, for any Borel set $B$,
	
	$$\Biggl(\bigintss_{B} f  d\nu_s\Biggr)^2 \leq
	\Biggl(\bigintss_{B}  g d\nu_s\Biggr) \Biggl(\bigintss_{B} d\mu_s\Biggr)^{\frac{1}{2}}
	\Biggl(\bigintss_{B} d\nu_s\Biggr)^{\frac{1}{2}}.
	$$
	By taking a Borel set $E$ such that $\mu_s(E)=0$ and $\nu_s(E)=1$, we thus get, for any $B \subset E$,
	
	$$\Biggl(\bigintss_{B} f  d\nu_s\Biggr)^2 \leq
	\Biggl(\bigintss_{B}  g d\nu_s\Biggr) \Biggl(\bigintss_{B} f^2 d\nu_s \Biggr)^{\frac{1}{2}}
	\Biggl(\bigintss_{B} d\nu_s\Biggr)^{\frac{1}{2}}.$$
	
	It follows from Martingale convergence theorem that:
	$$ f(t) \leq g(t) {\textrm {~for~almost~all~}}z {\textrm{~with~respect~to~}} \nu.$$
	
	Indeed, let ${{\mathcal  P}}_n = \{A_{n,1}, A_{n,2}\cdots, A_{n, k_n}\},$  $n =1,2,\cdots$, be a refining sequence of finite partitions of $E$ into Borel sets such that they tend to the partition of singletons. If $\{x\} = \ds \bigcap_{n=1}^\infty A_{n,j_n}$,
	
	$$\Biggl(\frac{1}{\mu_s (A_{n, j_n})}\bigintss_{B} f  d\nu_s\Biggr)^2 \leq$$
	$$
	\Biggl(\frac{1}{\mu_s (A_{n, j_n})}\bigintss_{A_{n,j_n}}  g d\nu_s\Biggr) \Biggl(\frac{1}{\mu_s (A_{n, j_n})}\bigintss_{A_{n,j_n}} f^2(t) d\nu_s \Biggr)^{\frac{1}{2}}
	\Biggl(\frac{1}{\mu_s (A_{n, j_n})}\bigintss_{A_{n,j_n}} d\nu_s\Biggr)^{\frac{1}{2}}.$$
	Letting $n\rightarrow \infty$ we have, by Martingale convergence theorem as applied to the theory of derivatives,
	for a.e $x \in E$ w.r.t. $\nu$,
	$$(f(x))^2 \leq g(x) f(x), ~~{\rm{whence}} ~~f(x)\leq g(x)$$
	
	\noindent{}For the converse note that for any continuous positive function $h$ we have
	$$\bigintss_{\R} g h d\nu_s = \lim_{j \longrightarrow +\infty}\bigintss_{\R} h(t) R_{n_j}(t)   d\nu_s $$
	$$ \leq   \lim_{j \longrightarrow \infty} \Biggl(\bigintss_{\R} h R_{n_j}^2 d\nu_s\Biggr)^{\frac{1}{2}} \Biggl(\bigintss_{\R} h d\nu_s\Biggr)^{\frac{1}{2}}$$
	$$ \leq  \Biggl(\bigintss_{\R} h d\mu_s\Biggr)^{\frac{1}{2}}  \Biggl(\bigintss_{\R} h  d\nu_s\Biggr)^{\frac{1}{2}}.$$
	
	\noindent{}As before we deduce that $g(t) \leq f(t)$ for almost all $t$ with respect to $\nu_s$.
	Consequently, we have proved that $g=f$ for almost all $t$ with respect to $\nu_s$ and this complete the proof of the proposition.
\end{proof}
\begin{proof}[\bf {Proof of Theorem \ref{th6}}] We will show that
	$\beta_n \stackrel{\textrm {def}}{=} \bigintss_{\R} \mid R_n-f\mid d\nu \rightarrow 0$ as $n \rightarrow \infty$,
	where $f=\sqrt{\frac{d\mu_s}{d\nu_s}}$. Now,
	
	$$\frac{d\mu_s}{d\nu_s}=R_n^2(t)\frac{d\mu_{n,s}}{d\nu_{n,s}}~~ {\rm{and}}~~  \sqrt{\frac{d\mu_s}{d\nu_s}}=R_n(t)\sqrt{\frac{d\mu_{n,s}}{d\nu_{n,s}}}$$
	
	\noindent{}Put $$f^2_n =\frac{d\mu_{n,s}}{d\nu_{n,s}},$$ Then,
	
	$$\bigintss_{\R}f_n^2d\nu_s = \bigintss_{\R}\prod_{k=1}^n\mid Q_k\mid^2d\mu_{n,s} \rightarrow \bigintss_{\R}d\nu_s =1,$$
	by assumption (3). The functions $f_n, n=1,2,\cdots$ are therefore bounded in $L^2(\R, \nu_s)$.
	Hence, there exists a subsequence
	$f_{n_j} = \sqrt{\frac{d\mu_{n_j,s}}{d\nu_{n_j,s}}}, j = 1,2, \cdots$ which converges weakly to some $L^2(\R, \nu_s)$-function $\phi$. We show that $0 \leq \phi \leq 1$ a.e ($\nu_s$). For any continuous positive function $h$, we have
	
	$$\Biggl(\bigintss_{\R} h f_{n_j} d\nu_s\Biggr)^2 \leq \Biggl(\bigintss_{\R} h d\nu_s\Biggr) \Biggl(\bigintss_{\R} h f_{n_j}^2 d\nu_s\Biggr)$$
	$$\leq \Biggl(\bigintss_{\R} h d\nu_s\Biggr) \Biggl(\bigintss_{\R} h \frac{d\mu_{n_j}}{d\nu_{n_j,s}}d\nu_s\Biggr).$$
	
	\noindent{} Hence, by letting $j$ go to infinity combined with our assumption (3), we deduce that
	$$ \bigintss_{\R} h(t) \phi(t) d\nu_s \leq  \bigintss_{\R} h(t) d\nu_s.$$
	Since this hold for all continuous positive functions $h$, we conclude that $0 \leq \phi \leq 1$ for almost all $t$ with respect to $\nu_s$. Thus any subsequential limit of the sequence $f_n, n=1,2,\cdots$ assumes values between $0$ and $1$.
	Now, for any subsequence $n_j, j =1,2,\cdots$ over which $f_{n_j}, j=1,2,\cdots$ has a weak limit , from our assumption (2) combined with Cauchy-Schwarz inequality, we have
	$$
	\Biggl(\bigintss_{\R}|R_{n_j}-f| d\nu_s\Biggr)^2=\Biggl(\bigintss_{\R} |R_{n_j}-R_{n_j} f_{n_j}| d\nu_s\Biggr)^2$$
	
	$$=\Biggl(\bigintss_{\R} R_{n_j}|1-f_{n_j}| d\nu_s\Biggr)^2 $$
	$$\leq \Biggl(\bigintss_{\R} R_{n_j}(t)|1-f_{n_j}|^2 d\nu_s\Biggr) \Biggl(\bigintss_{\R} R_{n_j}(t) d\nu_s \Biggr)$$
	$$\leq \Biggl(\bigintss_{\R} R_{n_j} d\nu_s-2\bigintss_{\R} R_{n_j}f_{n_j}d\nu_s+\bigintss_{\R} R_{n_j}(f_{n_j})^2 d\nu_s\Biggr)\Biggl(\bigintss_{\R}R_{n_j}d\nu_s\Biggr)$$
	
	$$\leq \Biggl(\bigintss_{\R} R_{n_j} d\nu_s-2 \bigintss_{\R} f d\nu_s+\bigintss_{\R} R_{n_j}f_{n_j}. f_{n_j}  d\nu_s\Biggr)\Biggl(\bigintss_{\R}R_{n_j}d\nu_s\Biggr)$$
	
	$$\leq \Bigl(\bigintss_{\R} R_{n_j} d\nu-2 \bigintss_{\R} f d\nu+\bigintss_{\R} f. f_{n_j}  d\nu\Bigr)\Bigl(\bigintss_{\R}R_{n_j}d\nu\Bigr)$$
	
	\noindent{}Hence, letting $j$ go to infinity,
	$$\Biggl(\lim_{j \rightarrow \infty}\bigintss_{\R}\mid R_{n_j}-f\mid d\nu_s\Biggr)^2$$
	$$\leq \bigintss_{\R} f d\nu_s-2 \bigintss_{\R} f d\nu_s+\bigintss_{\R} f. \phi  d\nu_s$$
	$$\leq \bigintss_{\R}(\phi(t)-1)f(t) d\nu_s(t).$$
	$$\leq 0,$$
	\noindent{}and this implies that $R_{n_j}, j =1,2,\cdots $  converges to $ f$ in $L^1(\R, \nu_s)$  and the proof of the theorem is achieved.
\end{proof}
\begin{rem}\label{rem3}
	\textnormal{Notice that
		$\ds \int_{\R} \frac{d\mu_s}{d\nu_s} d\nu_s=1,$
		implies the convergence of $ \prod_{j=0}^{N}|R_j|$ to $ \sqrt{\frac{d\mu_s}{d\nu_s}}$ in $L^2(d\nu_s)$, by virtue of the classical results on ``when weak convergence implies strong convergence".}
\end{rem}
\noindent{}We further have \cite{Nadkarni1}
\begin{Cor}\label{cor3}Let $s \in (0,1].$
	Two generalized Riesz products $\mu_s = \prod_{j=1}^\infty\big| P_j\big|^2$,
	$\nu_s = \prod_{j=1}^\infty\big| Q_j\big|^2$  satisfying the conditions of Theorem \ref{th6} are mutually singular
	if and only if $$\bigintss_{\R} \prod_{j=0}^n\Big|\frac{P_j}{Q_j}\Big| d\nu_s \rightarrow 0~~ {\rm{as}} ~~n\rightarrow \infty.$$
\end{Cor}
Corollary \ref{cor3} generalized Bourgain criterion (Theorem \ref{Bourg-cri}.)
\section{Mahler Measure of the spectral type of rank one flow.}
In this section, we will extended, as far as possible, the formula established by el Abdalaoui-Nadkarni in \cite{elabdal-Nad1}. Our extension is based on the entropy method.\\

For that, we start by introducing the notion of Mahler measure. Let $(X,\B,\rho)$ be a probability space and $f \in L^1(X,\rho)$. The Mahler measure of the measure $\mu=f(x)d\rho(x)$ is defined by
\[
M(\mu)=\exp \Biggl(\bigintss_{X} \log\big(\big|f(x)\big|\big) d\rho \Biggr).
\]
In our case the measure $\rho$ is equal to $\lambda_s = K_s(t)dt$ and the Mahler measure of a trigonometic polynomial $P$ and a measure $\mu$ is given, respectively, by 
\begin{align}\label{def-Mahler}
M_s(P)&=\exp \Biggl(\bigintss_{\R} \log\big(\big|P(t)\big|\big)  d\lambda_s(t)  \Biggr),\\
M_s(\mu)&=\exp \Biggl(\bigintss_{\R} \log\Big(\Big|\frac{d\mu}{d\lambda_s}(t)\Big|\Big) d\lambda_s(t) \Biggr)
\end{align}
\noindent{}Here are some elementary properties of  the Mahler measure. But, we provide a proof for the reader's convenience.
\begin{Prop}\label{basic}Let $(X,\B,\rho)$ be a probability space. Then,
	for any two positive functions $f,g \in L^1(X,\rho)$, we have
	\begin{enumerate}[i)]
		\item $M(f)$ is a limit of the norms $||f||_{\delta}$ as $\delta$ goes to $0$, that is,  \[
		||f||_{\delta} \setdef \Biggl(\bigintss f^{\delta} d\rho\Biggr)^{\frac1{\delta}} \tend{\delta}{0} M(f),
		\]
		provided that $\log(f)$ is integrable.
		\item If $\rho\Big\{ f >0 \Big\} <1$ then $M(f)=0$.
		\item If $0<p<q<1$, then $\bigl\|f\bigr\|_p \leq \bigl\|f\bigr\|_q$.
		\item If $0<p< 1$, then $M(f) \leq \bigl\|f\bigr\|_p$.
		\item $\ds \lim_{\delta \longrightarrow 0}\int f^\delta d\rho = \rho\Big\{f>0\Big\}.$
		\item $M(f) \leq \bigl\|f\bigr\|_1$.
		\item $M(fg)=M(f)M(g)$.
	\end{enumerate}
\end{Prop}
\begin{proof}We start by proving ii). Without loss of generality, assume that $f \geq 0$ and put
	$$B=\Big\{ f >0 \Big\},$$
	and let $\delta=1/k$  be in $]0,1[$, $k \in \N^*$. Then $1/(1/\delta)+1/(1-\delta)=1/k+(k-1)/k=1$. Hence, by H\"{o}lder inequality, we have
	\begin{eqnarray*}
		\bigintss f^{\delta} d\rho &=& \bigintss f^{1/k} .\1_B d\rho \\
		&\leq& \Biggl(\bigintss (f^{1/k})^k dz\Biggr)^{1/k} \Biggl(\bigintss \1_B^{k/k-1} dz\Biggr)^{k-1/k}\\
		&\leq& \Biggl(\bigintss f d\rho\Biggr)^{1/k} \Biggl(\bigintss \1_B dz\Biggr)^{k-1/k}\\
		&\leq& \Biggl(\bigintss f d\rho\Biggr)^{1/k} \Bigl(\rho(B)\Bigr)^{(k-1)/k}
	\end{eqnarray*}
	Therefore we have proved
	\begin{eqnarray*}
		||f||_{\delta} &\leq& \Biggl(\bigintss f d\rho \Biggr) \Bigl(\rho(B)\Bigr)^{(1-\delta)/\delta}\\
		&\leq& \Biggl(\bigintss f d\rho \Biggr) \Big(\rho(B)\Big)^{k-1} \tend{k}{+\infty}0,
	\end{eqnarray*}
	To prove i), apply the Mean Value Theorem to the following functions
	\[\left\{
	\begin{array}{ll}
	\delta \longmapsto x^\delta, & \hbox{if $x \in ]0,1[;$}  \\
	t \longmapsto t^\delta , & \hbox{if $x>1,$}
	\end{array}
	\right.
	\]
	Hence, for any $\delta \in ]0,1[$ and for any $x>0$, we have
	\[
	\Biggl|\frac{x^\delta-1}{\delta}\Biggr| \leq x+\Bigl|\log(x)\Bigr|.
	\]
	Furthermore, it is easy to see that
	\[
	\frac{f^\delta-1}{\delta}=\frac{e^{\delta \log(f)}-1}{\delta}
	\tend{\delta}{0}\log(f),
	\]
	and, by Lebesgue Dominated Convergence Theorem, we get that
	\[
	\bigintss\frac{f^\delta-1}{\delta} d\rho \tend{\delta}{0} \bigintss \log(f) d\rho.
	\]
	On the other hand, for any $\delta \in ]0,1[$
	\[
	\bigl|\bigl|f\bigr|\bigr|_{\delta}=\exp\Biggl({\frac1{\delta}}\log\Biggl(\bigintss f^{\delta} d\rho\Biggr)\Biggr),
	\]
	and for a sufficiently small $\delta$, we can write
	\[
	{\frac1{\delta}}\log\Biggl(\bigintss f^{\delta} d\rho\Biggr)\sim
	\bigintss \frac{f^\delta-1}{\delta} d\rho
	\]
	since $\log(x) \sim x-1$ as $x \longrightarrow 1$.
	Summarizing we have proved
	\[
	\lim_{\delta \longrightarrow 0}||f||_{\delta}=exp\Biggl(\bigintss \log(f) d\rho\Biggr)=M(f).
	\]
	For the proof of iii) and iv), notice that the function $x \mapsto \ds x^\frac{q}{p}$ is a convex function and
	$x \mapsto \log(x)$ is a concave function. Applying Jensen's inequality to $\ds \bigintss \bigl|f\bigr|^p d\rho$ we get
	$$\bigl\|f\bigr\|_p \leq \bigl\|f\bigr\|_q,~~~~~~~~~
	\bigintss \log(\bigl|f\bigr|) d\rho \leq \log\Bigl(\bigr\|f\bigl\|_p \Bigr),$$
	and this finish
	es the proof, the rest of the proof is left to the reader.
\end{proof}
Szeg\"{o}-Kolmogorov-Krein theorem established a connection between a given measure and the Mahler measure of its derivative. Precisely, we have

\begin{thm}[Szeg\"{o}, Kolmogorov-Krein {\cite[p.49]{Hoffman}.}]\label{Szego}Let $\sigma$ be a finite positive Baire measure on the unit circle and let $h$ be the derivative of $\sigma$ with respect to normalized Lebesgue measure. Then, for any $r>0$,
	\[
	M(h)=\inf_{P}\Big\|1-P\Big\|_r^{r}=\inf_{P}\Biggl(\bigintss \Big|1-P\Big|^r h(z) dz\Biggr),
	\]
	where $P$ ranges over all analytic trigonometric polynomials with zero constant term. The right side is $0$ if $\log(h)$ is not integrable.
\end{thm}
Clearly, Szeg\"{o}-Kolmogorov-Krein theorem gives an alternative definition to Mahler measure (that is, the Malher measure
of a given measure is the Mahler measure of its derivative). For other definitions, we refer the reader to \cite{degot}.\\


\noindent Following Helson and Szeg\"{o} \cite{HS}, Szeg\"{o}-Kolmogorov-Krein theorem solved the first problem of the theory of prediction. This theorem can be interpreted in the entropy language. We are going to recall the entropy between two positive measures and present this interpretation. For more details, we refer to \cite{Simon}.\\ 

\noindent Let $\mu$ and $\nu$ be two (positive) measures on a Polish space $X$. We define their relative entropy by
$$ \textrm{En}\Big(\mu \mid\mid {\nu}\Big)=\begin{cases}
\ds -\int_X \log\Big(\frac{d\mu}{d\nu}(x)\Big) d\mu(x) & \textrm{if~~} \mu \ll  \nu\\
-\infty & \textrm{if~not}.
\end{cases} .$$
For the entropy interpretation of Szeg\H{o} theorem on the circle for refer to \cite{Simon}. Let us denote $\mathcal{P}(X)$ the set of probability measures on $X$. In this setting, we have
\begin{thm}\label{AN}Let $\sigma_{0,s}=\prod_{k=0}^{+\infty}|P_k(\theta)|^2$ be a spectral type of some rank one flow. Then

	$$M_s(\sigma_{0,s}) \geq \limsup_{K\rightarrow \infty}\int \log\Big(\prod_{k=0}^{K}|P_k(\theta)|^2)\Big) K_s(\theta) d\theta,$$
\end{thm} 

\noindent{}For the proof we need the following lemma.\\
\begin{lemm}\label{semi-continue}The Relative entropy $ E~~:~~\mathcal{P}(X) \times \mathcal{P}(X) \longrightarrow [0, +\infty]$
is a non-positive, convex and upper semicontinuous function with respect to the weak-star topology on $\mathcal{P}(X).$	
\end{lemm}
\begin{proof}[\textbf{Proof of Theorem \ref{AN}}] We start by writing the Lebsegue decomposition of $\sigma_{0,s}$ with respect to $\lambda_s$ as follows
	$$\sigma_{0,s}= f d\lambda_s+\mu_{\textrm{sin}}.$$
Therefore
$$\textrm{En}(d\lambda_s||\sigma_{0,s}) =\int log(f) d\lambda_s.$$
	
Moreover, applying Lemma \ref{semi-continue}, we get 
	$$\textrm{En}(d\lambda_s||\sigma_{0,s}) \geq \limsup_{K\rightarrow \infty}\int \log\Big(\prod_{k=0}^{K}|P_k(\theta)|^2\Big) d\lambda_s.$$
since $\ds \prod_{k=0}^{K}|P_k(\theta)|^2 d\lambda_s$ converge weakly to $\sigma_{0,s}$.
Whence 
$$M_s(\sigma_{0,s}) \geq \limsup_{K\rightarrow \infty}M_s(\prod_{k=0}^{K}|P_k(\theta)|^2).$$
The proof of the theorem is complete.
\end{proof}
\begin{que} We ask whether the following extension of Theorem \ref{AN} is true:  
$$	M_s(\sigma_{0,s})= \lim_{K\rightarrow \infty}M_s\Big(\prod_{k=0}^{K}|P_k(\theta)|^2\Big)?$$
For the case of $\Z$-action this formula holds. Moreover, by appealing to Proposition 2 from the appendix of \cite{Bourgain}, the author in unpublished paper \cite{elabdal-Mahler} proved that the Mahler measure of the spectral type of any rank one maps with the cutting parameter $(m_j)$ satisfying
$m_j=\theta(j^\beta)$, for some $\beta \leq 1$  is zero. We ask also if it is possible to extend this result to the rank one flow.
\end{que}

\section{Banach problem for flow and flat polynomials on real line}

In this section, we strengthen Lemma \ref{Absc-Guenais} by establishing that Banach problem for flow has a positive solution if and only if there is a sequence of flat polynomials on real line. We start by recalling the notion of flatness.\\

\noindent Let $s \in (0,1]$, and $\alpha \in [0,+\infty]$. If $\alpha>0$, the sequence $\big(P_n(t)\big)$ of analytic trigonometric polynomials of $L^2(\R,d\lambda_s)$
norm $1$ is said to be $L^\alpha$-flat if  the sequence $\big(| P_n(t)|\big)$ converges in $L^\alpha$-norm to the constant function $1$ as $n \longrightarrow +\infty$. For $\alpha=+\infty$, the sequence is said to be almost everywhere (a.e) ultraflat. If
 $\alpha=0$,
we say that $(P_n)$ is $L^\alpha$-flat, if the sequence of the Mahler measures $\big(M(P_n)\big)$ converge to $1$. We recall that the Mahler measure of a function $f \in L^1(\R,d\lambda_s)$ is defined by
$$ M(f)=\|f\|_0=\lim_{\beta \longrightarrow 0}\|f\|_{\beta}=\exp\Big(\int_{\R} \log(|f(t)|) d\lambda_s(t)\Big).$$

\noindent The sequence $\big(P_n(t)K_s(t)\big)$ is said to be flat in a.e. sense (almost everywhere sense) if the sequence $\big(| P_n(t)|\big)$, converges a.e. to $1$ with respect to $d\lambda_s$ as $n \longrightarrow +\infty$. Since $\lambda_s$ is equivalent to the Lebesgue measure, the a.e. is stand also  for the Lebesgue measure.\\

\noindent The sequence $\big(P_n(t)K_s(t)\big)$ can be seeing as a sequence of functions in $C_0(\R)$(the subspace of continuous functions which vanish at infinity). We thus say that $\big(P_n(t)\big)$ is ultraflat if there is a sequence of compact subset $K_m \subset \R$, $m \geq 1$ such that 
\begin{enumerate}[(i)]
\item $\ds \bigcup_{m \in \R}K_m=\R$, and
\item $\sup_{t\in K_m} \Big|\big|P_m(t)|-1|\Big| \tend{m}{+\infty}0,$ for each $m \in \N.$
\end{enumerate}


\noindent We further say that a sequence $\big(P_n\big)$ of $L^2$-normalized polynomials is flat in the sense of Littlewood  if there is a sequence of compact subset $K_m \subset \R$, and a constants $0<A_m<B_m,$ $m \geq 1$ such that for all $n \in \N$ (or at least for sufficiently large $n \in \N$), we have
$$A_m \leq \big| P_n(t)\big| \leq B_m, \forall t \in K_m, \forall m \in \N.$$
\\
 Erd\"{o}s and Newman asks on the existence of ultraflat polynomials on the torus with coefficients of modulus one \cite[Problem 22]{ErdosII}, and Littlewood on the existence of flat polynomials with coefficients $\pm 1$ in his sense \cite{Littlewood2}. One may ask the same questions about polynomials on $\R$. But, it is easy to see that if $(P_n)$ is a flat polynomials in any sense on the torus then $(P_n)$ is flat as polynomials on $\R$.   Here, we will first establish that the existence of such polynomials in $L^1$ or a.e. sense implies the existence of rank one flow acting on infinity measure space with Lebesgue spectrum. We thus obtain that the existence of those polynomials implies that Banach problem  has a positive answer in the class of conservative flows. The complete solution of Banach problem will be given in section \ref{Banach-rank1}.\\

\paragraph{\textbf{Generalized Riesz Products of Dynamical Origin.}}
We start by introducing the notion of dissociation. Let $\Gamma$ be a subset of $\R$ and denote by $\textrm{W}(\Gamma)$ the set of all element $w$ of $\R$ of the form
\begin{eqnarray}\label{word}
	w=\sum_{j=0}^{n}\epsilon_j \gamma_j,
\end{eqnarray}
where all $\gamma_j$ are distinct elements of $\Gamma$, $\epsilon_k =\pm 1$. Following \cite{Hewitt-Zuckermann}, the subset $\Gamma$ of $\R$ is a dissociate
set if each element of $\textrm{W}(\Gamma)$ has a unique representation of the form \eqref{word}. For the general definition of dissociate subset of locally compact abelien group, and its connection to the classical Riesz products, we refer to \cite{KS} and the references therein. \\

In our setting, we formulate the dissociation notion as follows.

\begin{defn}\label{def2}  Finitely many trigonometric polynomials on real line  $P_0,P_1,\cdots,P_n$,\linebreak $P_j(\theta)=\sum_{k=-N_j}^{N_j} d_k^{(j)}
	e^{i t_k \theta}$,$j=0,1,2,\cdots,n$ are said to be dissociated if in their product \linebreak $P_0(\theta)P_1(\theta)\cdots P_n(\theta)$, (when expanded formally, i.e., without grouping terms or canceling identical terms with opposite signs), the frequencies $t_{l_0}+t_{l_1}+\cdots+t_{l_n}$ in non-zero terms
	$$d_{l_0}^{(0)}d_{l_1}^{(1)}\cdots d_{l_n}^{(n)}e^{i (t_{l_0}+t_{l_1}+\cdots+t_{l_n}) \theta }$$ are all distinct.\\
	
	A sequence $P_0,P_1,\cdots,$ of trigonometric polynomials on real line is said to be dissociated if for each $n$ the polynomials
	$P_0,P_1,\cdots,P_n$ are dissociated.
\end{defn}

Now, let $s \in ]0,1]$ and $K_s$ be the Fej\'{e}r Kernel as in Theorem \ref{max-spectr-typ}, and
consider the polynomials appearing in the spectral type of some rank one flow. 
$$P_j(\theta) = \frac{1}{\sqrt{m_{k}}}\Big(1+e^{i R_{1,k}\theta } + \cdots + e^{ i R_{m_k-1,k}\theta} \Big).$$

The exponent $R_{j,k}, 1 \leq j \leq m_k-1, j =1,2,\cdots$, is the
$j$-th return time of a point in $\overline{B}_{k,s}$ into $\overline{B}_{k-1,s}$, $0< s \leq 1$. Also
$$R_{j,h} = jh_{k-1} + s_{0,k} + \overline{s}_{1,k} + \cdots +\overline{s}_{j-1,k}, 1 \leq j \leq m_{k}-1$$
where $h_{k-1}$ is the height of the tower after $(k-1)$-th stage of the construction is complete, and $s_{k,l}$ is the number of spacers on the $l$-th column, $ 0 \leq l\leq m_k-2$.
We observe that\\
\begin{enumerate}
	\item  $h_1 = R_{m_{1}-1,1} +1$,\\
	\item  $R_{1,k} \geq h_{k-1} > R_{m_{k-1},k-1}$,\\
	\item   $R_{j+1, k} - R_{j,k} \geq h_{k-1}$. \\
\end{enumerate}
These properties (1), (2), (3) of the powers $R_{j,k}$, $1 \leq j \leq m_k -1, k=1,2,\cdots$ indeed characterize generalized Riesz products which arise from  rank one transformations . More precisely consider a generalized Riesz product
$$\prod_{k=1}^\infty\mid Q_k(\theta)\mid^2. $$
where $$Q_k(\theta) = \frac{1}{\sqrt{m_k}}\sum_{i=0}^{m_k-1} e^{i r_{i,j}\theta}.$$
Define inductively: $$h_0 = 1, h_1 = r_{m_1,1} +h_0, \cdots , h_k = r_{m_k,k} +h_{k-1}, k \geq 2$$
Note that $ h_k > r_{m_k,k}$.\\

We further have for any $\varrho>1$, the $L^2-$ norm of $P_j(\varrho\theta)$ is $1$. Indeed,  by changing the variable of integration, we have
\begin{align}\label{Katz}
	\big\|P_j(\varrho\theta)\|_2&=\int_{\R}\big|P_j(\varrho\theta)|^2 K_s(\theta)d\theta \nonumber\\
	&=\int_{\R}\big|P_j(\theta)\big|^2 \frac{1}{\varrho}K_{s}\Big(\frac{\theta}{\varrho}\Big)d\theta \\
	&=\int_{\R}\big|P_j(\theta)\big|^2 K_{\frac{s}{\varrho}}(\theta)d\theta=1.\nonumber
\end{align}

\begin{Prop}\label{prop-1}
	Assume that for each $k=1,2,\cdots$,
	$$r_{1,k} \geq h_{k-1}, ~~~r_{j+1,k} - r_{j,k} \geq h_{k-1}$$
	Then $r_{j,k}, h_k$, satisfy (1), (2) and (3). The generalized product $\prod_{k=1}^\infty \mid Q_k(\theta)\mid^2$ describes the maximal spectral type (up to possibly a discrete part) of a suitable rank one flow.
\end{Prop}
\begin{proof}
	That the $r_{i,j}, h_j$ satisfy (1), (2), (3) is obvious. The needed rank one flow $(T^t)$ is given by cutting parameters $ p_{k} =m_k, j = 1,2,\cdots$, and spacers $s_{j-1,k} = r_{j,k} - r_{j-1,k}- h_{k-1}$, $1 \leq j \leq m_k-1, k =1,2,\cdots$. This proves the proposition.\\
\end{proof}

\begin{defn}\label{def1}
	A generalized Riesz product $\mu_s = \prod_{k=1}^\infty\big| Q_k(\theta)\big|^2$,
	where $Q_k(\theta) = \frac{1}{\sqrt{m_{k}}}\sum_{j=0}^{m_k}  e^{-2 \pi i r_{j,k}\theta},$ is said to be of dynamical origin if
	with $$h_0 = 1, h_1 = r_{m_1,1} +h_0, \cdots , h_k = r_{m_k,k} +h_{k-1}, k \geq 2$$
	it is true that for  $k=1,2,\cdots$,
	$$r_{1,k} \geq h_{k-1}, ~~~r_{j+1,k} - r_{j,k} \geq h_{k-1}.$$
	\\
\end{defn}

Let us further observe that the generalized Riesz products $\prod_{j=1}^\infty| P_j|^2$ raised from rank one $(T^t)$ have the property that
the sequence of their tails $\mu_{n,s}=\prod_{j=n+1}^\infty| P_j|^2,$ $n=1,2,\cdots$ converges weakly to $\lambda_s = K_s(\theta)d\theta$. In the rest of this section we will assume that the generalized Riesz products have this additional property,
although  is not assumed that they arise from rank one transformations as above.
\begin{defn}\label{def3} Let $s \in ]0,1]$,
	a generalized Riesz product $\mu_s  = \prod_{j=1}^\infty |P_j|^2 $ is said to be of class
	(L) if for each sequence $k_1 < k_2 < \cdots$ of natural numbers, the tail measures
	$\prod_{j=n+1}^\infty\mid{P_{k_j}} \mid^2, n = 1,2,\cdots$ converge weakly to  $\lambda_s$.
\end{defn}
\begin{Prop}\label{prop2} Let $s \in ]0,1]$, 
	if the generalized Riesz product $\mu_s =\prod_{j=1}^\infty\mid P_j\mid^2$ is of class (L) then the partial products $\prod_{j=1}^n\mid P_j\mid, n=1,2,\cdots$
	converge in $L^1(\R, \lambda_s)$ to $\sqrt{\frac{d\mu}{d\lambda_s}}$, and the convergence is almost everywhere (w.r.t $d\theta$) over a subsequence.
\end{Prop}
\begin{proof}
	In Theorem \ref{th6}  we put $Q_j(\theta) =1$ for all $j$, so that $\nu$ is the probability measure on $\lambda_s$. The first conclusion follows from  theorem \ref{th6}. The second conclusion follows since $L^1$ convergence implies convergence a.e over a subsequence and $\lambda_s$ is equivalent to Lebesgue measure.
\end{proof}
\noindent{}The following formula follows immediately from this:
\begin{Cor}\label{cor4} Let $s \in ]0,1]$, and  $\mu$ be a generalized Riesz product of class (L). Let ${{\mathcal K} }_1,{{\mathcal K}}_2$
	be two disjoint subsets of natural numbers and let ${\mathcal K}_0$ be their union.
	Let $\mu_{1,s}, \mu_{2,s}$ and $\mu_{0,s}$ be the generalized Riesz subproducts of $\mu_s$ over
	${\mathcal K}_1, {\mathcal K}_2$, and ${\mathcal K}_0$  respectively. Then we have:
	$$\frac{d\mu_0}{d\lambda_s} = \frac{d\mu_1}{d\lambda_s}\frac{d\mu_2}{d\lambda_s},  \eqno (1)$$
	where equality is a.e. with respect to the measure $d\theta$.
\end{Cor}

\paragraph{\textbf{Flat Polynomials and Generalized Riesz Products.}}

\begin{lemm}\label{lem1}  
	Given a sequence  of trigonometric polynomials  $$P_n(\theta) = \frac{1}{\sqrt{m_n}}\Big( 1+\sum_{j=1}^{m_n-1} e^{it_j\theta}\Big), \,  t_j \geq 0, 1 \leq j \leq m_n-1,\, \, n=1,2,\cdots,$$ then there exist a sequence of positive real numbers $\varrho_1, \varrho_2, \cdots$ such that
	$$\prod_{j=1}^\infty\mid P_j(\varrho_j \theta)\mid^2$$
	is a generalized Riesz product of dynamical origin.
\end{lemm}
\begin{proof}
	For each $j\geq 1$, let
	$$P_j(\theta) = \sum_{k=0}^{m_j-1}b_{k,j}e^{i r_{k,j}\theta}, b_{k,j}=\frac1{\sqrt{\sqrt{m_j}}} \neq 0,  ~~b_{0, j}=1 > 0, ~~\sum_{i=1}^{n_j} \mid b_{i,j}\mid^2 =1.$$
	Let $\varrho_1 =1$ and $h_1 = H_1 = r_{n_1,1}+1$. Choose $\varrho_2 \geq 2H_1  > 2r_{n_1,1}$. Then
	$${\lambda_2\cdot r_{1,2}} > h_1, \varrho_2(r_{i+1,2} - r_{i,2}) > h_1.$$
	
	Since $\varrho_2 > 2r_{n_1,1}$ the   polynomials $\mid P_1(\varrho_1\theta)\mid^2$ and $\mid P_2(\varrho_2\theta)\mid^2$ are dissociated.
	Consider now  $P_1(\varrho_1\theta)P_2(\varrho_2\theta)$. Write $H_2 = \varrho_1r_{n_1,1} + \varrho2r_{n_2,2} +  h_1> \varrho_2r_{n_2, 2} + h_1 \setdef   h_2 $.
	Choose $\varrho_3 \geq 2H_2$.
	Then
	$$\varrho_3\cdot r_{1, 3} \geq h_2, \varrho_3(r_{i+1, 3} -r_{i,3}) > h_2.$$
	Since $\varrho_3 \geq  2 H_2 > 2(\varrho_1r(n_1,1) + \varrho_2r(n_2,2))$ the polynomial $\mid P_3(\varrho_3\theta)\mid^2$ is dissociated from 
	$\mid P_1(\varrho_1\theta)\mid^2$ and $\mid P_2(\varrho_2\theta)\mid^2$.
	Proceeding thus we get $\varrho_j, j =1,2,\cdots$ and  polynomials  $Q_j(\theta) = P_j(\varrho_j\theta), j  =1,2,\cdots$ such that
	\begin{enumerate}[(i)]
\item $\mid\mid Q_j \mid\mid_2 = 1$ (since $\mid\mid P_j\mid\mid_2 = 1$ and $\mid\mid P_j(\varrho_j\theta)\mid\mid_2 = 1$ by \eqref{Katz} .)
\item the polynomials $\mid Q_j\mid^2, j =1,2,\cdots$ are dissociated,
\item for each $j\geq 1$,
$$h_{j-1} < \varrho_{j}r_{1, j}, ~~h_{j-1} < \varrho_{j}(r_{i+1,j} - r_{i, j}) $$
	\end{enumerate}

	Since the polynomials $Q_j, j =1,2,\cdots$ have $L^2(\R,d\lambda_s)$ norm 1 and their absolute squares are dissociated, the generalized Riesz product $\prod_{j=1}^\infty\mid P(\varrho_j \theta)\mid^2$
	is well defined. Moreover, (iii) shows that the conditions for it to arise  from a
	rank one flow $(T^t)$ in the above fashion are satisfied. The lemma follows.
\end{proof}

An immediate application of this Lemma is the following:\\
\begin{thm}\label{th7}
	{\it Let $s \in ]0,1]$ and $P_j, j =1,2,\cdots$ be a sequence of trigonometric polynomials on real such that $\mid P_j(\theta)\mid \rightarrow 1 $ a.e. $(d\theta)$ as $j \rightarrow \infty$. Then there exists a subsequence $P_{j_k}, k=1,2,\cdots$ and real numbers $\varrho_1 < \varrho_2 <  \cdots$ such that for any $s \in ]0,1]$, the product
		$\mu_s =\prod_{k=1}^\infty \mid P_{j_k}(\varrho_k \theta)\mid^2$  is a generalized Riesz product of dynamical origin with $\frac{d\mu_s}{d\theta} > 0$ a.e. $(d\theta)$. }
\end{thm}
\begin{proof}
	Since $\mid P_j(\theta)\mid \rightarrow 1$ as $j \rightarrow \infty$ a.e. $(d\theta)$, by Egorov's theorem we can extract a subsequence $P_{j_k}, k = 1,2,\cdots$ such that
	for any $s \in ]0,1]$, the sets $$E_k \setdef \Big\{\theta: \mid (1- \mid P_{j_l}(\theta)\mid )\mid < \frac{1}{2^k} ~~\forall ~~l \geq k\Big\} $$
	increase to $\R$ (except for a $\lambda_s \setdef K_s d\theta$ null set, $s \in ]0,1]$), and $\sum_{k=1}^\infty (1 -\lambda_s (E_k)) < \infty.$
	Write $Q_k = P_{j_k}$. Then for $\theta \in E_k$, $\big| 1 - \mid Q_k(\theta)\mid\big| < \frac{1}{2^k}$.  By the lemma above we can choose $\varrho_1, \varrho_2, \cdots$ such that
	$$\prod_{k=1}^\infty \mid Q_k(\varrho_k \theta) \mid^2$$ is a generalized Riesz product of dynamical origin. We show that $\lim_{L\rightarrow \infty}\prod_{k=1}^L\mid Q_k(\varrho_k\theta)\mid$
	is nonzero a.e. $(d\theta)$, which will imply, by proposition \ref{prop2}, that $\frac{d\mu_s}{d\theta} > 0$ a.e $(d\theta)$.\\
	
	Now the maps $S_{\varrho}: \theta \rightarrow \varrho \theta, \varrho>0$, preserve the  Fej\'er kernel family, and since $\sum_{k=1}^\infty \lambda_s(\R \setminus E_k) < \infty$ 
	we have $\sum_{k=1}^\infty \lambda_s(S_{\rho_k}^{-1}(\R \setminus E_k)) < \infty$, for all $s \in ]0,1]$. Let $F_k = S_{\varrho_k}^{-1}(\R \setminus E_k)$ and
	$F = \limsup_{k\rightarrow \infty} F_k = \ds \bigcap_{k=1}^\infty\bigcup_{l=k}^\infty F_l$.
	Then $\lambda_s(F) =0$, and if $\theta \notin F$, $\theta \notin S_{\frac{1}{\rho_k}}(\R \setminus E_k)$ hold for all but
	finitely many $k$, which in turn implies that $S_{\varrho_k}\theta \in E_k$ for all but finitely many $k$.
	Thus, if $\theta \notin F$, then $\mid 1 - \mid Q_k(\varrho_k \theta)\mid\mid < \frac{1}{2^k}$ for all but finitely many $k$. Also the set of points $\theta$ for which some finite product $\prod_{k=1}^L\mid Q_k(\varrho_k\theta)\mid$ vanishes is countable. Clearly
	$$\ds \lim_{L\rightarrow \infty}\prod_{k=1}^L\mid Q_k(\varrho_k \theta)\mid$$
	is nonzero a.e. $(d\lambda_s )$ and the theorem is proved.\\
\end{proof}

\begin{Cor}\label{cor1}$\quad$
	
	\begin{enumerate}[(i)]
	\item  If $P_k,~~ k=1,2,\cdots$  are as in the above theorem and if $\ds \limsup_{k\rightarrow \infty} M_s(P_k) =1$, for all $s \in ]0,1]$, then we can choose $P_{j_k},~~ k=1,2,\cdots$ and  $\varrho_1, \varrho_2, \cdots $ in such a way that $M_s(\mu_s) $ is positive, for all $s \in ]0,1]$.
	\item  If $P_k,~~ k =1,2,\cdots$ are as in the above theorem  and  if $\ds \liminf_{k\rightarrow \infty}M_s(P_k) < 1$, for all $s \in ]0,1]$,  then we can choose $P_{j_k},~~ k =1,2,\cdots$ and $\varrho_1, \varrho_2, \cdots $ in such a way that $M_s(\mu_s) =0$, and $\frac{d\mu_s}{d\theta} > 0$ a.e. $(d\theta)$.
	\end{enumerate}
	
\end{Cor}

\begin{rem}\label{rem2}$\quad$
	
	\begin{enumerate}
	\item Now it is easy to construct polynomials $P_k, ~~k=1,2,\cdots$ satisfying the hypothesis of
	part (ii) of the above corollary, so one can obtain generalized Riesz products $\mu_s$ with
	zero Mahler measure and $\frac{d\mu}{d\lambda_s}$ positive a.e $(d\theta)$.\\ 
      \item In our presentation, we use the notion of dissociation but since the coefficients of our polynomials are positives this assumption can be dropped (see Lemma \ref{PWC} ).\\
	\item The results of this section can be straightforward generalized to the non-singular rank on flow  and its $\Z_2=\{\pm 1\}$ extension. The details are left to the reader.  
\end{enumerate}
	
	

\end{rem}
 
\section{The spectral singularity Theorem for Ornstein mixing flows}\label{Orn-sin}
In this section, we extend Ornstein random construction of rank one transformations to flows. We extend Bourgain theorem in this setting by showing that the spectrum of Ornstein flow is almost surely singular. Our method, following ideas introduced in \cite{elabdaletds}, uses the Central Limit Theorem.

In the classical Orsntein construction the spacers are chosen randomly and independently using the uniform distribution. In the present article, we will stay in this setting. Let us notice however, as it has been pointed out by J. Bourgain in \cite{Bourgain}, that other random choices are possible : in \cite{elabdalihp} the authors shows that Bourgain singularity theorem holds for some more general random constructions.

Let us now describe the classical Ornstein construction, adapted to flows.
We fix a sequence $(p_k)_{k\geq0}$ of integers $\geq2$ and a sequence $(t_k)_{k\geq 0}$ and of positive real numbers.

Put
$$\Omega=\prod_{k=0}^{+\infty}\left[-\frac{t_k}2,\frac{t_k}2\right]^{p_k-1}$$
equipped with the probability product measure $\ds \P:=\bigotimes_{k \geq0}\bigotimes_{1}^{p_k-1}\mathcal{U}_k$, 
where $\mathcal{U}_k$ is uniform measure on $[-\frac{t_k}2,\frac{t_k}2]$. We denote by 
$x_{k,j}$, ($j=1,\ldots , p_k-1$ and $k\geq0$) the projection maps. It follows that $x_{k,j}$ are 
independent random variables and that each $x_{k,j}$ is uniformly distributed in $[-\frac{t_k}2,\frac{t_k}2]$. We define the spacers in the following way
$$s_{k+1,j}=t_k+x_{k,j}-x_{k,j-1},\quad j=1,\ldots,p_k,$$ where
$x_{k,0}=0$ and the $x_{k,p_k}$ are chosen deterministically in $\R_{+}^{*}$. (Note that in this random construction the sequence of heights $(h_k)$ stays deterministic.) In order to insure that we get a finite measure preserving system, we impose the following condition on the parameters of the construction
\begin{equation}\label{fin-meas}
\sum _k \frac1{p_0p_1...p_k} (p_kt_{k} + x_{k,p_k})<  +\infty.
\end{equation}

\medskip

We get a family of probability preserving dynamical rank one 
flows denoted by\linebreak $\left(X,\A,\nu,(T_{t}^{\omega})_{\omega \in \Omega}\right)$. Following notation introduced in Subsection \ref{spect-interp}, we have
$\bar s_k(j)=\sum_{i=1}^j s_{k+1,i}=jt_k+x_{k,j}$.  According to Theorem~\ref{max-spectr-typ} the spectral type of each 
$T_{t}^\omega$, up to discrete part, is given by the weak limit
\begin{equation}\label{encore!}
d\sigma^{(\omega)}(\theta) =W-\lim \prod_{k=0}^n\left| P_k^\omega(\theta)\right|
^2d\theta,
\end{equation}   
where 
\begin{equation}\label{repete}
P_k^\omega (\theta)=\frac1{\sqrt{p_k}}\sum_{j=0}^{p_k-1} e^{i\theta(j(h_k+t_k)+x_{k,j}(\omega))}.
\end{equation} 

More precisely, the continuous part of the maximal spectral type is equivalent to the continuous part of $\sigma^\omega$.

\begin{thm}\label{orn-bourgain}
	Let $(p_k), (x_{k,p_k})$ and $(t_k)$ be a choice of parameters for the Ornstein construction of a rank one flow. Suppose that condition (\ref{fin-meas}) is satisfied and that there exists an infinite set of positive integers along which the sequences $(p_k)$ and $(t_k)$ go to infinity.
	
	Then the
	Ornstein flow has almost surely
	singular spectrum, i.e.
	\[
	\P\{\omega~:~ \sigma^{(\omega)} ~\bot ~\text{Lebesgue measure on $\R$}\}=1.
	\]
	
\end{thm}

We denote by $\sigma^{(\omega)}_s$ the measure with density $K_s$ with respect to $\sigma^\omega$. On each compact subset of $\R$, for all $s$ large enough the measures $\sigma^{(\omega)}_s$ and  $\sigma^\omega$ are equivalent. In order to prove Theorem \ref{orn-bourgain}, it is sufficient to prove that, for each $s\in(0,1)$, the measure $\sigma^{(\omega)}_s$ is almost surely singular. This will be done in the sequel of this section. The proof is organized as in the case of $\Z$ rank one actions, using the Central Limit Theorem as it appears in \cite{elabdaletds}. \\

We suppose in the rest of this section that hypothesis of Theorem \ref{orn-bourgain} are satisfied. In particular we denote by $D$ an infinite set of integers such that
$$
\lim_{\overset{k\to +\infty}{k\in D}} p_k=\lim_{\overset{k\to +\infty}{k\in D}} t_k=+\infty.
$$

We denote by $\lambda_s$ the measure with density $K_s$ with respect to Lebesgue measure. As we know, 
\[
d\sigma_s^{(\omega)}(\theta) =W-\lim \prod_{k=0}^n\left| P_k^\omega(\theta)\right|
^2d\lambda_s(\theta),
\]   
\begin{prop}\label{w-lim}
	There exists a subsequence of
	the sequence $\left ((\theta,\omega)\mapsto\left |\left| P_k^\omega(\theta)\right|-1 \right|\right)_{k\in D}$
	which converges weakly in $L^2\left( \mu_s\otimes\P \right)$ to some non-negative function
	$\phi(\theta,\omega)$ 
	below by a universal positive constant $c$.
\end{prop}

\begin{proof}{}Since for all $k$ and all $\omega$, we have $\left\|\P_k^\omega\right\|_{L^2(\mu_s)}=1$, the sequence $\left(\left|\left|P_k\right|-1\right|\right)_{k\in D}$ is bounded in $L^2\left( \mu_s\otimes\P\right)$, thus admits a weakly convergent subsequence. Let us denote by $\phi$ one such weak limit function.
	
	Let us prove that the function $\phi$ is bounded by below by a universal positive constant.
	
	\begin{lemm}\label{centrage}
		We have 
		$$
		\lim_{\overset{k\to +\infty}{k\in D}} \int_\R \left| \int_\Omega P_k^\omega(\theta)\,d\P(\omega)\right|\;d\lambda_s(\theta)=0.
		$$
	\end{lemm}
	
	As a direct consequence of this lemma, if we define $P_k'(\theta)=P_k(\theta)-\int_\Omega P_k^\omega(\theta)\,d\P(\omega)$, we have
	\begin{equation}\label{lem44hetds}
	\lim_{\overset{k\to +\infty}{k\in D}} \int\!\!\!\int \left|\left|P_k\right|-\left|P_k'\right|\right|\,d\P d\lambda_s =0.
	\end{equation}
	
	\begin{proof}[Proof of Lemma \ref{centrage}] We come back to the expression of the polynomial $P_k$.
		\begin{multline*}
		\int_\Omega P_k^\omega(\theta)\,d\P(\omega) = \frac1{\sqrt{p_k}}\sum_{j=0}^{p_k-1} e^{i j(h_k+t_k)\theta}\int_\Omega e^{i\theta x_{k,j}(\omega)}\,d\P(\omega)\\
		= \frac1{\sqrt{p_k}}\sum_{j=0}^{p_k-1} e^{i j(h_k+t_k)\theta}\frac1{t_k}\int_{-t_k/2}^{t_k/2} e^{i\theta x}\,dx= \frac1{\sqrt{p_k}}\sum_{j=0}^{p_k-1} e^{i j(h_k+t_k)\theta}\frac{\sin(t_k\theta/2)}{t_k\theta/2}.
		\end{multline*}
		Using Cauchy-Schwarz inequality and ${\ds\int_\R} \left| \frac1{\sqrt{p_k}}\sum_{j=0}^{p_k-1} e^{i j(h_k+t_k)\theta}\right|^2\,d\lambda_s(\theta)=1$, we obtain
		$$
		\int_\R \left| \int_\Omega P_k^\omega(\theta)\,d\P(\omega)\right|\;d\lambda_s(\theta) \leq \left(\int_\R \left(\frac{\sin(t_k\theta/2)}{t_k\theta/2}\right)^2\;d\mu_s(\theta)\right)^{\frac12}
		$$
		Remind that we assume that the sequence $(t_k)$ goes to infinity. The function under the integral in the last expression goes to zero when $k$ goes to infinity, and it is uniformly bounded by 1. Thus the whole integral goes to zero.
	\end{proof}

	We will use the following version of the Central Limit Theorem. It is a particular case of the classical 
	Lindeberg-Feller Theorem (see for example \cite{Durrett}, Chapter 2, Section 4.b).
	
	\begin{prop}\label{CeLiTh}
		Let $(Y_{k,j}; 0\leq j< p_k, k\geq1)$ be an array of real random variables, with $p_k\to+\infty$ when $k\to+\infty$. Suppose that
		\begin{itemize}\item[(i)] The random variables are centered : for all $k,j$, $\E(Y_{k,j})=0$;
			\item[(ii)] The random variables are uniformly bounded : $\sup_{k,j,\omega} \left|Y_{k,j}(\omega)\right|<+\infty$;
			\item[(iii)] For each $k\geq1$, the random variables $Y_{k,j}, 0\leq j< p_k,$ are independent;
			\item[(iv)]$ \lim_{k\to+\infty} \frac1{p_k}\sum_{j=0}^{p_k-1}\E\left(Y_{k,j}^2\right)=\sigma^2\in[0,+\infty)$
		\end{itemize}
		Then the sequence 
		$$
		\frac1{\sqrt{p_k}}\sum_{j=0}^{p_k-1} Y_{k,j} 
		$$
		converges in law to the normal distribution $\mathcal N(0,\sigma^2)$ when $k$ goes to infinity.
	\end{prop}
	
	(If $\sigma=0$, this convergence is a convergence in probability to zero.)
	
	\begin{Cor}\label{conv-gauss}
		For all $\theta\neq0$, the sequence of random variables $\left(\text{Re}P_{k}'(\theta)\right)_{k\in D}$ converges in law to $\mathcal N(0,\frac12)$.
	\end{Cor}

	\begin{proof} We fix $\theta\neq0$.
		We define
		$$Y_{k,j}(\omega):=\cos\left(\theta(j(h_k+t_k)+x_{k,j}(\omega))\right)-\E\left(\cos\left(\theta(j(h_k+t_k)+x_{k,j}(\omega))\right)\right).
		$$
		We have 
		$$
		\text{Re}P_k'=\frac1{\sqrt{p_k}}\sum_{j=0}^{p_k-1} Y_{k,j} .
		$$
		Conditions (i), (ii) and (iii) of Proposition \ref{CeLiTh} are satisfied. Let us verify the last one.
		
		\begin{multline*}
		\E\left(Y_{k,j}^2\right)\\=\E\left(\left(\cos\left(\theta(j(h_k+t_k)+x_{k,j}(\omega))\right)\right)^2\right)-\left(\E\left(\cos\left(\theta(j(h_k+t_k)+x_{k,j}(\omega))\right)\right)\right)^2
		\\=\frac1{t_k}\int_{-t_k/2}^{t_k/2}\cos^2(\theta(j(h_k+t_k)+x))\;dx-\left(\frac1{t_k}\int_{-t_k/2}^{t_k/2}\cos(\theta(j(h_k+t_k)+x))\;dx\right)^2\\=
		\frac12\left(1 + \cos\left(2\theta j(h_k+t_k)\right)\frac{\sin(t_k\theta)}{t_k\theta}\right)-\left(\cos\left(\theta j(h_k+t_k)\right)\frac{\sin(t_k\theta/2)}{t_k\theta/2}\right)^2
		\end{multline*}
		If $t_k\to+\infty$, we see that $\E\left(Y_{k,j}^2\right)\to \frac12$, hence condition (iv) is satisfied with $\sigma^2=\frac12$.
	\end{proof}
	
	Coming back to the proof of Proposition \ref{w-lim}, we consider a weak limit $\phi$ in $L^2(\mu_s\otimes\P)$ of the sequence $\left(||P_k|-1|\right)_{k\in D}$. Let $A$ be a measurable subset of $\R$ and let $C$ be a cylinder subset of $\Omega$. By \emph{cylinder subset} we mean a set of $\omega$'s defined by a condition depending only of finitely many coordinates of $\omega\in\Omega=\prod_{k=0}^{+\infty}\left[-\frac{t_k}2,\frac{t_k}2\right]^{p_k-1}$. Note that, for all large enough $k$, for all $\theta$, the set $C$ is independent of the random variable $P_k^\cdot(\theta)$.
	We have
	\begin{multline*}
	\int_{A\times C}||P_k^\omega(\theta)|-1|\;d\mu_s(\theta)d\P(\omega)=\P(C) \int_{A\times \Omega}||P_k|-1|\;d\lambda_sd\P\\\geq
	\P(C)\int_{A\times \Omega}||P'_k|-1|\;d\mu_sd\P-\int_{\R\times\Omega} \left|\left|P_k\right|-\left|P_k'\right|\right|\,d\lambda_sd\P.
	\end{multline*}
	On one hand, we have
	$$
	\int_{A\times \Omega}||P'_k|-1|\;d\mu_sd\P\geq\int_A\P\left(\text{Re}P'_k>2\right)\;d\mu_s,
	$$
	and, by Corollary \ref{conv-gauss},
	$$
	\lim \P\left(\text{Re}P'_k>2\right)=\frac1{\sqrt\pi}\int_2^{+\infty} e^{-t^2}\;dt =:c,
	$$
	thus, by dominated convergence,
	$$
	\lim \int_A\P\left(\text{Re}P'_k>2\right)\;d\lambda_s(\theta) = c\lambda_s(A).
	$$
	On the other hand, by Lemma \ref{centrage}
	$$
	\lim\int_{\R\times\Omega} \left|\left|P_k\right|-\left|P_k'\right|\right|\,d\lambda_sd\P=0.
	$$
	We conclude that
	$$
	\liminf \int_{A\times C}||P_k^\omega(\theta)|-1|\;d\lambda_s(\theta)d\P(\omega)\geq c\P(C)\lambda_s(A),
	$$
	which proves that the weak limit $\phi$ satisfies $\int_{A\times C}\phi\;d\lambda_sd\P\geq c\P(C)\lambda_s(A)$.
	
	Since this is true for all choices of $A$ and $C$, the function $\phi$ is (a.s.) bounded by below by $c$. Proposition \ref{w-lim} is proved.
\end{proof}

Let us state now a randomized version of Lemma \ref{limsup}. Its proof is straightforward by integration with respect to $\P$ of the inequality which follows the statement of Lemma \ref{limsup}.

\begin{lemm}\label{limsuprand} Let $E$ be an infinite set of positive integers. Let $L$ be a positive integer and  $
	0\leq n_1<n_2<\cdots<n_L$ be integers. Denote $Q(\theta,\omega)=\prod_{\ell=1}^L\left | {P_{n_\ell}^\omega(\theta)}\right|$. Then
	\begin{eqnarray*}
		\displaystyle \limsup_{\overset{k\longrightarrow +\infty}{k\in E}}\int\!\!\!\int  Q \left| P_k\right|\;d\lambda_sd\P
		\leq 
		\displaystyle \int\!\!\!\int Q \;d\mu_sd\P -\frac 18\left(
		\liminf_{\overset{k\longrightarrow +\infty}{k\in E}}
		\displaystyle \int\!\!\!\int  Q \left| \left| P_k\right| ^2-1\right|
		\;d\lambda_sd\P\right) ^2. 
	\end{eqnarray*}
\end{lemm}

\begin{proof}[Proof of Theorem \ref{orn-bourgain}]
	Applying Proposition \ref{w-lim}, we fix an infinite set of integers $E$ such that the sequence $\left (\left |\left| P_k\right|-1 \right|\right)_{k\in E}$ converges weakly to a function bounded by below by $c$. We use notations of Lemma \ref{limsuprand}. Since we have of course
	$$
	\int\!\!\!\int  Q \left| \left| P_k\right| ^2-1\right|
	\;d\lambda_sd\P \geq \int\!\!\!\int  Q \left| \left| P_k\right| -1\right|
	\;d\lambda_sd\P,
	$$
	we know that
	$$
	\liminf_{\overset{k\longrightarrow +\infty}{k\in E}}
	\displaystyle \int\!\!\!\int  Q \left| \left| P_k\right| ^2-1\right|
	\;d\lambda_sd\P\geq c\int\!\!\!\int  Q 
	\;d\lambda_sd\P
	$$
	Thanks to Lemma \ref{limsuprand}, we can construct by induction an increasing sequence $(n_\ell)_{\ell\geq1}$ such that, for all $L\geq1$,
	$$
	\int\!\!\!\int  \prod_{\ell=1}^{L+1}\left| P_{n_\ell}\right|
	\;d\lambda_sd\P\leq\int\!\!\!\int  \prod_{\ell=1}^{L}\left| P_{n_\ell}\right|
	\;d\lambda_sd\P-\frac{c^2}9\left(\int\!\!\!\int  \prod_{\ell=1}^{L}\left| P_{n_\ell}\right|
	\;d\lambda_sd\P\right)^2.
	$$
	This inequation gives us the existence and the value of the following limit:
	$$
	\lim_{L\to+\infty}\int\!\!\!\int  \prod_{\ell=1}^{L}\left| P_{n_\ell}\right|
	\;d\lambda_sd\P=0.
	$$
	Almost surely (for the probability measure $\P$), along a subsequence of $L$'s, we have
	$$
	\int  \prod_{\ell=1}^{L}\left| P_{n_\ell}\right|
	\;d\lambda_s\longrightarrow0.
	$$

	By Bourgain criterion (Theorem \ref{Bourg-cri}) we conclude that, for almost all $\omega$, the measure $\sigma^{(\omega)}_s$ is singular. This is what had to be proved.
\end{proof}
\begin{xrem}The results above can extended to establish that for almost all choose of $(\omega,\omega')$ in the Ornstein probability space of spacers, the spectral type of $T^t_{(p_n),\omega}$ and  $T^t_{(p_n),\omega}$ are singular, that is, 
 $T^t_{p,\omega}$ and  $T^t_{p,\omega}$ are spectrally disjoint. The details of the proof is left to the reader.
\end{xrem}
\section{Klemes-Reinhold's theorem for rank one flows.}\label{Klemes1}
In this section we present the extension of Klemes-Reinhold's theorem to the rank one flow. This theorem, for a rank one map, assert that if the sequence square of the cutting parameter is not summable then the spectrum is singular. The principal ingredients of the proof is based on the Bourgain singularity tools, as presented in subsection \ref{Bourgain-sin}, combined with the extension of Peyri\`ere criterion to the generalized Riesz products to $\R$. 
This later extension can be obtained easily. But, for sake of completeness, we will present its proof. Let us state  Klemes-Reinhold's theorem to the rank one flow.
\begin{thm}\label{Klemes-Reinhold} If $\ds \sum_{n \geq 1}\frac{1}{p_n^2}=+\infty$, then the spectral type of the rank one flow $\sigma$ is singular.
\end{thm}
\noindent{}We begin by proving the extension of Peyri\`ere criterion to the generalized Riesz products to $\R$.

\begin{lemm}[of Peyri\`ere]\label{Peyriere} Let $s \in (0,1]$ and $\mu$ a Borel probability measure on $\R$. If
	there exists a sequence of reals numbers $\{\xi_k\}_{k \geq 1 }$ such that
\begin{enumerate}[label=(\roman*)]
\item $|\xi_j-\xi_k| > s$, for any $j \neq k \in \N^*$.
\item \label{Ortho-ii}$\big(e^{2\pi i  \xi_k t}-\widehat{\mu}(\xi_k)\big)_{k \geq 1}$ is an orthogonal system in $L^2(\mu)$ and $\big(\widehat{\mu}(\xi_k)\big)_{k \geq 1}
\not \in \ell^2.$
\item $\widehat{\mu}\big(\xi_j-\xi_k\big)=\widehat{\mu}(\xi_j)\widehat{\mu}(\xi_k),$
for all $j \neq k \in \N^*$.
\end{enumerate}
Then, $\mu \perp \lambda_s.$
\end{lemm}
\begin{proof}Since the support of $\widehat{K_s}$ is $[-s,s]$ and by our assumption $\xi_j-\xi_k \not \in [-s,s],$ for all $j \neq k$, it follows that the family $\big(e^{2\pi i \xi_k t}\big)_{k \in \N^*}$ is a bounded orthogonal system in
$L^2(\lambda_s)$, We further have $\big\|e^{2\pi i \xi_k}-\widehat{\mu}(\xi_k)\big\|_{L^2(\mu)}$ is bounded since $\mu$ is a probability measure. By our assumption \ref{Ortho-ii}, we have also that $\big(\widehat{\mu}(\xi_k)\big)_{k \geq 1}
\not \in \ell^2.$ Therefore, by Banach-Steinhauss theorem, there is a sequence $(c_k) \in \ell^2$ such that 
\begin{equation}\label{diff-BS}
\sum_{k=1}^{+\infty}c_k\overline{\widehat{\mu}(\xi_k)} =\infty.
\end{equation}

Now, consider the sequences of functions 
\begin{align*}
f_n(t)=\sum_{k=1}^{n}c_k e^{2\pi i \xi_k t}, \textrm{~~~~and}\\
g_n(t)=\sum_{k=1}^{n} (e^{2\pi i \xi_k t}-\widehat{\mu}(\xi_k)).
\end{align*}
Obviously $f_n$ and $g_n$  converge in $L^2(\lambda_s)$ and $L^2(\mu)$ respectively. We can thus extract a subsequence $(n_j)$ such that $f_{n_j}(t)$ and $g_{n_j}(t)$ converge a.e. respectively with respect to $\lambda_s$ and $\mu$. We thus deduce that both series cannot converge for the same $t$ because their difference diverge by \eqref{diff-BS}. Hence, the set $E$ on which the first series converges is a Borel set such that
$\lambda_s(E^c) = 0$ and$ \mu(E) = 0$, which ends the proof. 
\end{proof}
\noindent{}We need also the following lemma.
\begin{lemm}[of the weak convergence]\label{PWC}Let $P_n(t)=\sum_{k=1}^{n}a_k e^{2\pi i \xi_k t},$ $n=1,2,\cdots,$ be a family of trigonometric polynomials on $\R$
with positive coefficients, and $s \in (0,1]$. Let $d\rho_n=\prod_{k=1}^{n}\big|P_n(t)\big|^2 d\lambda_s,$
 and suppose that $\big\|\prod_{j=1}^{n}P_j(t)\big\|_{L^2(\lambda_s)}= 1,$ and 
 $\big\|P_n(t)\big\|_{L^2(\lambda_s)}= 1,$ for all $n \in \N^*$. Then $(\rho_n)$ 
 converge in the weak topology.
\end{lemm}
\begin{proof}By our assumption, we have for any $n$, $\rho_n$ is a probability measure on $\R$. Furthermore, for any $\xi \in \R$, we have 
\begin{eqnarray*}
\widehat{\rho_{n+1}}(\xi)&=&\int e^{-2\pi i \xi_j t} \prod_{j=1}^{n}|P_j(t)|^2 .|P_{n+1}(t)|^2 d\lambda_s\\
&=& \int e^{-2\pi i \xi_j t} \prod_{j=1}^{n}|P_j(t)|^2 .(1+R_{n+1}(t)) d\lambda_s,
\end{eqnarray*}
where $R_{n+1}(t)=\sum_{1 \leq j \neq k \leq n}a_j a_ke^{2\pi i (\xi_j-\xi_k) t}$. We thus get 
$$\widehat{\rho_{n+1}}(\xi)=\widehat{\rho_{n}}(\xi)+\sum_{1 \leq j \neq k \leq n}a_j a_k
\widehat{\lambda_s}(\xi_j-\xi_k).$$
Whence, $\widehat{\rho_{n+1}}(\xi) \geq \widehat{\rho_{n}}(\xi)$ since the coefficients $a_i a_j$ are positive and $\widehat{\lambda_s}(t) \geq 0$. We thus conclude that 
$\lim \rho_n$ exists in the weak topology.
\end{proof}
According to Bourgain criterion, we need to construct a subseqence $(n_k)$ for which the sequence of probability measures $\rho_k=\prod_{j=1}^{k}|P_{n_j}|^2 d\lambda_s, k \in \N^*$ converge to a probability measure $\alpha$, by Lemma \ref{PWC}. By construction, $\alpha$ will satisfy the condition $(ii)$ and $(iii)$ of Pyri\`ere Lemma. We start by putting
$$R_n(t)=|P_n(t)|^2,$$ and $$Q_k(t)=\prod_{j=1}^{k}|P_{n_j}(t)|^2.$$
By construction, the dynamic of the rank one flow is obtained at each stage $n$ by cutting in subcolumns of the tower $\overline{B}_n$, we further have that the height of $j$th subclumn is  $c_j=h_n+s_{n,j}$ with $j=1,\cdots, p_n-1$ and $c_0=0$. Put 
$d_j=\sum_{k=0}^{j}c_k$. Then, we have 
$$ P_n(t)=\frac{1}{p_n}\sum_{j=0}^{p_n-1}e^{2 \pi i d_j t}.$$
\begin{eqnarray}\label{Formula-for-R}
R_n(t)=1+\frac{1}{p_n}\sum_{\xi \in \Gamma_n}e^{2 \pi i \xi t}
\end{eqnarray}
where $\Gamma_n= \big\{d_j-d_k / k \neq j=0,\cdots,p_n-1\big\}$. Let $d_n=\max\{|\gamma|, \gamma \in \Gamma_n\}$. Then $d_n=h_{n+1}-h_n-s_{n,p_n}<h_{n+1}$. We further have
$h_n \leq \frac{h_{n+1}}{p_n} \leq \frac{h_{n+1}}{2}.$\\

Now, let $n_1<n_2<\cdots<n_k$ such that $n_{j+1} \geq n_j+3$. It follows that 
$$Q_k(t)=1+\frac{1}{p_{n_1}\cdots p_{n_k}}\sum_{\xi \in S_k}e^{2 \pi i \xi t },$$
where $S_k=\big\{\xi_1+\cdots+\xi_k| \xi_i \in \Gamma_{n_i}, i=1,\cdots,k\big\}.$ We thus get, by the triangle inequality, 
\begin{eqnarray}\label{tel}
q_k=\max\big\{|x| | x \in S_k\big\} &\leq& d_{n_1} +\cdots+d_{n_k} \nonumber \\
&\leq& h_{n_1+1}-h_{n_1}+h_{n_2+1}-h_{n_2}+\cdots+h_{n_k+1}-h_{n_k} \nonumber\\
&<& h_{n_k+1},
\end{eqnarray}
since, for any $j=1,\cdots,k$, $n_{j+1} \geq n_j+3$ and  
$h_{n_j+1} \leq \frac{h_{n_{j+1}}}{4}.$ Whence, \eqref{tel} follows by telescoping. We will now  summarize
some proprieties of Fourier coefficients of $R_{n_k}$ and $Q_k$ with respect to $L^2(\lambda_s)$, $s \in (0,1]$, as follows.
\begin{Prop} \label{Key-Prop}Let $n_1<n_2<\cdots<n_k$ such that $n_{j+1} \geq n_j+3,$ for each $j=1,\cdots,k$, and $s \in (0,1]$. Then, the Fourier coefficients of $\pi_k=R_{n_k}d\lambda_s$ and $\tau_k=Q_kd\lambda_s$ satisfy
\begin{enumerate}[label=(\alph*)]
\item \label{a} $\widehat{\pi_k}(0)=1$, and $\widehat{\pi_k}(t)=0$, if $s<|t|<h_{n_k}-s.$ 
\item \label{b}$\widehat{\pi_k}(d_{n_k})=\frac{1}{p_{n_k}},$  and  $\widehat{\pi_k}(t)=0,$
if, $d_{n_k}-h_{n_k}+s<|t|<d_{n_k}-s.$
\item \label{c} $\widehat{\tau_{m+k}}(t)=\widehat{\tau_{k}}(t)$, whenever $|t|<q_k$, $m \geq 0$.
\item\label{d} $\widehat{\tau_{k}}(0)=1$ and $\widehat{\tau_{k}}(d_{n_k})=\frac{1}{p_{n_k}}.$
\end{enumerate}
\end{Prop}
\begin{proof}By definition of $\pi_k$,  as in the proof of Lemma \ref{Key1}, for any $t \in \R$ we have,
	$$\widehat{\pi_k}(t)=\widehat{K_s}(t)+\frac{1}{p_{n_k}}\sum_{\gamma \in \Gamma}\widehat{K_s}(t-\gamma),$$
Therefore, $\widehat{\pi_k}(0)=1$ and for $s<|t|<h_{n_k}-s$, the first term is zero. The rest of \ref{a}
follows from the fact that for $i > j$, $d_{i,j} \geq  c_i \geq h_{n_k} $ , and by symmetry, $|d_{i,j}| \geq  h_{n_k}$
for $i < j$ also. Indeed, assume $i > j$ and $t \geq 0$. Then, $d_{i,j}-t \geq h_{n_k}-t >s$, and by symmetry, we conclude
that $t-\gamma \not \in \Gamma,$ for any  $s<|t|<h_{n_k}-s,$ and $\gamma \in \Gamma.$ For the proof of the first part of \ref{b} , it suffices to observe that $d_{n_k}=\max(d_{i,j}).$ For the second part, suppose $i>j$ and $(i,j) \neq (p_{n_k}-1,0)$. Then $d_{i,j}=\sum_{l=j+1}^{i}c_l$ is a sum over a proper subset $I$ of the indexes 
$\{1,2,\cdots, p_{n_k}-1\}$. Whence $d_{n_k}-d_{i,j} \geq \min\{c_1,c_2,\cdots, c_{p_{n_k}-1}\} \geq h_{n_k},$ since $d_{n_k}-d_{i,j}$ is the sum over the complement of $I$. We thus get for $t \in [d_{n_k}-h_{n_k}+s,d_{n_k}-s ],$ $t-\gamma >s$, for all $\gamma \in \Gamma,$ and by symmetry for $t < 0.$  
The first part of \ref{d} follows from Lemma \ref{riesz-prop}. For the rest of \ref{d} and the proof of \ref{c}. Consider $Q_{k+1}=Q_kR_{n_{k+1}}$. Then,
\begin{align}\label{eqkey-1}
	\widehat{\tau_{k+1}}(t+d_{n_k})=\widehat{\tau_{k+1}}(t)\frac{1}{p_{n_k}},~~~~~~~~~~~~\forall t \in [-q_k,q_k].
\end{align}
Indeed,
\begin{align}\label{eqkey-2}
\widehat{\tau_{k+1}}(t+d_{n_{k+1}})&=\widehat{\pi_k}(t+d_{n_{k+1}})+\frac{1}{p_{n_{k+1}}}\sum_{\xi \in S_k}\widehat{\pi_{k+1}}(t+d_{n_{k+1}}-\xi),\nonumber\\
&=\frac{1}{p_{n_k}}\widehat{\tau_{k}}(t),
\end{align}	
by \eqref{tel} and \ref{a}. We thus get for $t=0$ the second part of \ref{d}. By the same reasoning, we get \ref{c}. This finish the proof of the proposition.
\end{proof}
Now, we are going to construct, for any $s \in ]0,1]$, a generalized Riesz product  $\nu_s$ which satisfy the condition of Lemma \ref{Peyriere}. Let $(n_j)$ such that 
$n_{j+1} \geq n_j+3$, and  put 
$$  \nu_s=\prod_{j=0}^{+\infty}|P_{n_j}(\theta)|^2.$$
\begin{lemm}There is a subsequence $\Big\{t_j, j \in \N\Big\} \subset \R_{+}$ such that  the generalized Riesz products $(\nu_s)_{s \in ]0,1]}$ satisfy  
\begin{enumerate}[(i)]
\item \label{i} $\widehat{\nu_s}(\pm t_j)=\frac{1}{p_{n_j}}$,
\item \label{ii} $\widehat{\nu_s}(t_j\pm t_k)=\widehat{\nu_s}(t_j) \widehat{\nu_s}(t_k).$
\end{enumerate}
\end{lemm}
\begin{proof}By definition of the family $(\nu_s)_{s \in ]0,1]}$, we have
	$$\widehat{\nu_s}(t)=\lim_{J \longrightarrow +\infty} \int_{\R} e^{-it \theta} \prod_{j=0}^{J}|P_{n_j}(\theta)|^2 d\lambda_s.$$
	Let $t_k=d_{n_k}$. Then, by Proposition \ref{Key-Prop}, we have $\widehat{\nu_s}(t)=\pi_k(t)$, whenever $t \in [-q_k,q_k]$. We further have
	$\widehat{\nu_s}(t_k)=\frac{1}{p_{n_k}}$ and since $\widehat{\nu_s}(-t)=\widehat{\nu_s}(t)$, for any $t \in \R_{+}$, we get
	 	$\widehat{\nu_s}(-t_k)=\frac{1}{p_{n_k}}.$ This proves \ref{i}. To prove \ref{ii}, suppose $j <k$ and apply \eqref{eqkey-1} combined \eqref{eqkey-2} to get
	 \begin{eqnarray}
	 	\widehat{\nu_s}(t_j\pm t_k)&=\pi_k(t_j\pm t_k))=\pi_{k-1}(\pm t_j)\frac{1}{p_{n_k}}\\
	 	&=\widehat{\nu_s}(\pm t_j) \widehat{\nu_s}(t_k)\\
	 	&=\widehat{\nu_s}(t_j) \widehat{\nu_s}(t_k),
	 \end{eqnarray}	
	 	since the support of $\widehat{\pi_j}$ is a subset of the support of $\pi_{k-1}$ which is subset of $[-q_{k-1},q_{k-1}]$. The proof of the lemma is complete. 
\end{proof}
\noindent{}We proceed now to the proof of Theorem \ref{Klemes-Reinhold}.
\begin{proof}[\textbf{Proof of Theorem  \ref{Klemes-Reinhold} }]According to our assumption the series $(\frac{1}{p_n})_{n \in \N}$ is not in $\ell^2(\N).$ Therefore, 
	for any $k=3, \cdots,$ there is $r \in \{0,1,\cdots, k-1\}$ such that 
$$\sum_{n \equiv r [k] }\frac{1}{p_n^2}=+\infty.$$
Therefore, we can choose $(n_j)$ such that $n_{j+1} \geq n_j+3$ and  the series $(\frac{1}{p_{n_j}})_{j \in \N}$ is not in $\ell^2(\N).$ Now, by applying Lemma 
\ref{Peyriere} with $\mu=\nu_s$ we obtain that  $\nu_s$ is singular for any $s \in ]0,1]$. This finish the proof of the theorem.
\end{proof}
\section{Klemes-Parreau's theorem for linear staircase rank one flows.}\label{Klemes2}
In this section, we extended Klemes theorem \cite{Klemes2} to linear staircase rank one flows. For that, we will adapt the material from Tuesday, 20 December 1994 conference of Fran\c{c}ois Parreau given at the seminar of ergodic theory of Paris Jussieu on Klemes's theorem. In this conference, F. Parreau presented an extension of Klemes's result. The condition on the cutting parameter was improved. But, the result still unpublished until this paper.\\

We recall that the rank one flow is in the class of linear staircase  if
the family of spacers $(s_{n,k})_{k=0}^{p_n-1}$ is given by 
$s_{n,1}=s_{n,p_n-1}=0$ and $s_{n,k}=k \alpha,$ $\alpha>0, k=2,\cdots,p_n-2.$ In this case, the polynomials associated to the spectral type are given by
$$P_n(\theta)=\frac{1}{\sqrt{p_n}}\sum_{j=0}^{p_n-1}e^{i(jh_n+\overline{s_n}(j))\theta},$$
where $\overline{s_n}(0)=0$ and $\overline{s_n}(j)=\frac{j (j+1)}{2}\alpha, j=1,\cdots,p_n-1$. We fix also, as before, $K,n_1<n_2<\cdots<n_K$, and
$$Q(\theta)=\prod_{j=1}^{K}|P_{n_j}(\theta)|. $$

\noindent{}The subject of this section is to establish the following theorem.
\begin{thm}\label{Klemes} If $\ds \frac{p_n^3}{h_n} \longrightarrow 0,$ as $n \longrightarrow +\infty$, then 
	the spectral type of the linear staircase flow $\sigma$ is singular.
\end{thm}

Let us stress that Theorem \ref{Klemes} extended Klemes's theorem even for the case of $\Z$ action.
 We start by putting, for any $\ell \in \N^*$,
$$D_{\ell}(\theta)=\sum_{j=0}^{\ell-1}e^{i j \theta},  \textrm{~~~and~~~} a_{n,j}=jh_n+\overline{s_n}(j), j=0,\cdots,p_n-1,$$
and for convenience, we further put
$$D_{-\ell}(\theta)=\overline{D_{\ell}(\theta)} \textrm{~~~and~~~} a_{n,-j}=\overline{a_{n,j}}.$$
For $\theta \not \in 2\pi\Z$, we have 
$$D_{\ell}(\theta)=\frac{e^{i\ell \theta}-1}{e^{i\ell \theta}-1},$$
Moreover, for any $x \in \R$, $|x| \leq \frac{1}{8 \ell}$, we have $\ds \Re(D_{\ell}(2 \pi x)) \geq \frac{\ell}{\sqrt{2}}.$ Indeed,
\begin{align}\label{Estim-D-1}
\Re(D_{\ell}(2 \pi x))&=\sum_{j=0}^{\ell-1}\cos(2j \pi x)\\
& \geq 	\frac{\ell}{\sqrt{2}},
\end{align}
since, for any $j=0,\cdots,\ell-1$, $|2j \pi x| \leq \frac{|j| \pi}{4\ell} \leq \frac{\pi}{4}$, and $\cos(2j \pi x) \geq \cos(\frac{\pi}{4})=\frac{1}{\sqrt{2}}.$\\
 
\noindent We have also the following lemma.
\begin{lemm}\label{PK} Let $\theta \in \R$ and $n \in \N^*$, then, we have
$$\big|P_n(\theta)\big|^2-1=g_n(\theta)+\overline{g_n}(\theta),$$ where
$$g_n(\theta)=\frac{1}{p_n}\sum_{k=1}^{p_n-1}e^{i a_{n,k}\theta}D_{p_n-k}(k\theta\alpha).$$
We further have,  
$$\Re\Big(D_{p_n-k}(k\theta\alpha)\Big) \geq \frac{p_n}{4 \sqrt{2}}, \textrm{~~for~~} 0<k \leq \frac{3. p_n}{4} \textrm{~~and~~} \Big|\theta-\frac{2 \pi j}{\alpha k}\Big|< \frac{4 \pi}{\alpha p_n^2}, $$
\end{lemm}
\begin{proof}Write
	\begin{align*}
		\big|P_n(\theta)\big|^2-1&=\frac{1}{p_n}\sum_{j \neq k}e^{i\big(a_{n,j}-a_{n,k}\big)\theta}\\
		&=\frac{1}{p_n}\sum_{j \neq k}e^{i\big((j-k)h_n+\overline{s_{n,j}}-\overline{s_{n,k}}\big)\theta}\\
		&=g_n(\theta)+\overline{g_n}(\theta).
	\end{align*}
where 
\begin{align*}
	g_n(\theta)&=\frac{1}{p_n}\sum_{\ell=1}^{p_n-1}e^{i \ell h_n \theta}\sum_{k=0}^{p_n-\ell-1}e^{i (\overline{s_{n,k+\ell}}-\overline{s_{n,k}})\theta}\\
	&=\frac{1}{p_n}\sum_{\ell=1}^{p_n-1}e^{i a_{n,\ell}\theta}D_{p_n-\ell}(\ell\theta\alpha),
\end{align*}
since $\overline{s_{n,k+\ell}}-\overline{s_{n,k}}=k\alpha+\frac{\ell(\ell+1)}{2}\alpha,$
which finish the proof of the lemma.
\end{proof}
\begin{rem}It follows from \eqref{Estim-D-1} that if $ |x-\frac{j}{k}| \leq \frac{1}{8k(p_n-k)}$ then $\Re\big(D_{p_n-k}(kx)\big) \geq \frac{p_n-k}{\sqrt{2}}$. Moreover, if 
$ |x-\frac{j}{k}| < \frac{2}{{p_n}^2}$ and	$0 <k<\frac{3p_n}{4}$ then $\Re\big(D_{p_n-k}(kx)\big) \geq \frac{p_n}{4\sqrt{2}}.$ 
\end{rem}
Following the strategy of the proof given by Klemes, as it was presented by Fran\c{c}ois Parreau, for the case of the torus, we shall construct a family of positives $\frac{2\pi}{\alpha}$-periodic functions 
$(f_{n,k})$ with disjoint support such that for any non-negative continuous function $\omega$, we have
\begin{enumerate}
\item \label{I} $\ds \sum_{\frac{p_n}{4}  \leq  k \leq \frac{3p_n}{4}} f_{n,k} (\theta)\leq 1,$ for all $\theta \in \R$.
\item \label{II}$\ds \frac{1}{p_n}\sum_{\frac{p_n}{4}\leq  k \leq \frac{3p_n}{4}}  \Re\Big(\int_{\R} \omega(\theta) D_{p_n-k}(\theta) f_{n,k}(\theta) d\lambda_s(\theta) \Big)
\geq c \int_{\R}\omega d\lambda_s(\theta),$ where $c$ is an absolute constant,\\
\item \label{III}and, under the condition of Theorem \ref{Klemes}, we have $$\ds \frac{1}{p_n}\sum_{1 \leq k \neq k' \leq n} \Re\Big(\int_{\R} \omega(\theta) e^{i(a_{n,k}-a_{n,k'})\theta} 
 D_{p_n-k'}(\theta) f_{n,k}(\theta) d\lambda_s(\theta) \Big)\tend{n}{+\infty}0.$$ 
\end{enumerate}
 Put
 $$\chi_n(\theta)=\1_{[-\frac{\pi}{2 \alpha  p_n^2},\frac{\pi}{2  \alpha p_n^2}]}(\theta), \textrm{~~and~~} \chi_n(\theta+\frac{2\pi}{\alpha})=\chi_n(\theta), \forall \theta \in \R,$$
 and
 $$f_{n,k}(\theta)=\sum_{\overset{j \wedge k=1}{1\leq  j \leq k-1}}\chi_n\Big(\theta-\frac{2 \pi j}{\alpha k}\Big).$$
The functions $f_{n,k}$ register the number of time for which $\Re(D_{p_n-k'}(\theta)) \geq \frac{p_n}{4 \sqrt{2}}.$ \\

\noindent Define 
 $$F_{n}(\theta)=\sum_{\frac{p_n}{4 } \leq  k \leq \frac{3p_n}{4}}e^{i a_{n,k} \theta}f_{n,k}(\theta),~~~ \theta \in \R.$$
By construction, we have 
\begin{lemm}\label{bounded}For all $\theta \in \R$, $\big|F_{n}(\theta)\big| \leq 1.$
\end{lemm}
\begin{proof}It suffices to see that the support of the functions $f_{n,k}$ is disjoint. For that, let $(k,j) , (k',j') \in \big\{\frac{p_n}{4},\cdots, \frac{3p_n}{4}\big\}$ 
such that $(k,j) \neq (k',j')$, $k \wedge j=k' \wedge j'=1.$ Then 
\begin{eqnarray*}
	\big|\frac{j}{k}-\frac{j'}{k'}\big|&=&\frac{\big|jk'-kj'\big|}{kk'} \\
	&\geq& \frac{1}{kk'}\\
	&\geq& \frac{16}{9.p_n^2}>\frac{1}{p_n^2}=2.\frac{1}{2 .p_n^2}.
\end{eqnarray*}
This implies
$$\Big(\frac{2 \pi j}{\alpha k}+[-\frac{\pi}{2 \alpha  p_n^2},\frac{\pi}{2 \alpha p_n^2}]\Big) \bigcap \Big(\frac{2 \pi j'}{\alpha k'}+[-\frac{\pi}{2 \alpha p_n^2},\frac{\pi}{2 \alpha p_n^2}]\Big)=\emptyset,$$
which complete the proof of the lemma.
\end{proof}
We proceed now to check \eqref{II} and \eqref{III}. For that, we observe first that by Lemma \ref{bounded}, we have
\begin{align}\label{Minorer}
	\Big|\int Q \big(\big|P_n(\theta)\big|^2-1\big)\overline{F_n}(\theta)  d\lambda_s(\theta)\Big| \leq 	\int Q \big|\big|P_n(\theta)\big|^2-1\big| d\lambda_s(\theta).
\end{align}
We further have
\begin{align*}
	\int Q(\theta) \big(\big|P_n(\theta)\big|^2-1\big)\overline{F_n}(\theta)  d\lambda_s(\theta) &= \int Q(\theta) (g_n(\theta)+\overline{g_n}(\theta)\big)\overline{F_n}(\theta)  d\lambda_s(\theta)\\
	&= \textrm{I}_n+\textrm{II}_n,
\end{align*}
where 
$$\textrm{I}_n=\frac{1}{p_n}\int Q(\theta) \sum_{\frac{p_n}{4} \leq k \leq \frac{3p_n}{4}}e^{i a_{n,k}\theta}D_{p_n-k}(k\theta\alpha)\big)f_{n,k}(\theta)  d\lambda_s(\theta),$$
and
$$\textrm{II}_n=\frac{1}{p_n}\int Q(\theta) \sum_{\overset{k \neq k'}{\frac{p_n}{4} \leq k' \leq \frac{3p_n}{4},1 \leq |k|\leq (p_n-1)}}e^{i (a_{n,k}-a_{n,k'})\theta}D_{p_n-k}(k\theta\alpha)\big)f_{n,k'}(\theta)  d\lambda_s(\theta),$$
We claim:
\begin{claim}\ \label{KP}
	\begin{enumerate}[(i)]
		\item \label{KP-1} $ \ds \liminf_{n\rightarrow \infty}|\textrm{I}_n| \geq c \int Q(\theta)  d\lambda_s(\theta),$ for som absolute constant $c$.
		\item \label{KP-2}$ \textrm{II}_n \tendn 0.$
	\end{enumerate}
\end{claim}
At this point, let us notice that from \ref{KP-1} and \ref{KP-2}, we have
$$	\int Q(\theta) \big(\big|P_n(\theta)\big|^2-1\big)\overline{F_n}(\theta)  d\lambda_s(\theta) \geq |\textrm{I}_n|-|\textrm{I}_n|.$$
Whence
$$ 	\liminf_{n\rightarrow \infty}\int Q \big|\big|P_n(\theta)\big|^2-1\big| d\lambda_s(\theta) \geq c \int Q(\theta)  d\lambda_s(\theta).$$
We thus get, by Proposition \ref{CDsuffit} that $\sigma_{0,s}$ is singular and the proof of Theorem \ref{Klemes} is complete. Therefore, we need only to establish \ref{KP-1} and \ref{KP-2} of Claim \ref{KP}.\\
For the proof of Claim \ref{KP}, we need the following observation.
\begin{obs}Let $f \in L^1(\R)$ and define $\widetilde{f}$ by
	$$\widetilde{f}(\theta)=\frac{\alpha}{2 \pi}\sum_{n \in \Z} f(\theta+\frac{2 \pi n}{\alpha}), ~~\theta \in \R.$$
Obviously, $\phi(\theta)$ depend only on $\theta$ mod $\frac{2 \pi}{\alpha}$. We can thus consider $\widetilde{f}$ as defined on $[0, \frac{2\pi}{\alpha})$ equipped with the 
normalized Lebesgue measure $d\theta_\alpha=\frac{\alpha}{2 \pi} d\theta$.   We further have
$$\widehat{\widetilde{f}}(n)=\widehat{f}(n), \forall n \in \Z,$$
and 
$$\big\|\widetilde{f}\big\|_1 \leq \big\|f\big\|_1.$$
\end{obs}
\noindent{}According to this observation, we shall proved the following
$$\liminf \Re(\textrm{I}_n) \geq c \int Q(\theta)  d\lambda_s(\theta),$$ 
for some absolute constant $c.$ Fix $k \in \big\{\frac{p_n}{4},\cdots,\frac{3 p_n}{4}\}$, and write
$$\textrm{I}_{n,k} \setdef \int_{0}^{\frac{2 \pi}{\alpha}} D_{p_n-k}(\theta)f_{n,k}(\theta) \widetilde{QK_s}(\theta) d\theta_\alpha.$$
Whence
\begin{align*}
Re(\textrm{I}_{n,k})&=\frac{1}{p_n}\int_{0}^{\frac{2 \pi}{\alpha}} \Re(D_{p_n-k}(\theta))f_{n,k}(\theta) \widetilde{QK_s}(\theta) d\theta_\alpha,\\
& \geq \frac{1}{p_n}\sum_{\overset{j \leq k}{j \wedge k=1}}\int_{0}^{\frac{2 \pi}{\alpha}} \Re(D_{p_n-k}(\theta))\chi\Big(\theta-\frac{2 \pi j}{k \alpha}\Big) \widetilde{QK_s}(\theta) d\theta_\alpha
\\
& \geq \frac{1}{4 \sqrt{2}} \int_{0}^{\frac{2 \pi}{\alpha}} f_{n,k}(\theta)  \widetilde{QK_s}(\theta) d\theta_\alpha
\end{align*}
Therefore
\begin{align}
	Re(\textrm{I}_{n})&=Re\Big(\sum_{\frac{p_n}{4} \leq k \leq \frac{3p_n}{4}}\textrm{I}_{n,k}\Big)\\
	& \geq \frac{1}{4 \sqrt{2}} \int_{0}^{\frac{2 \pi}{\alpha}} \sum_{\frac{p_n}{4} \leq k \leq \frac{3p_n}{4}}f_{n,k}(\theta)  \widetilde{QK_s}(\theta) d\theta_\alpha
\end{align}
\noindent We proceed now to the estimation of 
 $$ \liminf_{n \longrightarrow +\infty}  \int_{0}^{\frac{2 \pi}{\alpha}}  \sum_{\frac{p_n}{4} \leq k \leq \frac{3p_n}{4}}f_{n,k}(\theta)  \widetilde{QK_s}(\theta) d\theta_\alpha.$$
For that we claim that 
\begin{claim}\label{Weak-leb}
	$\mu_n \setdef  \frac{\Big(\ds \sum_{\frac{p_n}{4} \leq k \leq \frac{3p_n}{4}}f_{n,k}(\theta)\Big) d\theta_\alpha}{
		\ds \int \ds \sum_{\frac{p_n}{4} \leq k \leq \frac{3p_n}{4}}f_{n,k}(x) dx_\alpha }$ converge weakly to $d\theta_\alpha$.
\end{claim}
\noindent We start by computing 
$$\int_{0}^{\frac{2 \pi}{\alpha}} \ds \sum_{\frac{p_n}{4} \leq k \leq \frac{3p_n}{4}}f_{n,k}(x) dx_\alpha.$$
For each $k \in  \{\frac{p_n}{4}, \cdots, \frac{3p_n}{4}\}$, we have
\begin{align*}
 	\int_{0}^{\frac{2 \pi}{\alpha}} f_{n,k}(x) dx_\alpha&=\sum_{\overset{j \leq k}{j \wedge k=1}} \int_{0}^{\frac{2 \pi}{\alpha}} \chi(x-\frac{2 \pi j}{k \alpha})dx_\alpha\\
 	&= \frac{\phi(k)}{p_n^2},
 \end{align*} 
where $\phi(k)$ is the indicator of Euler. 
\noindent At this point we need the following classical lemma  from Number theory for the proof of it we refer to \cite{Sierpinski}.
\begin{lemm}\label{Ser}For a large enough $x>0$, we have $\ds \sum_{1 \leq k \leq x}\phi(k) \sim \frac{3}{\pi^2}x^2$.
\end{lemm}
 Applying Lemma \ref{Ser}, we obtain
\begin{align}
\int_{0}^{\frac{2 \pi}{\alpha}} 	\sum_{\frac{p_n}{4} \leq k \leq \frac{3p_n}{4}}f_{n,k}(\theta)\Big) d\theta_\alpha
&=\frac{1}{ p_n^2}\sum_{\frac{p_n}{4} \leq k \leq \frac{3p_n}{4}}\phi(k)\notag\\
&=\frac{1}{ p_n^2}\sum_{k \leq \frac{3p_n}{4}}\phi(k)-\frac{1}{ p_n^2}\sum_{ k \leq \frac{p_n}{4}}\phi(k) \notag\\
&\tendn \frac{3}{2\pi^2}
\end{align} 
We thus deduce that for the proof of Claim \ref{Weak-leb} we need to establish
\begin{lemm}\label{Weak-leb-2}
	$\mu_n =\ds \frac{2}{p_n}\sum_{\frac{p_n}{4} \leq k \leq \frac{3p_n}{4}} \frac{\big(p_n^2.f_{n,k}(\theta)\big) d\theta_\alpha}{\phi(k) }$ converge weakly to $d\theta_\alpha$.
\end{lemm}
\begin{proof}For each $k \in \big\{\frac{p_n}{4} \leq k \leq \frac{3p_n}{4}\big\}$, put
	$$\mu_{n,k}=\frac{\big(p_n^2.f_{n,k}(\theta)\big) d\theta_\alpha}{\phi(k) }.$$
It suffit to see that $(\mu_{n,k})$ converge weakly to $d\theta_\alpha$. For that, let us put
$$\nu_{n,k}=\frac{1}{\phi(k)}\sum_{\overset{j \leq k}{j \wedge k=1}} \delta_{\frac{j}{k}}.$$
We are going to show that $(\nu_{n,k})$ converge weakly to Lebesgue measure $dx$ on $[0,1)$. Define
$$\eta_{n,k}= \sum_{1 \leq j \leq k } \delta_{\frac{j}{k}}.$$
Therefore, for any $a \in \Z^{*}$,
$$\widehat{\eta_k}(a)=\begin{cases}
0 & \textrm{if~~} k \nmid a\\
k & \textrm{if ~~} k \mid a.
\end{cases}
$$
But, by M\"obius inverse formula, we have
\begin{align*}
	\phi(k) \nu_{n,k}&=\sum_{d \mid k}\bmu(d) \eta_{\frac{k}{d}}\\
	&=\sum_{d \mid k}\bmu\Big(\frac{k}{d}\Big) \eta_{d}
\end{align*}
Hence
\begin{align*}
\Big|\phi(k) \widehat{\nu_{n,k}}(a)\Big|&=\Big|\sum_{\overset{d \mid k}{d \mid a}}\bmu\Big(\frac{k}{d}\Big) d\Big|\\
& \leq \sum_{d \mid a}d,
\end{align*}
The last term is a constant which depend only on $a$.  We thus conclude that
$$\widehat{\nu_{n,k}}(a) \tendn 0=\widehat{dx}(a).$$
\end{proof}
Now, for any interval $I \subset [0,1)$, we obtain
\begin{align}
\mu_{n,k}(I)-2). \frac{4}{{p_n}^2} \leq  \int \chi_I	f_{n,k} d\theta &\leq \mu_{n,k}(I). \frac{4}{{p_n}^2},\\
\end{align} 
since the intervals $\frac{j}{k}+[-\frac{2}{p_n^2},+\frac{2}{p_n^2}]$ are disjoints, only two interval can contain the endpoints of $I$. Moreover,
$$\int f_{n,k} d\theta= \phi(k)\frac{4}{{p_n}^2}.$$ We thus conclude, since $\phi(k) \longrightarrow  \infty$ as $n \longrightarrow +\infty$, that 
    \begin{equation}
\frac{\displaystyle \int_{0}^{1} \chi_I \, f_{n,k}(x) \, dx}%
{\displaystyle \int_{0}^{1} f_{n,k}(x) dx}= \int_{0}^{1}\chi_I \,d\nu_{n,k} \tendn \int_{0}^{1} \chi_I(x)\, dx..
\end{equation}
Summarizing, we have proved
\begin{align}
	\liminf_{n\rightarrow \infty} Re(\textrm{I}_{n}) &\geq \frac{1}{4 \sqrt{2}}\int \widetilde{QK_s}(\theta) d\theta_\alpha  \liminf \int \sum_{\frac{p_n}{4} \leq k \leq \frac{3p_n}{4}}f_{n,k}(x) dx_\alpha\\
	&\geq \frac{1}{4 \sqrt{2}} \frac{3}{2 \pi^2} \int_{\R} Q(\theta) d\lambda_s. 
\end{align}
This proved \ref{KP-1} of Claim \ref{KP} with the constant $c=\frac{3}{4 \sqrt{2}  \pi^2}.$ We proceed in the same manner to prove \ref{KP-2}, by reducing the study to the case
of the torus $[0,1).$ We point out here that we modify slightly the construction of the rank one by taking at the first stage a rectangle with height $\alpha$, so that 
$\frac{h_n}{\alpha} \in \N$, for any $n \geq 1$. We consider, without loss of generality, that $Q$ is a trigonometric polynomial given by 
$$ Q(\theta)=\sum_{-\omega \leq j \leq \omega} c_j e^{\frac{2 \pi}{\alpha} i j \theta}.$$ 
Therefore 
\begin{align}
\textrm{II}_n&=\sum_{-\omega \leq j \leq \omega}  \frac{c_j}{p_n} \int_{0}^{1} \sum_{ k \neq k'}
e^{\frac{2 \pi}{\alpha} i (a_{n,k}-a_{n,k'}+j) \theta} D_{p_n-k}(\theta) f_{n,k'}(\theta)  d\theta\\
&= \sum_{0 \leq |m| \leq \omega+p_n^2} \frac{d_m}{p_n} \int_{0}^{1} \sum_{ k \neq k'}
e^{\frac{2 \pi}{\alpha} i (a_{n,k}-a_{n,k'}+m) \theta} f_{n,k'}(\theta)  d\theta.
\end{align}
The last equality follows from writing, for each $\frac{p_n}{4} \leq k \leq \frac{3 p_n}{4}$,
$$Q(\theta) D_{p_n-k}(\theta)= \sum_{0 \leq |m| \leq \omega+p_n^2} d_m e^{\frac{2 \pi}{\alpha} i m \theta}, \textrm{~~with~~} d_m=0  \textrm{~~for~~}
|m| \geq \omega+k(p_n-k).$$ 
We thus have, for all $ m \in \big\{-(p_n^2+\omega), \cdots,  (p_n^2+\omega) \big\}$, $$|d_m| \leq p_n \sup_{|j| \leq \omega}|c_j|.$$
Whence
\begin{align}
\textrm{II}_n=\sum_{0 \leq |m| \leq \omega+p_n^2} \frac{d_m}{p_n} \sum_{ k \neq k'} \widehat{f_{n,k'}}(a_{n,k'}-a_{n,k}-m). 
\end{align}
But, for each $k'$, we have
$$\widehat{f_{n,k'}}(a_{n,k'}-a_{n,k}-m)=\sum_{\overset{j \leq k'}{j \wedge k'=1}}\widehat{\chi_{\frac{j}{k'}}}(a_{n,k'}-a_{n,k}-m),$$
where  $\chi_{\frac{j}{k'}}(x)=\chi_{[-\frac{2}{p_n^2},\frac{2}{p_n^2}]}(x-\frac{j}{k'}).$\\
Now, a straightforward computation, gives
$$|\widehat{\chi_{\frac{j}{k'}}}(\ell)|=\Big|\frac{\sin\Big(\frac{2\pi \ell}{\alpha p_n^2}\Big)}{\pi \ell}\Big|,  \textrm{~~for~~} \ell=a_{n,k'}-a_{n,k}-m.$$
Moreover, 
\begin{align}
	|\ell|  &\geq |a_{n,k'}-a_{n,k}|-|m| \\
	&\geq h_n-(p_n^2+\omega)= h_n(1-\frac{p_n^2+\omega}{h_n})
\end{align}
Hence, by our assumption, for $n$ largely enough, we have $\ds |\ell| \gg \frac{h_n}{2}$, and
\begin{align}
\Big|	\textrm{II}_n \Big|&\leq C. p_n^3 \sup_{ |\ell| \geq \frac{h_n}{2}} |\widehat{\chi_{\frac{j}{k'}}}(\ell)|\\
&\leq C. p_n^3 \sup_{ |\ell| \geq \frac{h_n}{2}} \frac{1}{\pi |\ell|} \\
& \leq C' \frac{p_n^3}{h_n} \tendn 0. 
\end{align}
the last assertion follows for our assumption. This completes the proof of\ref{KP-2} of Claim \ref{KP} and the proof of the theorem (Theorem \ref{Klemes}.)
\section{Exponential staircase rank one flows}
The main purpose of this section is to study
the spectrum of a subclass of a rank one flows called
{\it exponential staircase rank one flows} which are defined as follows.

Let $(m_n,p_n)_{n \in \N}$ be a sequence of positive integers such that
$m_n$ and $p_n$ goes to infinity as $n$ goes to infinity. Let $\varepsilon_n$ be a sequence of
rationals numbers which converge to $0$. Put
$$
\omega_n(p)=\frac{m_n}{\varepsilon_n^2}p_n\Big(\exp\big(\frac{\varepsilon_n.p}{p_n}\big)-1\Big)
{\rm {~~for~any~~}} p \in \{0,\cdots,p_n-1\},
$$
and define the sequence of the spacers $((s_{n+1,p})_{p=0,\cdots,p_n-1})_{n \geq 0}$ by
$$
 h_n+s_{n+1,p+1}=\omega_n(p+1)-\omega_n(p),~~~p=0,\cdots,p_n-1, n \in \N.
$$
In this definition we assume that
\begin{enumerate}
 \item \label{CD1} $m_n \geq \varepsilon_n.h_n$, for any $n \in \N$.
\item \label{CD2} $\displaystyle \frac{\log(p_n)}{m_n} \tendn 0$ if $p_n \geq \frac{m_n}{\varepsilon_n}$
\item \label{CD3} $\displaystyle \frac{\log(p_n)}{m_n} \tendn 0$ and  $\displaystyle \frac{\log(p_n)}{p_n} \leq \varepsilon_n$
if $p_n < \frac{m_n}{\epsilon_n}$
\end{enumerate}
We will denote  this class of rank one flow by

$$(T^t)_{t \in \R} \stackrel{\text{def}}{=}
\left(T^t_{(p_n,\omega_n)_{n\geq0}}\right)_{t \in \R}.$$
It is easy to see that this class of flows contain a large class of examples introduced by Prikhodko \cite{prikhodko}. Indeed,
assume that $h_n^{\beta} \geq p_n \geq h_n^{1+\alpha}$, $\beta \geq 2$ and $\alpha \in ]0,\frac14[$. Then, by  assumption
(1), we have
$$\frac{\log(p_n)}{m_n} \leq \beta \frac{\log(h_n)}{\varepsilon_n h_n},$$
Taking $\beta=\varepsilon_n \big(\lfloor h_n^{\delta} \rfloor+1\big)$, $0<\delta<1$, we get
$$
\frac{\log(p_n)}{m_n}  \tendn 0.
$$


At this point we stat the main result of this section.
\begin{thm}\label{main-expo}
Let $(T^t)_{t \in \R}=
\left(T^t_{(p_n,\omega_n)_{n\geq0}}\right)_{t \in \R}$
 be a exponential staircase rank one flow
associated to
\begin{eqnarray}
\nonumber \omega_n(p)=\frac{m_n}{\varepsilon_n^2}p_n \exp\big({\frac{\varepsilon_n.p}{p_n}}\big),
~~p=0,\cdots,p_n-1,
\end{eqnarray}  
which satisfy the condition \eqref{CD1}, \eqref{CD2} and \eqref{CD3}.
Then the spectrum of $(T_t)_{\R}$ is singular.
\end{thm}

The proof of Theorem \ref{main-expo} is the subject of the next section.

\section{The CLT method for trigonometric sums and the singularity of the spectrum
	of exponential staircase rank one flows}
The main goal of this section is to prove the following proposition
\begin{prop}\label{key-sasha} Let $(T^t)_{t \in \R}=
\left(T^t_{(p_n,\omega_n)_{n\geq0}}\right)_{t \in \R}$
 be a exponential staircase rank one flow
associated to
\begin{eqnarray}
\nonumber \omega_n(p)=\frac{m_n}{\varepsilon_n^2}p_n \exp\big({\frac{\varepsilon_n.p}{p_n}}\big),
~~p=0,\cdots,p_n-1.
\end{eqnarray} Then, there exists a constant $c>0$ such that,
for any positive function $f$ in $L^2(\R,\lambda_s)$, we have
$$\liminf_{m\longrightarrow +\infty} \bigintss_{\R}f(t)\left||P_m(t)|^2-1\right|
\;d\mu_s(t) \geq c \bigintss_{\R}f(t) \;d\lambda_s(t).$$
\end{prop}
The proof of the proposition \ref{key-sasha} is based on the
study of the stochastic behavior of the sequence $ |P_m|$.
For that, we follow the strategy introduced in \cite{elabdaletds} based on the method
of the Central Limit Theorem for trigonometric sums.

This methods takes advantage of the following classical expansion
\[
\exp(ix)=(1+ix)\exp\Big(-\frac{x^2}2+r(x)\Big),
\]
\noindent{}where $|r(x)| \leq |x|^3$, for all real number $x$ \footnote{this is a direct consequence
of Taylor formula with integral remainder.}, combined with some ideas developed in the proof
of martingale central limit theorem due to McLeish \cite{mcleish}. Precisely, the main ingredient
is the following theorem proved in \cite{elabdaletds}.

\begin{thm}\label{Mcleish}
Let $\{X_{nj}:1 \leq j \leq
k_n, n\geq 1\}$ be a triangular array of random variables and $t$ a real number. Let
$$\displaystyle S_n=\sum_{j=1}^{k_n}X_{nj},~~~~~~~\displaystyle T_n=\prod_{j=1}^{k_n}(1+itX_{nj}),$$
and $$\displaystyle U_n=\exp \left (-\frac{t^2}2\sum _{j=1}^{k_n}
X_{nj}^2+\sum_{j=1}^{k_n}r(tX_{nj})\right ).$$ Suppose that
\begin{enumerate}
\item $\{T_n\}$ is uniformly integrable.
\item $\E(T_n) \tendn 1$.
\item $\displaystyle \sum_{j=1}^{k_n} X_{nj}^2 \tendn 1$ in probability.
\item $\ds \max_{1 \leq j \leq k_n} |X_{nj}|\tendn 0$ in probability.
\end{enumerate}
Then $\E(\exp(it S_n))\longrightarrow \exp(-\displaystyle\frac{t^2}2).$
\end{thm}
We remind that the sequence  $\{X_n, n\geq 1\}$ of random  variables is said to be uniformly integrable if and only
if
\[
\lim_{c \longrightarrow +\infty}\bigintss_{\big\{|X_n| >c\big \}} \big|X_n\big | d\Pr =0 ~~~
{\textrm {uniformly~in~}} n.
\]
and it is well-known that if
\begin{eqnarray}\label{uniform}
\sup_{ n \in \N}\bigg(\ce\big(\big|X_n\big|^{1+\varepsilon}\big)\bigg) < +\infty,
\end{eqnarray}
for some $\varepsilon$ positive, then $\{X_n\}$ is uniformly integrable.\\

Using Theorem \ref{Mcleish} we shall prove the following extension to $\R$ of Salem-Zygmund CLT theorem,
which seems to be of independent interest.
\begin{thm}\label{Gauss}
Let $A$ be a Borel subset of $\R$ with $\mu_s(A)>0$ and
let $(m_n,p_n)_{n \in \N}$ be a sequence of positive integers such that
$m_n$ and $p_n$ goes to infinity as $n$ goes to infinity. Let $\varepsilon_n$ be a sequence of
rationals numbers which converge to $0$ and
$$
\omega_n(j)=\frac{m_n}{\varepsilon_n^2}p_n\exp\big(\frac{\varepsilon_n.j}{p_n}\big)
{\rm {~~for~~any~~}} j \in \{0,\cdots,p_n-1\}.
$$    Then, the distribution of the
sequence of random variables
$\frac{\sqrt{2}}{\sqrt{p_n}}\sum_{j=0}^{p_n-1} \cos(\omega_n(j)t)$
converges to the Gauss distribution. That is, for any real number $x$, we have

\begin{eqnarray}{\label{eq:eq6}}
\frac1{\mu_s(A)}\mu_s\left \{t \in A~~:~~\frac{\sqrt{2}}{\sqrt{p_n}}
\sum_{j=0}^{p_n-1} \cos(\omega_n(j)t) \leq x \right \}\nonumber \\
\tend{n}{\infty}\frac1{\sqrt{2\pi}}\bigintss_{-\infty}^{x}
e^{-\frac12t^2}dt\stackrel{\rm {def}}{=}{\mathcal {N}}\left( \left] -\infty,x \right] \right) .
\end{eqnarray}
\end{thm}






As in the proof of Theorem \ref{orn-bourgain} we need to see that the function $\phi$ defined in \ref{Gw-lim} is bounded by below by a universal positive constant. For that
we need to prove Proposition \ref{key-sasha}.
Let $n$ be a positive integer and put

\begin{eqnarray*}
\W_n {\stackrel {\rm {def}}{=}}&&  \Big \{ {\sum_{j \in
I}\eta_j \omega_n(j)} ~~:~~
\eta_j \in \{-1,1\}, I \subset \{0,\cdots,p_n-1\}\Big \}.
\end{eqnarray*}

\noindent{} The element $w =\sum_{i \in
I}\eta_j \omega_n(j)$ is called a word.\\

We shall need the following two combinatorial lemmas. The first one is a classical result
in the transcendental number theory and it is due to Hermite-Lindemann.\\

\begin{lemm}[{\bf Hermite-Lindemann, 1882}]\label{Hermite}
 Let $\alpha$ be a non-zero algebraic number. Then, the number $\exp(\alpha)$ is transcendental.
\end{lemm}

\noindent{}We state the second lemma as follows.

\begin{lemm} For any $n \in \N^*.$
 All the words of $\W_n$ are distinct.
\end{lemm}
\begin{proof}{}
Let $w,w' \in \W_n$, write
\begin{eqnarray*}
w &=&\sum_{j \in I}\eta_j \omega_n(j),\\
w' &=&\sum_{j \in I'}\eta'_j \omega_n(j).
\end{eqnarray*}
\noindent{}Then $w=w'$ implies
\[
\sum_{j \in I}\eta_j \omega_n(j)-\sum_{j \in I'}\eta'_j \omega_n(j)=0
\]
Hence
\[
\sum_{j \in I}\eta_j \exp(\frac{\varepsilon_n}{p_n}j)-\sum_{j \in I'}\eta'_j \exp(\frac{\varepsilon_n}{p_n}j) =0
\]
But Lemma \ref{Hermite} tell us that $e^{\ds \varepsilon_n/p_n}$ is a transcendental number. This clearly
forces $I=I'$ and
the proof of the lemma is complete.
\end{proof}

\begin{proof}[Proof of Theorem \ref{Gauss}] Let $A$ be a Borel set with
$\lambda_s(A)>0$ and notice that for any positive integer $n$, we have
$$\bigintss_{\R}\Big|\frac{\sqrt{2}}{\sqrt{p_n}}
\sum_{j=0}^{p_n-1} \cos(\omega_n(j)t)\Big |^2 d\lambda_s(t) \leq 1.$$
Therefore, applying the Helly theorem we may assume that the
sequence $$\Big(\frac{\sqrt{2}}{\sqrt{p_n}}
\sum_{j=0}^{p_n-1} \cos(\omega_n(j)t)\Big)_{n \geq 0}$$ converge in distribution. As is well-known, it is sufficient to show that for every real
number $x$,

\begin{eqnarray}{\label {eq:eq7}}
  \displaystyle \frac1{\mu_s(A)}\bigintss_A \exp\left \{-ix\frac{\sqrt{2}}{\sqrt{p_n}}
  \sum_{j=0}^{p_n-1} \cos(\omega_n(j)t) \right \} d\lambda_s(t)
  \tend{n}{\infty} \exp(-\frac{x^2}2). \nonumber
\end{eqnarray}

\noindent{} To this end we apply theorem \ref{Mcleish} in the following
context. The measure space is the given Borel set $A$ of positive
measure with respect to the probability measure $\lambda_s$ on $\R$ equipped with the normalised measure
$\ds \frac{\lambda_s}{\lambda_s(A)}$ and the random
variables are given by
$$
X_{nj}=\frac{\sqrt{2}}{\sqrt{p_n}} \cos(\omega_n(j)t),~~~~{\rm {where}}~~~0 \leq j \leq p_n-1,~
n \in \N.
$$

\noindent{}It is easy to check that the variables $\{X_{nj}\}$
satisfy condition (4). Further, condition (3) follows from the fact that
$$
\bigintss_{\R} \Big |\sum_{j=0}^{p_n-1} X_{nj}^2-1 \Big |^2 d\lambda_s(t)
\tend{n}{\infty} 0.
$$
\noindent{} It remains to verify conditions (1) and (2) of Theorem \ref{Mcleish}. For this purpose, we set
\begin{eqnarray*}
\Theta_n(x,t)&=&\Prod_{j=0}^{p_n-1}
\Big(1-ix\frac{\sqrt{2}}{\sqrt{p_n}}\cos(\omega_n(j)t\Big)\\
&=&1+\sum_{w \in W_n}{\rho_w}^{(n)}(x) \cos(wt), \\
\end{eqnarray*}
\noindent{} and
$$ \W_n= \bigcup_{r} \W_n^{(r)},$$
\noindent{}where $\W_n^{(r)}$ is the set of words of length $r$. Hence
\begin{eqnarray*}
\Big |\Theta_n(x,t)\Big|  \leq
\left \{\prod_{j=0}^{p_n-1}\Big ( 1+\frac{2x^2}{p_n}\Big) \right
\}^{\frac12}.
\end{eqnarray*}

\noindent{}But, since $1+u \leq e^u$, we get

\begin{eqnarray}
\Big|\Theta_n(x,t)\Big| \leq e^{x^2}.
\end{eqnarray}

\noindent{}This shows that the condition (1) is satisfied. It still remains to prove
that the variables $\{X_{nj}\}$
satisfy condition (2). For that, it is sufficient to show that
\begin{eqnarray}\label{eq:eq11}
\bigintss_A \Prod_{j=0}^{p_n-1}\left (
1-ix\frac{\sqrt{2}}{\sqrt{p_n}}\cos(\omega_n(j)t)\right)d\lambda_s(t)
\tend{n}{\infty}\lambda_s(A).
\end{eqnarray}







\noindent{}Observe that

$$\bigintss_{A} \Theta_n(x,t) d\lambda_s(t) =
  \mu_s(A)+\sum_{w \in \W_n}{\rho_w}^{(n)}(x) \bigintss_{A}
  \cos(wt)d\lambda_s(t)$$

\noindent{}and for
$w =\sum_{j=1}^{r}\omega_n(q_j)\in \W_n$, we have

\begin{eqnarray}\label{nicemajoration}
|{\rho_w}^{(n)}(x)| \leq \frac{2^{1-r}|x|^r}{p_n^{\frac{r}2}},
\end{eqnarray}

\noindent{}hence
\[
\max_{w \in \W_n}|{\rho_w}^{(n)}(x)| \leq
\frac{|x|}{p_n^{\frac{1}2}}\tend{n}{\infty}0.
\]
We claim that it is sufficient to prove the following

\begin{eqnarray}{\label{density}}
\bigintss_{\R} \phi \Prod_{j=0}^{p_n-1} \Big
(1-ix\frac{\sqrt{2}}{\sqrt{p_n}}\cos(\omega_n(j)t)\Big) dt
\tend{n}{\infty}\bigintss_{\R} \phi dt,
\end{eqnarray}

\noindent{}for any function $\phi$ with compactly supported
Fourier transforms. Indeed, assume
that (\ref{density}) holds and let $\epsilon>0$.
Then, by the density of the functions with compactly supported
Fourier transforms \cite[p.126]{Katznelson}, one can find a
function $\phi_{\epsilon}$ with compactly supported
Fourier transforms  such that
\[
\Big\|\chi_A.K_s-\phi_\epsilon\Big\|_{L^1(\R)} <\epsilon,
\]
where
\noindent{}$\chi_A$  is indicator function of $A$.
\noindent{}Hence, according to (\ref{density}) combined with \eqref{eq:eq11}, for $n$
sufficiently large, we have

\begin{eqnarray}{\label{eq :eq12}}
&&\Big |\bigintss_A \Theta_n(x,t) d\mu_s(t) -\lambda_s(A)\Big|
=\Big|\bigintss_A\Theta_n(x,t)d\lambda_s(t) -\bigintss_{\R} \Theta_n(x,t) \phi_\epsilon(t) dt+
\\&&\bigintss_{\R} \Theta_n(x,t) \phi_\epsilon(t) dt-\bigintss_{\R}  \phi_\epsilon(t) dt +
\bigintss_{\R}
\phi_\epsilon(t) dt-\lambda_s(A)|< e^{x^2}\epsilon+2\epsilon.
\end{eqnarray}

\noindent{}The proof of the claim is complete. It still remains to prove
(\ref{density}). For that, let us compute the cardinality of words of length $r$
which can belong to the support of $\phi$.\\

\noindent{}By the well-known sampling theorem, we can assume that
the support of $\phi$ is $[-\Omega,\Omega]$, $\Omega>0$. First, it is easy to check  that
 for all odd $r$, $|w_n^{(r)}| \tendn +\infty$. It suffices to consider the words with even length. The case
$r=2$ is easy, since it is obvious to obtain the same conclusion. Moreover, as we will see later,
it is sufficient to consider the case $r=2^k$, $k \geq 2$. We argue that the cardinality of words of length $2^k$,
$k \geq 2$ which can belong to $[-\Omega,\Omega]$ is less than $\displaystyle \Omega.\frac{p_n^k~{(\log(p_n))}^{k-1}}{m_n~ \varepsilon_n^{k-2}}$. Indeed,
for $k=2$. Write
$$w_n^{(4)}=\eta_1\omega_n(k_1)+\eta_2\omega_n(k_2)+\eta_3\omega_n(k_3)+\eta_4.\omega_n(k_4),
$$
with $\eta_i \in \{-1,1\}$ and $k_i \in \{0,\cdots,p_n-1\},~~i=1,\cdots,4,$
and put
$$e_n(p)=\exp\big(\frac{\varepsilon_n}{p_n}.p\big),~~p\in \{0,\cdots,p_n-1\}.$$
If $\sum_{i=1}^{4}\eta_i \neq 0$ then there is nothing to prove since the $\big|w_n^{(4)}\big| \tendn +\infty$. Therefore,
let us assume that  $\sum_{i=1}^{4}\eta_i = 0.$  In this case, without loss of generality (WLOG), we will assume that
$k_1 <k_2<k_3<k_4$. Hence
\begin{eqnarray*}
w_n^{(4)}&=&\eta_1\omega_n(k_1)+\eta_2\omega_n(k_2)+\eta_3\omega_n(k_3)+\eta_4\omega_n(k_4)\\
&=& \displaystyle \frac{m_n}{\epsilon_n^2}p_n e_n(k_1)\Big(\eta_1+\eta_2e_n(\alpha_1)+
\eta_3e_n(\alpha_2)+\eta_4e_n(\alpha_3)\Big)
\end{eqnarray*}
where, $\alpha_i=k_{i+1}-k_1$, $i=1,\cdots,3.$ At this stage, we may assume again WLOG that
$\eta_1+\eta_2=0$ and $\eta_3+\eta_4=0$. It follows that
\begin{eqnarray*}
w_n^{(4)}= \displaystyle \frac{m_n}{\epsilon_n^2}p_n e_n(k_1)\Big(\eta_2\big(e_n(\alpha_1)-1\big)+
\eta_4 e_n(\alpha_2)\big(e_n(\alpha'_1)-1\big)\Big),
\end{eqnarray*}
with $\alpha'_1=\alpha_3-\alpha_2.$ Consequently, we have two cases to deal with.
\begin{itemize}
 \item {\bf Case 1:} $\eta_2=\eta_4$. Then
\begin{eqnarray*}
 \Big| w_n^{(4)} \Big| &\geq& \displaystyle \frac{m_n}{\epsilon_n^2}p_n e_n(k_1) \Big( e_n(\alpha_1)-1 \Big) \\
&\geq& \displaystyle \frac{m_n}{\epsilon_n^2}p_n \frac{\varepsilon_n}{p_n} \tendn +\infty.
\end{eqnarray*}
\item  {\bf Case 2:} $\eta_2=-\eta_4.$ We thus get
\begin{eqnarray*}
 w_n^{(4)}&=& \displaystyle \frac{m_n}{\epsilon_n^2}p_n.e_n(k_1) \eta_4 \Big(
e_n(\alpha_2)\big(e_n(\alpha'_1)-1\big)-\big(e_n(\alpha_1)-1\big)\Big)
\end{eqnarray*}
Hence
\begin{eqnarray*}
 \big|w_n^{(4)}\big|= \displaystyle \frac{m_n}{\epsilon_n^2}p_n e_n(k_1)\Big|
e_n(\alpha_2)\big(e_n(\alpha'_1)-1\big)-\big(e_n(\alpha_1)-1\big)\Big|
\end{eqnarray*}
Therefore, we have three cases to deal with.
\begin{itemize}
 \item {\bf Case 1:} $\alpha'_1>\alpha_1$. In this case,
\begin{eqnarray*}
\big|w_n^{(4)}\big| &=& \displaystyle \frac{m_n}{\epsilon_n^2}p_n e_n(k_1)\Big(
e_n(\alpha_2)\big(e_n(\alpha'_1)-1\big)-\big(e_n(\alpha_1)-1\big)\Big)\\
&\geq& \displaystyle
\frac{m_n}{\epsilon_n^2}p_n \big(e_n(\alpha'_1)-e_n(\alpha_1)\big)\\
&\geq& \displaystyle
\frac{m_n}{\epsilon_n^2}p_n \big(e_n(\alpha'_1-\alpha_1)-1\big) \tendn +\infty.
\end{eqnarray*}
\item {\bf Case 2:} $\alpha'_1<\alpha_1$. Write $\alpha_1=\alpha'_1+\beta$. Thus
\begin{eqnarray*}
\big|w_n^{(4)}\big| &=& \displaystyle \frac{m_n}{\varepsilon_n^2}p_n e_n(k_1) \Big|
e_n(\alpha_2)\big(e_n(\alpha'_1)-1\big)-\big(e_n(\alpha'_1+\beta)-1\big)\Big|\\
&\geq& \displaystyle
\frac{m_n}{\varepsilon_n^2}p_n \Big|e_n(\alpha'_1+\alpha_2)-e_n(\alpha_2)-e_n(\alpha'_1+\beta)+1\Big|   \\
&\geq& \frac{m_n}{\varepsilon^2_n}p_n \Big| \big(e_n(\alpha_2)-1\big)\big(e_n(\alpha'_1)-1\big)-
e_n(\alpha'_1)\big(e_n(\beta)-1\big)
\Big|\\
&\geq& \frac{m_n}{\varepsilon^2_n}p_n {\frac{\varepsilon_n\beta \alpha'_1}{p_n}} \Big(
\frac{e_n(\alpha'_1)\big(e_n(\beta)-1\big)}{\frac{\varepsilon_n\beta \alpha'_1}{p_n}}-
{\frac{\big(e_n(\alpha_2)-1\big)\big(e_n(\alpha'_1)-1\big)}{\frac{\varepsilon_n\beta \alpha'_1}{p_n}}}
\Big) \\
&\geq& \frac{m_n}{\varepsilon_n}\beta \alpha'_1 \Big(
\frac{e_n(\alpha'_1)\big(e_n(\beta)-1\big)}{\frac{\varepsilon_n\beta \alpha'_1}{p_n}}-
{\frac{\big(e_n(\alpha_2)-1\big)\big(e_n(\alpha'_1)-1\big)}{\frac{\varepsilon_n\beta \alpha'_1}{p_n}}}
\Big)
\end{eqnarray*}
\noindent{}But for any $x \in [0,\log(2)[$, we have $ x \leq e^x-1 \leq 2x$. Therefore
$${\frac{\big(e_n(\alpha_2)-1\big)\big(e_n(\alpha'_1)-1\big)}{\frac{\varepsilon_n\beta \alpha'_1}{p_n}}}
\leq 4 \varepsilon_n.
$$
Hence
$$
\big|w_n^{(4)}\big| \tendn +\infty
$$
\item {\bf Case 3:} $\alpha'_1=\alpha_1$. In this case, we have
$$
\big|w_n^{(4)}\big| \geq \displaystyle \frac{m_n}{\epsilon_n^2}p_n \Big(
e_n(\alpha_2)-1\Big)\Big(e_n(\alpha_1)-1\Big).
$$
From this, we have
\begin{eqnarray*}
\big|w_n^{(4)}\big| &\geq& \displaystyle \frac{m_n}{\epsilon_n^2}p_n.\Big(\frac{\varepsilon_n}{p_n}\Big)^2.
\alpha_2.\alpha_1.\\
&\geq& \displaystyle \frac{m_n}{p_n} \alpha_2.\alpha_1
\end{eqnarray*}
\end{itemize}

It follows that
$$\big|w_n^{(4)}\big| \leq \Omega \Longrightarrow \alpha_2.\alpha_1 \leq \frac{p_n}{m_n}.\Omega.$$
But the cardinality of $(\alpha_1,\alpha_2)$ such that
$\displaystyle \alpha_2.\alpha_1 \leq \Omega \frac{p_n}{m_n}$ is less than\linebreak $\displaystyle \Omega \frac{p_n}{m_n} \log(p_n).$ This gives that the
cardinality of words of length $4$ which can belong to $[-\Omega,\Omega]$ is less than
$\displaystyle \Omega \frac{p_n^2}{m_n} \log(p_n).$
\end{itemize}
Repeated the same argument as above we deduce that the only words to take into account at the stage $k=3$ are the form
$$w_n^{(8)}=\frac{m_n}{\varepsilon^2_n}p_ne_n(\alpha)\Big((e_n(\alpha_1)-1)(e_n(\alpha_2)-1)+\eta.
e_n(\alpha_3)(e_n(\alpha_4)-1)(e_n(\alpha_5)-1)\Big),$$
where $\eta=\pm 1$. In the case $\eta=1$, it is easy to see that
\begin{eqnarray*}
 \big|w_n^{(8)} \big|\geq \frac{m_n \varepsilon_n}{p^2_n}\alpha_3  \widetilde{\alpha_1} \widetilde{\alpha_2},
{\textrm {~~where~~}}  \widetilde{\alpha_i}=\inf(\alpha_i,\alpha_{3+i}).
\end{eqnarray*}
For $\eta=-1$, write
\begin{eqnarray*}
\big|w_n^{(8)}\big|&=&\frac{m_n}{\varepsilon^2_n}p_ne_n(\alpha)\Big|(e_n(\alpha_1)-1)(e_n(\alpha_2)-1)-\\
&&(e_n(\alpha_3)-1)(e_n(\alpha_4)-1)(e_n(\alpha_5)-1)-
(e_n(\alpha_4)-1)(e_n(\alpha_5)-1)\Big|.
\end{eqnarray*}
Using the following expansion
\begin{eqnarray}\label{eq:eq13}
e_n(x)=1+\frac{\varepsilon_n .x}{p_n}+o(1),
\end{eqnarray}
we obtain, for a large $n$,
\begin{eqnarray*}
 \big|w_n^{(8)} \big|\geq \frac{m_n \varepsilon_n}{p^2_n}\alpha_3  \widetilde{\alpha_1} \widetilde{\alpha_2}.
\end{eqnarray*}
We deduce that the cardinality of words of length $8$
which can belong to $[-\Omega,\Omega]$ is less than
$$\Omega .p_n.\frac{p^2_n}{m_n \varepsilon_n}\sum_{\widetilde{\alpha_1}, \widetilde{\alpha_2}}
\frac1{\widetilde{\alpha_1} \widetilde{\alpha_2}}
\leq \Omega. \frac{p_n^{3}}{m_n \varepsilon_n} (\log(p_n))^{2} .$$
In the same manner as before consider the words of length $k$ in the following form
$$w_n^{(2^k)}= \displaystyle \frac{m_n}{\epsilon_n^2}p_ne_n(\alpha)\Big((e_n(\alpha_1)-1) \cdots (e_n(\alpha_k)-1)\Big).$$
Therefore
\begin{eqnarray*}
\big|w_n^{(2^k)}\big| &\geq& \displaystyle \frac{m_n}{\epsilon_n^2}p_n.\Big(\frac{\varepsilon_n}{p_n}\Big)^k.
\alpha_1.\alpha_2\cdots \alpha_k.\\
&\geq& \displaystyle \frac{m_n}{p_n^{k-1}} \varepsilon^{k-2}_n  \alpha_1.\alpha_2\cdots \alpha_k
\end{eqnarray*}
which yields as above that the cardinality of words of length $2^k$
which can belong to $[-\Omega,\Omega]$ is less than
$$\Omega .p_n.\frac{p_n^{k-1}}{m_n \varepsilon^{k-2}_n}\sum_{\alpha_2,\cdots,\alpha_k}\frac1{\alpha_2\cdots\alpha_k}
\leq \Omega. p_n.\frac{p_n^{k-1}}{m_n \varepsilon^{k-2}_n} (\log(p_n))^{k-1} .$$
Now, if $r$ is any arbitrary even number. Write $r$ in base 2 as
$$r=2^{l_s}+\cdots+2^{l_1} {\textrm {~~with~~}} l_s > \cdots > l_1 \geq 1.$$
and write
$$w_n^{(r)}=w_n^{(2^{l_s})}+\cdots+w_n^{(2^{l_1})},$$
with
$$
w_n^{(2^{l_j})}=\eta^{(j)}_1\omega_n(k_1^{(j)})+\cdots+\eta^{(j)}_{2^{l_j}}\omega_n(k_{l_j}^{(j)}),~~j=1,\cdots,s,$$
and
$$ \sum_{i=1}^{2^{l_j}}\eta^{(j)}_i=0,~~j=1,\cdots,s.$$
Observe that the important case to consider is the case
$$w_n^{(2^{l_j})}=\pm\frac{m_n}{\varepsilon^2_n}p_n\prod_{i=1}^{l_j}\big(e(\alpha_i^{(j)})-1\big), ~j=1,\cdots,s.$$
Using again \eqref{eq:eq13} we obtain for a large $n$ that
$$\big|w_n^{(r)}\big| \geq \displaystyle \frac{m_n}{p_n^{l_s-1}} \varepsilon^{l_s-2}_n
\prod_{i=1}^{l_s}\alpha^{(s)}_i. $$
Hence, the cardinality of
words of length $r$ which can belong to $[-\Omega,\Omega]$ is less than
$$\Omega.\frac1{m_n}.\frac{p_n^{\lfloor \log_2(r)\rfloor}}{\varepsilon^{\lfloor \log_2(r)\rfloor-2}_n}
(\log(p_n))^{\lfloor \log_2(r)\rfloor-1},$$
and, in consequence,
by \eqref{nicemajoration}, we deduce that
\begin{eqnarray*}
\displaystyle \sum_{w \in\W_n} \Big| {\rho_w}^{(n)}(x)  \int_{\R} e^{-iwt}
\phi(t)dt\Big| \\\leq \displaystyle \sum_{\overset{w \in\W_n^{(r)}, r {\textrm{~~even~~}}}
{4 \leq r \leq p_n}}\Big| {\rho_w}^{(n)}(x)\int_{\R} \phi(t) e^{-iwt} dt\Big|.
\end{eqnarray*}
Therefore, under the assumption (2), we get
\begin{eqnarray*}
&\displaystyle \sum_{w \in\W_n} \Big| {\rho_w}^{(n)}(x)  \int_{\R} e^{-iwt}
\phi(t)dt\Big|\\ &\leq \displaystyle \Omega \frac{|x|}{m_n}
\sum_{\overset{r {\textrm{~~even~~}}}
{4 \leq r \leq p_n}}
\frac{(\log(p_n))^{\lfloor \log_2(r)\rfloor -1}}{p_n^{\big(\frac{r}{2}-2\lfloor\log_2(r)\rfloor+2\big)}}
.\frac1{{\big(p_n.\varepsilon_n\big)}^{\lfloor \log_2(r)\rfloor -2}}\\
&\leq \displaystyle \Omega .|x|.\frac{\log(p_n)}{m_n}
\sum_{\overset{r {\textrm{~~even~~}}}
{4 \leq r \leq p_n}}
\frac1{p_n^{\big(\frac{r}{2}-2\lfloor\log_2(r)\rfloor+2\big)}}.
\end{eqnarray*}
The last inequality is due to the fact that for a large $n$ we may assume that
$\displaystyle\frac{\log(p_n)}{m_n}$ is strictly less than 1. In addition, since $p_n \geq 2$,
$\displaystyle \sum_{\overset{r {\textrm{~~even~~}}}{4 \leq r \leq p_n}}
\frac1{p_n^{\big(\frac{r}{2}-2\lfloor\log_2(r)\rfloor+2\big)}}$ is convergent. We conclude that
\begin{align*}
&\Big|\sum_{w \in W_n}{\rho_w}^{(n)}(x) \bigintss_{\R} \phi(t)
  \cos(wt)dt|\\
  &\leq \Omega .|x|.\frac{\log(p_n)}{m_n}
\sum_{\overset{r {\textrm{~~even~~}}}
{4 \leq r \leq p_n}}
\frac1{p_n^{\big(\frac{r}{2}-2\lfloor\log_2(r)\rfloor+2\big)}} \tend{n}{\infty}0.
\end{align*}

This finishes the proof, the other case is left to the reader.
\end{proof}
\begin{xrem}Let us point out that if
\begin{eqnarray}\label{factoriel}
\frac{h_n.\varepsilon_n^k}{p_n^k} \tendn +\infty, {\textrm {~~for~any~}} k \geq 1.
\end{eqnarray}
Then the spectrum of the associated exponential staircase flow is singular. Indeed, Let
$W_n^{(r)}=\displaystyle\frac{m_n.p_n}{\varepsilon^2_n}\sum_{j=1}^{r}\eta_j e_n(k_j)$ a word of length $r$.
If $W_n^{(r)} \neq 0$. Then, by Lemma \ref{Hermite}, there exists $\alpha \geq 1$ such that
$$\sum_{j=1}^{r}\frac{\eta_j k_j^{\alpha}}{\alpha!} \neq 0.$$
Therefore, using a Taylor expansion of $e_n$ and assuming that all the terms of degree less than $\alpha$ are 0, we have
\begin{eqnarray*}
|W_n^{(r)}| &=&\frac{m_n.p_n}{\varepsilon_n}\Big|
{\Big(\frac{\varepsilon_n}{p_n}\Big)}^{\alpha}\sum_{j=1}^{r}\frac{\eta_j k_j^{\alpha}}{\alpha!}
+o\Big({\Big(\frac{\varepsilon_n}{p_n}\Big)}^{\alpha}\sum_{j=1}^{r}\frac{\eta_j k_j^{\alpha}}{\alpha!}\Big)\Big|\\
 &\geq& \Big|\frac{h_n.\varepsilon^{\alpha-1}_n}{p_n^{\alpha-1}}\sum_{j=1}^{r}\frac{\eta_j k_j^{\alpha}}{\alpha!}
+o\Big(\frac{h_n.\varepsilon^{\alpha-1}_n}{p_n^{\alpha-1}}\sum_{j=1}^{r}\frac{\eta_j k_j^{\alpha}}{\alpha!}\Big)\Big|
\tendn +\infty.
\end{eqnarray*}
To ensure that the assumption \eqref{factoriel} holds, take $p_n=n, \varepsilon_n =\frac1{n^2}$ and
$m_n=\frac{h_n}{n^2}$.
\end{xrem}
\begin{proof}[Proof of Proposition \ref{Gauss}]
Let $A$ be a Borel subset of $\R$, and $x \in ]1,+\infty[$, then, for any
positive integer $m$, we have

\begin{eqnarray*}{\label{eq :eqfin}}
\bigintss_A \Big|\big|P_m(\theta)\big|^2-1\Big| d\lambda_s(\theta) &\geq& \bigintss_{\{\theta
\in A
~~~:~~~~|P_m(\theta)|>x\}}\Big|\big|P_m(\theta)\big|^2-1\Big| d\lambda_s(\theta)\\
&\geq&(x^2-1)\lambda_s\Big\{\theta \in A~:~|P_m(\theta)|>x \Big\}\\
&\geq& (x^2-1) \lambda_s\Big\{\theta \in A ~~:~~|\Re({P_m(\theta)})|>x\Big\}
\end{eqnarray*}
\noindent{} Let $m$ goes to infinity and use Theorem
\ref{Gauss} and Proposition \ref{w-lim} to get

\[
\bigintss_A \phi\; d\lambda_s \geq (x-1)\{1-{\mathcal
{N}}([-\sqrt{2}x,\sqrt{2}x])\}\lambda_s(A).
\]

\noindent{}Put $K=(x-1)\{1-{\mathcal {N}}([-\sqrt{2}x,\sqrt{2}x])\}$. Hence
$$
\bigintss_A \phi\; d\lambda_s \geq K \lambda_s(A),$$
\noindent{}for any Borel subset $A$ of $\R$. This end the proof of the proposition.
\end{proof}

\noindent{}Now, we give the proof of our main result.\\

\begin{proof}[Proof of Theorem \ref{main-expo}] Follows easily from the proposition \ref{key-sasha} combined with
proposition \ref{CDsuffit}.
\end{proof}
\noindent{}According to our results we ask the following question.

\begin{que}Does any exponential rank one flow have singular spectrum?
\end{que}
\noindent We are not able to address the previous question. However, in the next section, we establish that there exists a rank one flow acting on infinite measure space with Lebesgue spectrum. We will further discuss some flatness problem related to the previous question raised by A. A. Prikhod'ko in \cite{prikhodko}.
\section{Infinite rank one flow with simple Lebesgue spectrum}\label{Banach-rank1}
The main subject of the following section is to establish the following theorem.
\begin{thm}\label{smain1}There exist a conservative ergodic measure preserving flow on $\sigma$-finite space $(X,\A,\mu)$ with simple Lebesgue spectrum.
\end{thm}
\noindent The proof of Theorem \ref{smain1} is based essentially on the following result proved in \cite{elabdal-Banach}.
\begin{lemm}\label{smain2}
	There exist a sequence of analytic trigonometric polynomials $\big(P_n\big)_{n \in \N}$ with coefficients $0$ and $1$ such that the polynomials $\frac{P_n(z)}{\|P_n\|_2}$  
	are flat in almost everywhere sense, that is,
	$$\frac{P_n(z)}{\|P_n\|_2} \tend{n}{+\infty}1,$$
	for almost all $z$ of modulus $1$ with respect to the Lebesgue measure $dz$.
\end{lemm}
\noindent We are now able to proceed to the proof of Theorem \ref{smain1}.
\begin{proof}[\textbf{Proof of Theorem \ref{smain1}}]We start by establishing the existence of a sequence of $L^1(d\lambda_s)$-flat trigonometric polynomials on $\R$ with coefficients $0$ and $1$, for any $s \in ]0,1].$  For that, let $ \big(\frac{P_n(z)}{\|P_n\|_2})_{n \in \N}$ be a sequence of flat polynomials given by Lemma \ref{smain2}, and for any $\theta \in \R$, put
$$Q_n(\theta)=\frac{P_n(e^{i\theta})}{\|P_n\|_2}.$$
As before, for any $s>0$, define
$$\widetilde{K}_s(\theta)=2\pi\sum_{n \in \Z}K_s(\theta+2n\pi), \forall \theta \in \R.$$
Then $\widetilde{K}_s(\theta)$ is $2\pi$-periodic. We further have
$$\frac{1}{2\pi}\int_{0}^{2\pi}\Big|\big|Q_n(\theta)\big|-1\Big|\widetilde{K}_s(\theta) d\theta=
\int_{\R}\Big|\big|Q_n(t)\big|-1\Big| dt.$$
But  $\widetilde{K}_s(\theta)$ is bounded. Therefore,
$$\int_{\R}\Big|\big|Q_n(t)\big|-1\Big| dt \tend{n}{+\infty}0,$$
since $Q_n(z)$ is $L^1(dz)$-flat. Now, applying Theorem \ref{th7}, we get that there exist a rank one flow $(T^t)$ acting on $\sigma$-finite measure space with simple Lebesgue spectrum. The proof of the theorem is complete.
\end{proof}
\section*{\appendixname. On the locally flat polynomials on the real line} 
The purpose of this appendix  is to investigate the problem of flatness of some class of polynomials on the real line. This class of  polynomials was introduced in \cite{prikhodko} in connection with the study of the spectrum of some class of flows in ergodic theory. Therein, the author claimed that those polynomials are $L^1$-locally flat. Unfortunately, as we will see by applying carefully Karatsuba-Korolev theorem \cite{KK}, this is not the case.\\

Let us point out that the reader is not required to be familiar with the spectral theory of dynamical systems and the flat polynomials business. But, the only thing need it from \cite{prikhodko} is the definition of the sequences of polynomials which we recall in the  next section. We stress that in the very recent  republished version of Prikhod'ko's paper in the same journal \cite{prikhodko20}, the author state without proof that those polynomials are $L^1$-locally flat. We stress that the notion of $L^1-\epsilon$-flat is more general than $L^1$-locally flat but according to the construction 2 in  \cite{prikhodko20}, it is easy to see that the sequence of polynomials $(P_n(t))$ which is $L^1-\epsilon_n$-flat  on $W_{a_n,b_n}=(-b_n,-a_n) \bigcup (a_n,b_n)$ with $W_{a_n,b_n} \nearrow \R\setminus\{0\}$ is $L^1$-locally flat. \\

\noindent Let us start by recalling that the sequence of polynomials $(P_n)$ on the real line is $L^1$-locally flat if, for any $a<b$, we have
$$\int_{a}^{b} \Big||P_n(t)|^2-1|\Big| dt \tend{n}{+\infty}0.$$
We further demand that for $s>0$, 
$$\int |P_n(t)|^2 K_s(t) dt =1,$$
where 
$$K_s(t)=\frac{s}{2\pi}\cdot{\left(\frac{\sin(\frac{st}2)}{\frac{s t}2}\right)^2}.$$
Obviously, we have
\begin{eqnarray*}
	\int_{a}^{b} \Big||P_n(t)|^2-1|\Big| dt  &=& \int_{a}^{b} \Big||P_n(t)|-1| ||P_n(t)|+1|\Big| dt\\
	&\geq& \int_{a}^{b} \Big||P_n(t)|-1|  dt
\end{eqnarray*}

\noindent{}The connection between the notion of $L^1$-locally flat and the notion of $L^1$-flat is formulated in the following proposition.

\begin{Prop}Let $(P_n(t))$ be a sequence of real trigonometric polynomials $L^2(K_s(t) dt)$-normalized for any $s>0$, and   
	$(P_n(t))$ is 
	$L^1$-Locally flat, then there is a subsequence $(P_{n_k}(t))$ which is $L^1(K_s(t)dt)$-flat for any $s>0$.
\end{Prop} 
\begin{proof} Assume that  $(P_n(t))$ is $L^1$-Locally flat. Then, for any $a<b$, 
	$$\int_{a}^{b} \Big||P_n(t)|-1|\Big| dt \tend{n}{+\infty}0.$$
	We can thus extract a subsequence $(P_{n_k}(t))$ such that 
	$$ \Big||P_{n_k}(t)|-1|\Big| \tendx{n}{+\infty}{a.e.}0.$$
	But, for any $s >0$, we have
	$$\int_{\R} \Big||P_{n_k}(t)|-1|^2\Big| K_s(t) dt\leq 4.$$
	Whence, the sequence  $(P_{n_k}(t))$ is uniformly integrable. We can thus apply the classical Vitaly convergence theorem, to conclude
	$$\int_{\R} \Big||P_{n_k}(t)|-1|\Big| K_s(t) dt \tend{k}{+\infty}0.$$
	Hence, $(P_{n_k}(t))$ is $L^1(K_s(t)dt)$-flat, for any $s>0$, which finish the proof of the proposition.    
\end{proof}

\section*{Exponential-staircase polynomials are not $L^1$-locally flat  }

The main issue of this appendix is related to the problem of $L^1$-locally flatness of the following sequence of polynomials
$$P_n(t)=\frac1{\sqrt{q_n}}\sum_{j=0}^{q_n-1}e^{2 i \pi \omega_n(j) t },$$
where $q_n$ is a sequence of positives integers such that $q_n \longrightarrow +\infty$, $\omega_n$ is given by
$$ \omega_n(x)=\frac{m_n}{\beta_n^2}q_n e^{\beta_n\frac{x}{q_n}},$$
$\beta_n \longrightarrow 0$.
We call such polynomials the exponential-staircase polynomials. This is due to the fact that this class of polynomials is related to the spectral type of the exponential-staircase rank one flows. In \cite{prikhodko}, the sequences $m_n$, $\beta_n$, $q_n$ verify
\begin{enumerate}[(i)]
	\item $\frac1{\beta_n} \in \N$ and we put $\frac{m_n}{\beta_n}=h_n \in  \N$.
	\item For some $\alpha \in (0,\frac14)$,~~$h_n^{\frac12+\alpha}<q_n$ and $ m_n \leq h_n^{\frac12-\alpha}$.
\end{enumerate}
Notice that the condition (i) and (ii) gives
\begin{enumerate}[(A)]
	\item ${m_n}^{\frac12+\alpha}< q_n {\beta_n}^{\frac12+\alpha}< {(q_n \beta_n)}^{\frac12+\alpha}.$ 
	\item From which we get $m_n <q_n \beta_n.$
	\item We further have $$\sqrt{m_n\beta_n} \leq \beta_n^\alpha <1.$$
\end{enumerate}
For a real number $x$ we will denote by $[x]$ the greatest integer less than or equal to $x$, and by $
\{x\}=x-[x]$ the fractional part. Let $||x||$ denote the distance of $x$ to the nearest integer, that is, $||x||=\min\Big\{\{x\},1-\{x\}\Big\}.$ If $f$ is a differentiable real function and $\rho$ is a positive number, we set
$$T_{f,x,\rho}=\left\{
\begin{array}{ll}
0, & \hbox{if $||f'(x)||=0$;}  \\
\min\{\sqrt{\rho},\frac1{||f'(x)||}\}, & \hbox{if not.}
\end{array}
\right.
$$
In \cite{prikhodko}, the author stated under condition (i) and (ii) that the following sequence of polynomials is $L^1$-locally-flat: 
$$Q_n(t)=\frac1{\sqrt{q_n}}\sum_{j=0}^{q_n-1}e^{2 i \pi \psi_n(j) t },$$
where $\ds \psi_n(j)=\frac{\omega_n(j)}{m_n}=\frac{q_n}{\beta_n^2}e^{\beta_n \frac{j}{q_n}}.$\\
There is many issues in the paper but the principal gap is in the proof of Lemma 6 and 7 from \cite{prikhodko}. Indeed, the author state that $\gamma(y_k)=e^{\frac{\beta_n}{2q_n}y_k} \sim 1$ with $t\psi'(y_k)=k$ (this is also hidden in the equation (3.27)). But, according to the definition of $y_k$ (equation (3.20) in \cite{prikhodko}) , we have
\begin{eqnarray}\label{Corrected}
e^{\frac{\beta_n}{q_n}y_k}=\beta_n.\frac{k}{t},
\textrm{~~with~~} k \in [\frac{t}{\beta_n},\frac{t}{\beta_n}e^{\beta_n}].
\end{eqnarray} 
Therefore, the computation need to be carried out carefully since the constant in the approximation in the equation (3.27) depend on $k$ and $t$. Moreover, the parameters depend tightly  on each other (see equation \eqref{PError} where $x_j$ play the same role as $y_k$.).\\

\noindent Let us state now the main result of this appendix.

\begin{thm}\label{main-expo-poly}For any positive numbers $\tau_1<\tau_2$, we have
	$$\limsup_{n \rightarrow +\infty}|Q_n(t)| \leq \sqrt{t},$$
uniformly for any $t \in [\tau_1,\tau_2]$.	
\end{thm}
\noindent The proof  of Theorem \ref{main-expo-poly} is a direct consequence of the so-called van der Corput method, for a nice account on this method and 
its generalization by Salem , we refer to \cite{Zygmund} and \cite{Salem}. For the classical results,  we refer to \cite{Tichmarsh} and \cite{GK}. \\


\noindent Here, we need the following theorem due to Karatsuba and Korolov \cite{KK}.
\begin{lemm}\label{KK-main}Let real function $f(x)$ satisfies the following conditions on
	a closed interval $[a,b]$:
	\begin{enumerate}[(a)]
		\item The fourth derivative $f^{(4)}(x)$ of $f$ is continuous.
		\item  There are non-negative numbers $U, A,\lambda,c_1 , c_2 , c_3 , c_4$ such
		that
		\begin{align}
		&U\geq 1,~~0<b-a \leq \lambda U,&\\
		&\frac{c_1}{A} \leq f''(x) \leq  \frac{c_2}{A},~~~~~~~|f^{(3)}(x)| \leq \frac{c_3}{AU},~~~~~~~~|f^{(4)}(x)| \leq \frac{c_4}{AU^2}.&
		\end{align}
	\end{enumerate}
	Then
	$$\sum_{a<x\leq b} e^{2 \pi i f(x)}=\sum_{f'(a) \leq j \leq f'(b)}c(j)Z(j)+E,$$
	with
	$$c(j)=\left\{
	\begin{array}{ll}
	1, & \hbox{if $f'(a)<j<f'(b)$;} \\
	0.5, & \hbox{if $j=f'(a)$ or $j=f'(b)$,}
	\end{array}
	\right.
	$$
	$$Z(j)=\frac{1+i}{\sqrt{2}}\frac1{\sqrt{f''(x_j)}}e^{2 \pi i (f(x_j)-jx_j)},$$
	where the numbers $x_j$ satisfy $f'(x_j)=j$. $E$ is given by
	$$E=\theta \Big(K_1\ln(f'(b) -f'(a) + 2) + K_2 + K_3 (T_{f,a,A} + T_{f,b,A} )\Big),~~~|\theta| \leq 1,$$
	and
	$$K_1=\frac1{\pi}\Big(6.5+2\frac{c_1}{c_2}\Big),$$
	$$K_3=2\Big(2+\frac1{\pi}\Big)+\frac1{\pi c_1}\Big(4+2.8\sqrt{c_1}+c_2+\frac{2c_2}{c_1}\Big),$$
	$$K_2=\frac1{\pi c_1^2}\Big(\Big(\lambda c_2+2\frac{A}{U}\Big)K+2c_2\Big(c_1+\frac{A}{b-a}\Big)\Big)+\Big(22.5+9\frac{c_2}{A}\Big),$$
	
	$$K=5c_3+\frac12\max\Big\{\frac{9}{8}c_4+\Big(\frac{13}{6}\Big)^2\frac1{c_1}(c_3+0.5kc_4)^2,2\frac{c_2}{k^2})\Big\},$$
	$$k=\min\Big\{\frac{c_1}{4c_2},\sqrt{\frac{c_1}{2c_2}}\Big\}.$$
\end{lemm}

\noindent{}\begin{proof}[\bf {Proof of Theorem \ref{main-expo-poly}}]We will apply Theorem \ref{KK-main} to estimate
	$$\frac{1}{\sqrt{q_n}}\sum_{j=0}^{q_n}e^{2 i \pi \psi_n(j) t },$$
	with $\beta_n<<1$. In this case $a=0$, $b=q_n$ and $f(x)=t.\psi_n(x)
	=t\frac{q_n}{\beta_n^2}e^{\frac{\beta_nx}{q_n}}.$\\ 
	
	\noindent{}We start by computing $U, A,\lambda,c_1 , c_2 , c_3,$ and $c_4$. We check at once that
	$$f'(x)=t.\frac{1}{\beta_n}e^{\frac{\beta_nx}{q_n}}, f'(0)=t.\frac{1}{\beta_n} ,
	f'(q_n)=t.\frac{1}{\beta_n} e^{\beta_n}.$$
	$$f''(x)=t.\frac{1}{q_n}e^{\frac{\beta_nx}{q_n}}, f^{(3)}(x)=t.\frac{ \beta_n}{q_n^2} e^{\frac{\beta_nx}{q_n}},$$
	\noindent{}	and
	$$f^{(4)}(x)=t.\frac{\beta_n^2}{q_n^3} e^{\frac{\beta_nx}{q_n}}.$$
	\noindent{}	We further have
	$$\tau_1.\frac{1}{q_n} \leq f''(x) \leq e.\tau_2\frac{1}{q_n}.$$
	\noindent{}From this, we can set
	$$c_1=\tau_1,c_2=c_3=c_4=e.\tau_2, A=q_n.$$
	\noindent{}	Of course we take $\lambda=1$ and $U=q_n$, and it is a simple matter to see that
	$$|f^{(3)}(x)|=f^{(3)}(x) \leq  \frac{c_3}{AU}=e\tau_2 \frac{1}{q_n}.\frac1{q_n},$$
	$$|f^{(4)}(x)|=f^{(4)}(x) \leq \frac{c_4}{AU^2}=e\tau_2 \frac{1}{q_n}.\frac1{q_n^2},$$
	\noindent{}	and
	$$\frac{f'(x)}{f(x)}=\frac{f''(x)}{f'(x)}=\frac{f^{(3)}(x)}{f''(x)}=\cdots=\frac{\beta_n}{q_n},$$
	\noindent{}From this, we deduce that
	\begin{align}\label{PError}
		f''(x_j)=\frac{\beta_n}{q_n} j, \textrm{~~~and~~} f(x_j)=\frac{q_n}{\beta_n}j \in \Z, 
		~~~\forall j \in [t\frac1{\beta_n},t\frac1{\beta_n} e^{\beta_n}],
	\end{align}
	\noindent since 
	\begin{align}\label{PError2}
	f'(x_j)=j,~~~~~~~~ \forall j \in [t\frac1{\beta_n},t\frac1{\beta_n} e^{\beta_n}].
	\end{align}
	 We further have
	$$x_j=\frac{q_n}{\beta_n}\log\big(\frac{t}{\beta_n}\big)+\frac{q_n}{\beta_n}\log(j)=C_n(t)+D_n\log(j),$$
	where $D_n=\frac{q_n}{\beta_n}\in \N$ and  $C_n(t)=D_n\log\big(\frac{\beta_n}{t}\big) .$\\
	
	\noindent{}Remembering that Lemma \ref{KK-main} gives us an estimate of $|P_n(t)|$ up to the error term $E_n=\frac{E}{\sqrt{q_n}}$. To achieve this goal, we need to estimate $E$. For that, notice that
	$$K_1=\frac{1}{\pi}\Big(6.5+2\frac{\tau_1}{e.\tau_2}\Big),~~~
	K_3=2\Big(2+\frac1{\pi}\Big)+\frac1{\pi
		\tau_1}\Big(4+2.8\sqrt{\tau_1}+e.\tau_2+2\frac{e.\tau_2}{\tau_1}\Big),$$
	\noindent{}	and since $\tau_1 <\tau_2$, we get
	$$k=\frac{\tau_1}{4 e.\tau_2}.$$
	\noindent{}	We further have
	$$\max\Big\{\frac98e.\tau_2+\Big(\frac{13}6\Big)^2 \frac1{\tau_1}\Big(e.\tau_2+\frac{\tau_1}{8}\Big)^2,32\frac{e^3.\tau_2^3}{\tau_1^2}\Big\}=
	32\frac{e^3.\tau_2^3}{\tau_1^2}.$$
	\noindent{}	Hence
	$$K=5e.\tau_2+16\frac{e^3.\tau_2^3}{\tau_1^2}.$$
	\noindent{}	We thus get
	$$K_2=\frac1{\pi \tau_1^2}\Big(\Big(e.\tau_2+2\Big)K+2e.\tau_2\Big(\tau_1+1\Big)+\Big(22.5+9\frac{e.\tau_2}{q_n}\Big)\Big).$$
	\noindent{}	Finally, we note that there exist a constant $\gamma(t)$ which depend only on $t$ such that for an infinitely many $n$, we have 
	$$\max\Big\{T_{f,a,A},T_{f,b,A}\Big\}=\max\Big\{T_{f,0,q_n},T_{f,q_n,q_n}\Big\} \leq \gamma(t).$$
	\noindent{}	Indeed, we have
	\begin{align}\label{error-3}
		T_{f,0,q_n} \leq \frac{1}{\|\beta_n^{-1}t \|}, \, \,\textrm{and} \,\, T_{f,q_n,q_n} \leq \frac{1}{\|e^{\beta_n}\beta_n^{-1}t \|}.
	\end{align}
	\noindent{}	It suffice to establish that 
	\begin{align}
	 \frac{1}{\|\beta_n^{-1}t \|}\leq \gamma(t), \, \,\textrm{and} \,\, \frac{1}{\|e^{\beta_n}\beta_n^{-1}t \|}\leq \gamma(t).
	\end{align}
	\noindent{}	For that, let $t=\sum_{k=0}^{+\infty}b_{k}(t) \beta_n^k$ be the $\beta_n^{-1}$-expansion of $t$. Then, 
	$\beta_n^{-1}t=b_0\beta_n^{-1}+\sum_{k=1}^{+\infty}b_k(t)\beta_n^{k-1}$. Whence
	\begin{align}\label{lower-bound}
		\|\beta_n^{-1}t \| &=\Big\| \sum_{k=2}^{+\infty}b_k(t)\beta_n^{k-1}\Big\|,\nonumber\\
		&=\|t-\beta_nb_1(t) \| \\
		&\geq |\|t\|-\|\beta_nb_1(t)\|| \nonumber
	\end{align}
 	\noindent{}But $ 0 \leq \beta_nb_1(t) \leq 1-\beta_n$ and $\beta_n \tend{n}{+\infty}0$. Hence
   $$0 \leq \liminf(\beta_nb_1(t))\setdef g(t) \leq \limsup(\beta_nb_1(t)) \leq 1.$$
	\noindent{}Combining this with \eqref{lower-bound}, we get 
   $$\limsup_{n \rightarrow +\infty} \|\beta_n^{-1}t \| \geq \Big|\|t\|-\|g(t)\|\Big|,$$
	\noindent{}   by the continuity of the map $ d~~:~~x \in \R \mapsto \|x\|=\ds \min_{n \in \Z}|x-n|.$ To finish the proof of \eqref{error-3}, we notice, again by the continuity of the map $d$, that 
   $$\|e^{\beta_n}\frac{t}{\beta_n}-\frac{t}{\beta_n}\| \tend{n}{+\infty} \|t\|.$$
 	\noindent{}Therefore, 
   \begin{align}
   	\limsup_{n \rightarrow +\infty} \|e^{\beta_n}\frac{t}{\beta_n}\|
   	&\geq \limsup_{n \rightarrow +\infty} \Big|\|\frac{t}{\beta_n}\|-\|e^{\beta_n}\frac{t}{\beta_n}-\frac{t}{\beta_n}\|\Big|\\
   	&\geq \big|2\|t\|-\|g(t)\|\big|\setdef \gamma(t).
   \end{align}
  
	\noindent{}We are now able to estimate $|P_n(t)|$. Write
	\begin{eqnarray*}
		Q_n(t)&=&\frac{1}{\sqrt{q_n}}\sum_{t\beta_n^{-1} \leq j \leq t\beta_n^{-1}.e^{\beta_n}}Z(j)+E_n \\
		&=&\frac{1+i}{\sqrt{2}}\sum_{t\beta_n^{-1} \leq j \leq t\beta_n^{-1}.e^{\beta_n}} \frac{1}{\sqrt{\beta_n j}}\exp(-j x_j)+E_n\\
		&\lesssim&\sqrt{t}+E_n
	\end{eqnarray*}
	
	\noindent{}	Since
	\begin{eqnarray*}
		\Big|\frac{1}{\sqrt{q_n}}\sum_{t\beta_n^{-1} \leq j \leq t\beta_n^{-1}.e^{\beta_n}}Z(j)\Big|&\leq& \frac{1}{\sqrt{q_n}}\sum_{t \beta_n^{-1}\leq j \leq t\beta_n^{-1}.e^{\beta_n}}|Z(j)|\\
		&=&\sum_{t \beta_n^{-1}\leq j \leq t\beta_n^{-1}.e^{\beta_n}} \frac{1}{\sqrt{\beta_n j}}, 
	\end{eqnarray*}
	\noindent{}	and the last term can be estimated by the integral test as follows
	\begin{eqnarray*}
		\frac1{\sqrt{\beta_n}}\sum_{t \beta_n^{-1}\leq j \leq t\beta_n^{-1}.e^{\beta_n}} \frac{1}{\sqrt{j}}
		&\sim& \frac1{\sqrt{\beta_n}} \int_{t \beta_n^{-1}}^{t \beta_n^{-1}e^{\beta_n}}\frac{dx}{\sqrt{x}}\\
		&=& \frac1{\sqrt{\beta_n}}2\sqrt{t \beta_n^{-1}}\Big(e^{\frac{\beta_n}{2}}-1\Big).
	\end{eqnarray*}
	\noindent{}	Therefore, for a sufficiently large  $n$ we can write
	$$
	\frac1{\sqrt{\beta_n}}\sum_{t \beta_n^{-1}\leq j \leq t\beta_n^{-1}.e^{\beta_n}} \frac{1}{\sqrt{j}} \sim
	\frac1{\sqrt{\beta_n}}2\sqrt{t \beta_n^{-1}}\Big(e^{\frac{\beta_n}{2}}-1\Big)
	\sim \sqrt{t}.$$
	\noindent{}Letting $n$ goes to infinity, we conclude that 
	$$\limsup_{n \rightarrow +\infty}|Q_n(t)| \leq \sqrt{t}.$$
	This complete the proof of the theorem.
\end{proof}
\begin{rem}Our result can be deduced also by using Lemma 5 from \cite{prikhodko}. But, one needs to replace carefully the exact value of $\gamma(y_k)$ given by equation \eqref{Corrected}. We further notice that therein, the error term is given as follows (see Lemma 6 equation (3.34)) 
	$$\int_{t_1}^{t_2} \Big|\frac{\mathcal{E}_n(t)}{\sqrt{q_n}}\Big| dt \lesssim (t_2-t_1)\frac{\ln(q_n)}{\sqrt{q_n}}.$$  
\end{rem}
	\noindent{}We end this appendix by mentioning that it is a simple matter to deduce from Theorem \ref{main-expo-poly} that $(Q_n(t))$ are not $L^1$-locally flat. Indeed, assume by contradiction that $(Q_n(t))$ is  $L^1$-locally flat and let $1>\tau_2 > \tau_1>\frac12.$  Then, we can extract a subsequence  $(n_k)$ such that $(Q_{n_k}(t)$ converge for a.e. $t \in [\tau_1,\tau_2]$. Therefore, by Vitali convergence theorem combined with Theorem \ref{main-expo-poly}, we have 
$$   \lim \int_{\tau_1}^{\tau_2} |Q_{n_k}(t)| dt = \int_{\tau_1}^{\tau_2} \lim |Q_{n_k}(t)| dt \leq \frac{2}{3}(\tau_2^{\frac32}-\tau_1^{\frac32}),$$
that is, 
\begin{eqnarray*}
	\lim \int_{\tau_1}^{\tau_2} |Q_{n_k}(t)| dt &=&\tau_2-\tau_1 \\
&\leq& \frac{2}{3}(\tau_2^{\frac32}-\tau_1^{\frac32}). 
\end{eqnarray*}
	\noindent{}Applying the classical Rolle's theorem, we get that there is $\theta \in ]\tau_1,\tau_2[$ such that
$1 \leq \sqrt{\theta},$ which is impossible.

\begin{thank}
The author wishes to express his thanks to Mahendra \linebreak Nadkarni and Giovanni Forni for many stimulating e-conversations 
on the subject. He would like also to thank  Corinna  Ulcigrai, Jean-Paul Thouvenot and Jean-Fran\c{c}ois Parreau  for the fruitful discussion on the computation of the spectral type of rank one flows. He would like further to thank her and the university of Bristol for the warm hospitality during a visit in 2010 where this work started. The author would also like to thank Michael Lin for bringing \cite{prikhodko20} to his attention during the final preparation of the second version of this manuscript.
\end{thank}


\begin{thebibliography}{}
%
%
\bibitem{elabdal-Banach}
e.~ H. ~el Abdalaoui, {\em Ergodic Banach problem, flat polynomials and Mahler's measures with combinatorics,},  arXiv:1508.06439v5 [math.DS].

\bibitem{Abd-Nad2}
E. H. ~el Abdalaoui and M. Nadkarni {\em  A non-singular transformation whose 
	Spectrum has Lebesgue component of multiplicity one ,} Ergodic Theory and Dynamical Systems 36 (2016), no. 3, 671-681. 

\bibitem{Abd-Nad3}
e.~  H. ~el Abdalaoui and M. Nadkarni {\em Notes on the flats polynomials,}
arXiv:1402.5457 [math.CV]

\bibitem{elabdal-Nad1}
e.~ H. ~el Abdalaoui and M. Nadkarni, \emph{Calculus of Generalized Riesz Products,}
Recent trends in ergodic theory and dynamical systems, 145-180,
Contemp. Math., 631, Amer. Math. Soc., Providence, RI, 2015. 


\bibitem{elabdal-Mahler}
E. H. ~el Abdalaoui, {\em The Mahler measure of the spectral type of a rank one map with cutting parameter $O(j^\beta)$, $\beta \leq 1$, is zero,}
preprint 2013.
 
\bibitem{elabdal-inde}
e.~ H. ~el Abdalaoui, {\em On the spectrum of the powers of Ornstein transformations,} 
Ergodic theory and harmonic analysis (Mumbai, 1999). Sankhy\={a} Ser. A 62 (2000), no. 3, 291-306.

\bibitem{elabdalihp}
e.~ H. ~el Abdalaoui,F.~ Parreau, and A.~ A.~ Prikhod'ko, 
{\em A new class of Ornstein transformations with singular spectrum,} 
Ann. Inst. H. Poincar\'e Probab. Statist. 42 (2006), no. 6, 671-681. 

\bibitem{elabdaletds}
e.~ H. ~el Abdalaoui, {\em A new class of rank-one transformations with singular spectrum,} 
Ergodic Theory Dynam. Systems 27 (2007), no. 5, 1541-1555. 

\bibitem{elabdalisr}
E.H. El Abdalaoui,La singularit\'e mutuelle presque sure du spectre des transformationsd’Ornstein, Israel J. Math.112(1999), 135-156

\bibitem{elabdal}
e.\, H.~ \ {el} Abdalaoui,
{\em }, Th\'ese d'habilitation, Rouen, 2008.



\bibitem {Adams1}
T. ~R. Adams, {\em On Smorodinsky conjecture, }  Proc. Amer. Math.
Soc. 126 (1998), 739--744.

\bibitem {Adams2}
T. ~R. Adams \& N.~A. Friedman, {\em  Mixing staircase, }
Preprint, 1992.

\bibitem{Ag}
O.~ Ageev, {Dynamical System With an Even-Multiplicity Lebesgue Component in the Spectrum},
Math. USSR, 64, 1987,  305.

\bibitem{Berkes}
I. ~Berkes, {\em On the central limit theorem for lacunary
trigonometric series,}  Anal. Math.  4  (1978), no. 3, 159-180.

\bibitem {Bourgain}
J. ~Bourgain, {\em On the spectral type of Ornstein class one
transformations, }{\it Isr. J. Math. },{\bf 84
}(1993), 53-63.


\bibitem{Brown-Hewitt}
G. ~Brown, and E. ~ Hewitt,
{\em Some new singular Fourier-Stieltjes series,} Proc. Nat. Acad. Sci. U.S.A. 75 (1978), no. 11, 5268-5269.

\bibitem{lebonchacon}
R.~V.~Chacon, {\em Approximation of transformations with continuous spectrum,} Pacific J. Math. 31 1969 293-302.

\bibitem {Nadkarni1}
J. ~R. ~Choksi ~and ~M. ~G. ~Nadkarni , {\em The maximal spectral
type of rank one transformation, } {\it Can. Math. Bull., } {\bf 37
(1)} (1994), 29-36.

\bibitem{Coquet-Kamae-Mendes-France}
J.~ Coquet, T.~ Kamae and M.~ Mend\`es France,
{\em Sur la mesure spectrale de certaines suites arithm\'etiques,} (French) Bull. Soc. Math. France 105 (1977), no. 4, 369-384.
\bibitem{CFS}
I.~P.~Cornfeld, S.~ V.~ Fomin and Ya.~G.~ Sina\v{i}, {\em
Ergodic theory,} Translated from the Russian by A. B. Sosinski\"{a}­.
Grundlehren der Mathematischen Wissenschaften
[Fundamental Principles of Mathematical Sciences], 245. Springer-Verlag, New York, 1982.

\bibitem{Danilenko-Silva}
A.~ I.~ Danilenko and C.~ E. Silva, {\em
Mixing rank-one actions of locally compact abelian groups,}
Ann. Inst. H. Poincar\'e Probab. Statist. 43 (2007), no. 4, 375-398.

\bibitem{deljunco98} A.~ del Junco, {\em A simple map with no prime factors,} Israel J. Math. 104 (1998), 301-320.


\bibitem{Deljunco-Park}
A.~del Junco and K.~Park,
{\em An example of a measure-preserving flow with minimal self-joinings,} J. Analyse Math. 42 (1982/83), 199-209.

\bibitem{degot}
J. ~D\'egot, Finite-dimensional Mahler measure of a polynomial and Szeg\"{o}'s theorem, J. Number Theory 62, No.2, (1997),
422-427.

\bibitem{Durrett} R.~Durrett, \emph{Probability : Theory and Examples}, Second Edition, Duxbury Press, Belmont, CA, 1996.

\bibitem{Erdos}
P.~Erd\"os, {\em On trigonometric sums with gaps.} Magyar Tud.
Akad. Mat. Kutat\'o Int. K\"ozl  7  1962 37--42.

\bibitem{ErdosII}
P. ~Erd\"{o}s, Some unsolved problems, Michigan Math. J., (4) 1957, 291-300.



\bibitem{Fukuyama}
  K.~Fukuyama and S.~Takahashi, {\em The central limit theorem for lacunary series.}
Proc. Amer. Math. Soc.  127  (1999),  no. 2, 599--608.

\bibitem{Friedman}
N.~A. ~Friedman, {\em Replication and stacking in ergodic theory. }
Amer. Math. Monthly {\bf 99} (1992), 31-34.




\bibitem{GK}
S.~ W.~ Graham and G.~ Kolesnik, Van der Corput's method of exponential sums,
London Mathematical Society Lecture Note  Series , 126 . Cambridge University   Press, Cambridge,1991.

\bibitem{Grenard}
V. ~Grenander and J~.Rosenblatt, \emph{Statistical analysis of stationary time series,} New York, 1957.

\bibitem{Guenais}
M. ~Guenais, {\em Morse cocycles and simple Lebesgue spectrum } Ergodic Theory Dynam. Systems,
\textbf{19} (1999), no. 2, 437-446.

\bibitem{HS}
H. ~Helson and G.~ Szeg\"{o}, \emph{A problem in prediction theory,} Acta Mat. Pura Appl., \textbf{51} (1960), 107-138.


\bibitem{Hewitt-Zuckermann}
E. ~Hewitt and H.~ Zuckerman, {\em Singular measures with absolutely continuous convolution squares,} 
Proc. Cambridge Philos. Soc. 62 1966 399-420. Corrigendum, ibid., 63 (1967), 367-368.


\bibitem{Hoffman}
K.~Hoffman, \emph{Banach spaces of analytic functions,} Reprint of the 1962 original. Dover Publications, Inc., New York, 1988.

\bibitem{Host-mela-parreau}
B.~ Host, J-F. M\'ela and F.~ Parreau, {\em 
	Nonsingular transformations and spectral analysis of measures,} Bull. Soc. Math. France 119 (1991), no. 1, 33-90.

\bibitem{Kahane}
J.~-P. Kahane, {\em Lacunary Taylor and Fourier series.} Bull.
Amer. Math. Soc.  70  1964 199-213.

\bibitem{Kamae}
T.~ Kamae, Spectral properties of automaton-generating sequences, unpublished.

\bibitem{KK}
A. ~A. ~ Karatsuba and M.~ A. ~ Korolev, A theorem on the approximation
of a trigonometric sum by a shorter one, 2007 Izv. Math. 71:2, 341-370.

\bibitem{Handbook-1B-11}
A.~ Katok and J.-P.~ Thouvenot, {\em Spectral Properties and Combinatorial Constructions in Ergodic Theory},
Chapter 11, Volume 1B, Handbook of Dynamical Systems (B. Hasselblatt and A. Katok, editors). Elsevier, North-Holland, 2006.

\bibitem{Kat-Step} A.  ~Katok \& A. ~Stepin,
{\em Approximations in ergodic theory.} (Russian)
Uspehi Mat. Nauk 22 1967 no. 5 (137), 81-106.

\bibitem{Katznelson}
Y.~ Katznelson, { \em An introduction to harmonic analysis,} Third edition.
Cambridge Mathematical Library. Cambridge University Press, Cambridge, 2004.

\bibitem{King}
J.~ King, The commutant is the weak closure of the powers, for rank-1 transformations. Ergodic Theory Dynam. Systems 6 (1986), no. 3, 363-384.
\bibitem{KS}
S.~J.~ Kilmer  and S.~ Saeki , On Riesz Product measures : mutual
absolute continuity and singularity, Ann. Inst. Fourier, t. 38 2, 1988,
p. 63-69.

\bibitem{Kiri}
A.  ~A. ~Kirillov, \emph{Dynamical systems, factors and group representations,} (Russian) Uspehi Mat. Nauk 22 1967 no. 5 (137), 67–80.

\bibitem{Klemes1}
I. ~Klemes, {\em The spectral type of staircase transformations,
}{\it Thohoku Math. J., }{\bf 48} (1994), 247-258.

\bibitem {Klemes2}
I. ~Klemes \& ~K. ~Reinhold, {\em Rank one transformations with
singular spectre type, } {\it Isr. J. Math., }{\bf 98} (1997),
1-14.
\bibitem{Littlewood2}
J.~ E. ~ Littlewood,  {\em  Some Problems in Real and Complex Analysis,} Heath Mathematical
Monographs. D.C. Heath and Company, Massachusetts, 1968.


\bibitem{MN}
J. ~Mathew and M.~G.~ Nadkarni,\emph{A measure-preserving transformation whose spectrum
	has Lebesgue component of multiplicity two,} Bull. London Math. Soc 16 (1984), 402-406.

\bibitem {mcleish}
D.~L.~McLeish, {\em Dependent central limit theorems and
invariance principles.} Ann. Probability  2  (1974), 620-628.

\bibitem{Murai}
T.~Murai, {\em The central limit theorem for trigonometric
series.} Nagoya Math. J.  87  (1982), 79-94.

\bibitem{Nadkarni4}
M.~G. ~Nadkarni, {\it Spectral theory of dynamical systems},
Birkh\"auser, Cambridge, MA, 1998. 

\bibitem{N-K}
L. ~Kuipers \& H.~ Niederreiter,  \emph{Uniform Distribution of Sequences,} Dover Publishing, 2006.

\bibitem {Ornstein}
D.~S.~Ornstein, {\em On the root problem in ergodic theory, } {\it
Proc. Sixth Berkeley Symposium in Math. Statistics and
Probability, } University of California Press, 1971, 347-356.

\bibitem{petit}
B.~Petit, {\em Le th\'eor\`eme limite central pour des sommes de
Riesz-Ra\u\i kov.} Probab. Theory Related Fields  93  (1992),  no.
4, 407--438.

\bibitem{Peyriere}
J. ~Peyri\`ere, {\it \'Etude de quelques propri\'et\'es
des produits de Riesz}, Ann.\ Inst.\ Fourier,
Grenoble \textbf{25}, 2 (1975), 127--169.

\bibitem{prikhodko20}
A. A. Prikhod'ko, On ergodic flows with simple Lebesgue spectrum,
Mat. Sb., 211:4 (2020), 123-144;  translation in Sb. Math., 211:4 (2020), 594-615.

\bibitem{prikhodko}
 A.~ A.~ Prikhod'ko, Littlewood polynomials and their applications to the spectral theory of dynamical systems. (Russian) Mat. Sb. 204 (2013), no. 6, 135--160; translation in Sb. Math., 204 (2013), no. 5-6, 910-935.

\bibitem{prikhodko-orn}
A.~A.~Prikhod'ko, {\em Stochastic constructions of flows of rank 1,}
Mat. Sb. 192 (2001), no. 12, 61-92; translation in Sb. Math. 192 (2001), no. 11-12, 1790-1828.

\bibitem{Queffelec1}
M. ~Queff\'elec, {\em Substitution dynamical systems spectral analysis,}
Second edition. Lecture Notes in Mathematics, 1294. Springer-Verlag, Berlin, 2010.

\bibitem{Q}
M. ~Queff\'elec,\emph{Une nouvelle propri\'et\'e des suites de Rudin-Shapiro,} Ann. Inst. Fourier 37
(1987), 115–138.

\bibitem{Rohlin}
V. ~A. ~ Rokhlin, Selected topics from the metric theory of dynamical systems, (Russian)
Uspehi Matem. Nauk (N.S.) 4, (1949). no. 2(30), 57-128, Amer. Math. Soc. Trans. (2) 49,
1966. 

\bibitem{RyzhikovWCT}
V.~ Ryzhikov, {\em
Mixing, rank and minimal self-joining of actions with invariant measure,} Mat. Sb. {\bf 183} (1992),
no. 3, 133-160.
\bibitem{wiener-wintner}
N.~Wiener and A.~ Wintner, {\em
On the ergodic dynamics of almost periodic systems,} Amer. J. Math. 63, (1941). 794-824.

\bibitem{Salem}
R.~ Salem, \emph{ Essais sur les s\'eries trigonom\'etriques}, Hermann \& Cie, \'editeurs, 1940.
\bibitem{Sierpinski}
W.  ~Sierpi\'{n}ski, Elementary theory of numbers. Transl. from the Polish. Edited by A. Schinzel. 2. ed.
\bibitem{Simon}
B.~ Simon, Szeg\H{o}'s Theorem and Its Descendants: Spectral Theory for L 2 Perturbations of Orthogonal
Polynomials, M. B. Porter Lectures, Princeton University Press, Princeton, NJ, 2011. 
\bibitem{Takahashi}
S.~Takahashi, {\em Probability limit theorems for trigonometric
series.} Limit theorems of probability theory (Colloq., Keszthely,
1974), pp. 381--397. Colloq. Math. Soc. Janos Bolyai, Vol. 11, North-Holland,
  Amsterdam, 1975.
\bibitem{Tichmarsh}
E. ~C.~ Titchmarsh, \emph{The theory of the Riemann zeta-function,} Second edition. Edited and with a preface by D. R. Heath-Brown. The Clarendon Press, Oxford University Press, New York, 1986.
\bibitem{Ulam}
S.~M.~ Ulam, {\em Problems in modern mathematics,} Science Editions John Wiley \& Sons, Inc., New York 1964.
\bibitem{Waldschmidt}
M.~Waldschmidt, {\em
Elliptic functions and transcendence, Surveys in number theory,}  Dev. Math., 17, Springer, New York, 2008.
\bibitem{zeitz} P.~ Zeitz, {\em The centralizer of a rank-one flow,} Israel J. Math. 84 (1993), no. 1-2, 129-145.
\bibitem {Zygmund}
A. ~Zygmund, {\it Trigonometric series {\rm vol. II}}, second ed.,
Cambridge Univ. Press, Cambridge, 1959.
\end{thebibliography}
\end{document}